\pdfoutput=1
\documentclass{elsarticle}

\usepackage[usenames,dvipsnames]{xcolor}  

\newcommand{\newtext}[1]{\textcolor{MidnightBlue}{#1}}
\renewcommand{\newtext}[1]{#1}                               

\usepackage[top=1in, bottom=1.25in, left=1.25in, right=1.25in]{geometry}

\usepackage{amsmath}
\usepackage{amsfonts}
\usepackage{amsthm}

\usepackage{relsize}

\usepackage{graphicx}
\graphicspath{{fig-rev1/}}

\usepackage{enumerate}

\usepackage{hyperref}       

\newcommand{\Alfven}{Alfv\'{e}n }
\newcommand{\En}{\mathcal{E}}
\newcommand{\Bvec}{\mathbf{B}}
\newcommand{\Avec}{\mathbf{A}}
\newcommand{\uvec}{\mathbf{u}}
\newcommand{\divg}{\nabla \cdot}
\newcommand{\curl}{\nabla \times}
\newcommand{\bigO}{\mathcal{O}}
\newcommand{\R}{\mathbb{R}}
\newcommand{\weirdI}{\mathcal{I}}
\newcommand{\weirdJ}{\mathcal{J}}
\newcommand{\StepTheta}{{\bf{Step-$\theta$}}}
\newcommand{\PhiWENO}{\Phi_{\text{WENO5}}}

\theoremstyle{plain}
\newtheorem{claim}{Claim}
\newtheorem{problem}{Problem}

\theoremstyle{remark}
\newtheorem{remark}{Remark}

\title{A high-order positivity-preserving single-stage single-step method for the ideal magnetohydrodynamic equations}

\begin{document}
\begin{frontmatter}

\author[andrew]{Andrew J. Christlieb}
\ead{christli@msu.edu}

\author[msu-math, ca]{Xiao Feng}
\ead{fengxia2@msu.edu}

\author[usna-math]{David C. Seal}
\ead{seal@usna.edu}

\author[rpi-math]{Qi Tang}
\ead{tangq3@rpi.edu}

\address[andrew]{Department of Computational Mathematics, Science, and Engineering, Department of Mathematics, and Department of Electrical and Computer Engineering,
	Michigan State University, East Lansing, MI 48824 USA}

\address[msu-math]{Department of Mathematics, Michigan State University, East Lansing, MI 48824, USA}

\address[usna-math]{Department of Mathematics, 
    U.S. Naval Academy,
    121 Blake Road,
    Annapolis, MD 21402, USA}

\address[rpi-math]{Department of Mathematical Sciences,
    Rensselaer Polytechnic Institute,
    Troy, NY 12180, USA}

\fntext[ca]{Corresponding author}

\begin{abstract}

We propose a high-order finite difference weighted ENO (WENO) method for the
ideal magnetohydrodynamics (MHD) equations.
The proposed method is single-stage (\newtext{i.e.,} 
it has no internal stages to store),
single-step (\newtext{i.e.,} 
it has no time history that needs to be stored), maintains a discrete
divergence-free condition on the magnetic field, and
has the capacity to preserve the positivity of the density and pressure.  To accomplish this, 
we use a Taylor discretization of the Picard integral formulation (PIF) of the finite difference 
WENO method proposed in [{\it SINUM}, {53} (2015), pp.\ 1833--1856], where the focus is on a high-order
discretization of the fluxes (as opposed to the conserved variables).
\newtext{We use the version where fluxes are expanded to third-order accuracy in time, and 
for the fluid variables space is discretized using the classical fifth-order finite difference WENO discretization.}
We use constrained transport in order to obtain divergence-free magnetic fields, which means
that we simultaneously evolve the magnetohydrodynamic (that has an evolution equation for the magnetic field) 
and magnetic potential equations alongside each other, and set the magnetic field to be the (discrete) curl of the magnetic potential
after each time step.  \newtext{In this work, we compute these derivatives to
fourth-order accuracy.}
In order to retain a single-stage, single-step method,
we develop a novel Lax-Wendroff discretization for the evolution of the magnetic potential,
where we start with technology used for 
Hamilton-Jacobi equations in order to construct a non-oscillatory magnetic field.
\newtext{The end result is an algorithm that is similar to our previous work
[{\it JCP}, {268}, (2014), pp.\ 302--325],
but this time the time stepping is replaced through a Taylor method with the addition of a positivity-preserving limiter.}
\newtext{Finally,} positivity preservation is realized by 
introducing a parameterized flux limiter that considers a linear combination of
high and low-order numerical fluxes.  The choice of the free parameter is then given 
in such a way that the fluxes are limited towards the low-order solver until
positivity is attained.  \newtext{Given the lack of additional degrees of freedom in the system, this positivity limiter 
lacks energy conservation where the limiter turns on.  However, this ingredient can be dropped for problems where the pressure
does not become negative.}
We present two and three dimensional numerical results for several standard test problems including
a smooth Alfv\'{e}n wave (to verify formal order of accuracy), 
shock tube problems (to test the shock-capturing ability of the scheme), \newtext{Orszag-Tang,} 
and cloud shock interactions. 
These results assert the robustness and verify the high-order of accuracy of the proposed scheme.

\end{abstract}


\end{frontmatter}

\section{Introduction}

The ideal magnetohydrodynamic (MHD) equations model the dynamics of a
quasi-neutral, perfectly conducting plasma~\cite{Schnack2009}.  A vast range
of areas including astrophysics and laboratory plasmas can be modeled with
this system.  \newtext{Among other methods, various high-order numerical schemes based on the
essentially non-oscillatory (ENO) \cite{LondrilloDelZanna00} as well as the 
weighted ENO (WENO) reconstruction technique \cite{Jiang1999,li2003modern,balsara2009divergence,shen2012cusp,Balsara2013,Christlieb2014,ChenTothGomb2016}
have been applied successfully to ideal MHD in the past two
decades.}
These \newtext{high order} schemes are capable
of resolving complex features such as shocks and turbulences using fewer grid
points than low-order schemes \newtext{for the same level of error}, as is common with many high-order
shock-capturing schemes.

It often happens in large-scale MHD simulations that the complex features are concentrated in a small portion of the simulation domain.  
Adaptive mesh refinement (AMR) is a technique that is designed for treating such locality of complexity in hydrodynamics 
and magnetohydrodynamics.  
\newtext{One of the chief difficulties with implementing high order schemes within an AMR framework is that boundary conditions for the refined region need to be 
specified in a consistent manner \cite{MccCol2011}.  This becomes difficult for multistage RK methods, because high order solutions cannot be found if one 
simply uses high-order interpolated values (in time) at the ghost points that are required for the intermediate stages of the method.}
Preliminary work that combines WENO spatial discretizations with 
strong stability preserving Runge-Kutta (SSP-RK) time-stepping is conducted in~\cite{Shen2011,Wang2015}, \newtext{and very recent work
makes use of curvilinear grids to extend finite difference methods to problems with geometry \cite{ChenTothGomb2016}. However, the authors in \cite{ChenTothGomb2016}  use global time steps (which precludes the possibility of introducing local time stepping), and perhaps more troublesome, they drop mass conservation for their framework
to work.}

In choosing building blocks for AMR code, it has been argued that single-stage, single-step methods are advantageous~\cite{Colella1990,Balsara2009},  partly because fewer synchronizations are needed per step than multistage RK methods.  
%
%
\newtext{The fact that single-stage, single-step methods do not have an issue with these synchronizations} 
is possibly one of the reasons \newtext{they } 
have gained much attention in the past two decades, \newtext{and one reason
we are choosing to pursue these methods.}
Broadly construed, these methods are based on Lax and Wendroff's original idea of using the Cauchy-Kovalevskaya procedure to convert temporal derivatives into spatial derivatives in order to define a numerical method~\cite{Lax1960}.  Notable high-order single-stage, single-step methods include the Arbitrary DERivative (ADER) methods~\cite{Titarev2002,Toro2001}, the \newtext{Lax-Wendroff} 
finite difference WENO methods~\cite{Qiu2003}, the Lax-Wendroff discontinuous Galerkin (DG) methods~\cite{Qiu2005} and \newtext{space-time schemes applied directly to second-order wave equations \cite{Henshaw2006, Banks2012}}.  
Of the three classes of high-order methods based upon Lax-Wendroff time stepping, only the ADER methods have been applied to magnetohydrodynamics~\cite{Balsara2009,Balsara2013,Boscheri2014}, whereas similar investigations have not been done for the other classes.  \newtext{An additional advantage that 
single-stage single-step Taylor methods offer is their \emph{low-storage} opportunities.  This requires 
care, 
because these methods can easily end up requiring the same amount of storage as their equivalent RK counterpart (e.g., if each time derivative is stored in order
to reduce coding complexity).}

The current work is based on the Taylor discretization of the
Picard integral formulation of the finite difference WENO (PIF-WENO) method~\cite{Seal2014b}.
Compared with other WENO methods that use Lax-Wendroff
time discretizations~\cite{Qiu2003}, our method has the advantage that
its focus is on constructing high-order Taylor expansions of the 
fluxes (which are used to define a conservative method through WENO reconstruction) as opposed to the conserved variables.
This allows, for example, the adaptation of a positivity-preserving limiter, which we describe in this document.

An important issue in simulations of MHD systems is the controlling of the
divergence error of the magnetic field, since numerical schemes based on the
transport equations alone will, in general, accumulate errors in the
divergence of the magnetic field. Failure to address this issue creates an unphysical force parallel
to the magnetic field~\cite{Brackbill1980}, and if this is not taken care of, it
will often lead to failure of the simulation code.  Popular techniques
used to solve this problem include (1) the non-conservative eight-wave
method~\cite{Gombosi1994}, (2) the projection method~\cite{Brackbill1980},
(3) the hyperbolic divergence cleaning method~\cite{Dedner2002}, and (4) the
various \newtext{constrained} 
transport methods~\cite{Balsara2004,Balsara1999a,Christlieb2014,Dai1998,Evans1988,Fey2003,Helzel2011,Helzel2013,Rossmanith2006}.
T\'{o}th conducts an extensive survey in~\cite{Toth2000}.

The current paper uses the unstaggered constrained transport framework proposed by Rossmanith \cite{Rossmanith2006}.  This framework evolves a vector potential \newtext{that} sits on the same mesh as the conserved quantities.  \newtext{T}his vector potential to correct the magnetic field.  Historically, the term ``constrained transport'' has been used to refer to a class of methods that incorporates the divergence-free condition into the discretization of the transport equation of the magnetic field, often done in a way that can be interpreted as maintaining an electric field on a staggered mesh~\cite{Evans1988,Toth2000}.  Some authors actually still distinguish between this type of ``constrained transport'' and the ``vector potential'' approach~\cite{Jardin2012}.  However, as is well-known, evolving a vector potential is conceptually equivalent to evolving an electric field.  The unstaggered approach has the added benefit of ease for potential embedding in an AMR framework.

An important piece in any vector-potential based constrained transport method
is the discretization of the evolution equation of the vector potential.  This
evolution equation is a nonconservative weakly hyperbolic system, and it can
be treated numerically from this viewpoint~\cite{Helzel2013}.  An alternative
approach is to view it as a modified system of Hamilton-Jacobi
equations~\cite{Christlieb2014}.  The current work adopts the latter approach,
and uses a method inspired by a \newtext{Lax-Wendroff} 
numerical scheme for
Hamilton-Jacobi equations that was proposed in~\cite{Qiu2007}.  The
artificial resistivity terms used in~\cite{Christlieb2014} are adapted
into the present single-stage, single-step scheme.

One further challenge for numerical simulations of MHD is that of retaining the positivity of the density and pressure.  This is critical when the thermal pressure takes up only a small portion of the total energy, \newtext{i.e.,} 
when the $\beta$ value is small.  Almost all positivity-preserving methods exploit the presumed positivity-preserving property of certain low-order schemes, such as the Lax-Friedrichs scheme.  Whereas earlier methods often rely on switching between high- and low-order updates~\cite{Balsara1999b,Janhunen2000}, more recent work tends to use combinations of the two.  Examples include the work of Balsara~\cite{Balsara2012} and Cheng et al.~\cite{Cheng2013}, which limit the conserved quantities at certain nodal points, and the work of Christlieb et al.~\cite{Christlieb2015}, which combines high- and low-order fluxes through a single free parameter at each flux interface.  In the current work, we adopt an approach similar to that used in~\cite{Christlieb2015}.  This approach seeks a suitable convex combination of the high- and low-order fluxes at each cell at each time step.  This flux limiter we adopt is an adaptation of the maximum principle preserving (MPP) flux limiter~\cite{Xu2014} (that has its roots in flux corrected transport schemes~\cite{Boris1997,Zalesak1979}) for the purposes of retaining positivity of the density and pressure.  We note that positivity-preserving has also been investigated for hydrodynamics.  The limiters mentioned in this paragraph can be viewed as generalizations of limiters that have been applied to Euler's equations~\cite{Zhang2010,Christlieb2015a,Seal2014c}.

The outline of the paper is as follows.  We first review the constrained transport framework in Section~\ref{sec:CT}.  Then in Section~\ref{sec:PIFWENO}, we give a brief review of the Taylor-discretization PIF-WENO method, which we use for updating the conserved quantities.  In Section~\ref{sec:DA}, we describe our method for evolving the vector potential, and the positivity-preserving limiter is described in Section~\ref{sec:PP}.  In Section~\ref{sec:NR}, we show the high-order accuracy and robustness of our method through a series of test problems.  The conclusion and future work is given in Section~\ref{sec:CFW}.

\section{A constrained transport framework}
\label{sec:CT}
In this section we review the unstaggered constrained transport framework that was proposed in~\cite{Rossmanith2006}.
In conservation form,  the ideal MHD equations are given by
\begin{gather}
    \label{eq:MHD}
    {\partial_t}
    \begin{bmatrix}
        \rho \\
        \rho \uvec \\
        \En \\
        \Bvec
    \end{bmatrix}
    + \divg
    \begin{bmatrix}
        \rho \uvec \\
        \rho \uvec \otimes \uvec + (p + \frac{1}{2}\lVert\Bvec\rVert^2)\mathbb{I} - \Bvec \otimes \Bvec \\
        \uvec(\En + p + \frac{1}{2}\lVert\Bvec\rVert^2) - \Bvec(\uvec\cdot\Bvec) \\
        \uvec \otimes \Bvec - \Bvec \otimes \uvec
    \end{bmatrix}
    = 0, \\
    \label{eq:DivergenceFree}
    \divg \Bvec = 0,
\end{gather}
where $\rho$ is the mass density, $\rho \uvec$ is the momentum, $\En$ is the
total energy, $p$ is the thermal pressure, $\left\|\cdot\right\|$ is the
Euclidean vector norm, $\gamma = 5/3$ is the ideal gas constant, and the
pressure satisfies the equation of state
\begin{equation}
    \label{eq:EquationOfState}
    \En = \frac{p}{\gamma - 1} + \frac{\rho\left\| \uvec \right\|^2}{2} + \frac{\left\|\Bvec\right\|^2}{2}.
\end{equation}
Equation~\eqref{eq:MHD} together with Equation~\eqref{eq:EquationOfState} form
a hyperbolic \newtext{system} of conservation laws. 
\newtext{The eigendecomposition for this system} 
is described in
\cite{Brio1987}.  For all the simulations in the current paper, we use
the eigen-decomposition described in~\cite{Powell1999} (other decompositions
have different scalings for the left and right eigenvectors).
Equation~\eqref{eq:DivergenceFree} is not a constraint, but rather an
involution~\cite{Dafermos2010}, in the sense that if it is satisfied for the
initial conditions, \newtext{then} it is satisfied \newtext{for} all later times.  However, 
this condition must be explicitly enforced 
in the discretization of the system 
at each time step in order to maintain numerical stability.

The unstaggered constrained transport method in~\cite{Rossmanith2006} makes use of the fact that because the magnetic field $\Bvec$ is divergence-free, it admits a vector potential $\Avec$ \newtext{that satisfies}
\begin{equation}
    \label{eq:BEqualsCurlA}
    \Bvec  = \nabla \times \Avec.
\end{equation}
It is therefore possible to maintain, at the discrete level, the divergence-free property of $\Bvec$ by evolving $\Avec$ alongside the conserved quantities, and correcting $\Bvec$ to be the curl of $\Avec$.

\newtext{
The evolution equation for the magnetic potential can be derived by starting with Maxwell's equations.  That is, if we start with Faraday's law (instead of starting with the
ideal MHD equations), we have
\begin{equation}
   \partial_t  \Bvec = - \curl \mathbf{E}.
\end{equation}
Next, we approximate the electric field by Ohm's law for a perfect conductor
 \begin{equation}
     \mathbf{E} = \Bvec \times \uvec,
 \end{equation}
which yields the induction equation used for ideal MHD:
\begin{equation}
    \label{eq:BEquationCurl}
    \partial_t  \Bvec + \curl (\Bvec \times \uvec) = 0.
\end{equation}
(After applying some appropriate vector identities,
Eqn.~\eqref{eq:BEquationCurl} can be written in a conservative form.)
By substituting~\eqref{eq:BEqualsCurlA} into~\eqref{eq:BEquationCurl}, we obtain
\begin{equation}
    \curl(\partial_t  \Avec + (\curl \Avec)\times \uvec) = 0.
\end{equation}
This implies there is a scalar function $\psi$ such that
\begin{equation}
    \label{eq:AEquationGaugeTBD}
    \partial_t  \Avec + (\curl \Avec) \times \uvec = -\nabla \psi.
\end{equation}
Different choices of $\psi$ in~\eqref{eq:AEquationGaugeTBD} correspond to
different gauge conditions.  It \newtext{is} noted in~\cite{Helzel2011} that the Weyl
gauge, which sets $\psi \equiv 0$, leads to stable numerical solutions.  With
the Weyl gauge, the evolution equation for magnetic
potential~\eqref{eq:AEquationGaugeTBD} becomes
\begin{equation}
    \label{eq:AEquationWeylGauge}
    \partial_t  \Avec + (\curl \Avec) \times \uvec = 0.
\end{equation}
}

In the constrained transport framework, both the conserved quantities ${\bf q } := (\rho, \rho\uvec, \En, \Bvec)$ and the magnetic potential $\Avec$ are evolved.  The time step from $t^n$ to $t^{n+1}$ proceeds as follows:
\begin{enumerate}[{\bf StepCT 1}]
    \item Discretize~\eqref{eq:MHD} and update the conserved quantities
        \begin{equation}
            \label{eq:qUpdate}
            (\rho^n, \rho\uvec^n, \En^n, \Bvec^n) \leadsto (\rho^{n+1}, \rho\uvec^{n+1}, \En^*, \Bvec^*),
        \end{equation}
        where $\En^*$ and $\Bvec^*$ are subject to corrections described in Steps~\ref{item:BCorrection} and~\ref{item:ECorrection} below.
    \item Discretize~\eqref{eq:AEquationWeylGauge} and update the magnetic potential:
            $\Avec^n \leadsto \Avec^{n+1}$.
    \item \label{item:BCorrection} Correct $\Bvec$ to be the curl of $\Avec$ through
        \begin{equation}
            \label{eq:BCorrection}
            \Bvec^{n+1} = \curl \Avec.
        \end{equation}
        \newtext{Note that this step modifies the pressure.}
    \item \label{item:ECorrection}  \newtext{Modify the total energy so the pressure remains unchanged (see below).} 
\end{enumerate}

In the current work, the discretization we use for~\eqref{eq:qUpdate} is the
Taylor-discretization of the PIF-WENO method~\cite{Seal2014b}.  This method is
a single-stage single-step method for \newtext{systems of} hyperbolic conservation laws.  We give a brief review of this method in Section~\ref{sec:PIFWENO}.

\newtext{It is noted} 
in~\cite{Helzel2011} that Equation~\eqref{eq:AEquationWeylGauge} is a \newtext{weakly hyperbolic system that must be} 
treated carefully in order to avoid numerical instabilities.  Along the lines of our previous work in~\cite{Christlieb2014}, we use a discretization that is inspired by numerical schemes for Hamilton-Jacobi equations, \newtext{in order to be able to define non-oscillatory derivatives of the magnetic field.}  We present this discretization in Section~\ref{sec:DA}.

For certain problems, such as those involving low-$\beta$ plasma, negative
density or pressure occur in numerical simulations, even when constrained
transport is present.  This requires the use of a positivity-preserving
limiter~\cite{Christlieb2015,Seal2014c,Xu2014}.  This limiter preserves the
positivity of density and pressure by modifying the numerical flux 
that evolves the conserved variables.  We describe this
limiter in detail in Section~\ref{sec:PP}.  
When this limiter is used, we also apply the following fix to
\newtext{the pressure} in {\bf StepCT \ref{item:ECorrection}} by \newtext{modifying the total energy through}
\begin{equation}
    \label{eq:KeepP}
    \En^{n+1} = \En^* + \frac{1}{2} (\lVert\Bvec^{n+1}\rVert^2 - \lVert\Bvec^{*}\rVert^2).
\end{equation}
This \newtext{energy ``correction''} keeps the pressure the same as before the magnetic field
correction \newtext{in~\eqref{eq:BCorrection}}.  While this clearly has the disadvantage of violating the
conservation of energy, it is nonetheless needed in order to preserve the
positivity of pressure \newtext{(otherwise the modifications made to the
magnetic field to obtain a divergence free property could potentially cause
the pressure to become negative)}.  This technique is explored in \cite{Balsara1999a,Christlieb2015,Toth2000}.  
In the results presented in the current paper, we use the energy correction~\eqref{eq:KeepP} if and only if the positivity-preserving limiter is turned on.

\section{The Taylor-discretization PIF-WENO method}
\label{sec:PIFWENO}
We now briefly review the Taylor-discretization Picard integral formulation
weighted essentially non-oscillatory (PIF-WENO) method \cite{Seal2014b}.  This method applies to
generic hyperbolic conservation laws in arbitrary dimensions, of which the
ideal MHD equation is an example.   

For the purpose of illustration, we present the method here for a 2D system.
In 2D, a hyperbolic conservation law takes the form
\begin{equation}
    \label{eq:2DHyperbolicConservationLaw}
   \partial_t \mathbf{q} + \partial_x \mathbf{f}(\mathbf{q}) + \partial_y {\mathbf{g}}(\mathbf{q}) = 0,
\end{equation}
where  $\mathbf{q}(t, x, y):\R^+ \times \R^2 \to \R^m$
is the vector of~$m$~conserved variables, and $\mathbf{f}, \mathbf{g}: \R^m \to \R^m$ are the two components of the flux function.  Formally integrating~\eqref{eq:2DHyperbolicConservationLaw} in time from $t^n$ to $t^{n+1}$, one \newtext{arrives at} the \emph{Picard integral formulation} of~\eqref{eq:2DHyperbolicConservationLaw} given by
\begin{equation}
    \mathbf{q}^{n+1} = \mathbf{q}^n - \Delta t \partial_x \mathbf{F}^n(x, y)  - \Delta t \partial_y \mathbf{G}^n(x, y),
\end{equation}
where the \emph{time-averaged fluxes} $\mathbf{F}$ and $\mathbf{G}$ are defined as
\begin{equation}
    \mathbf{F}^n(x, y) = \frac{1}{\Delta t} \int_{t^n}^{t^{n+1}}\mathbf{f}(\mathbf{q}(t, x, y))\,dt, \quad \text{and} \quad
    \mathbf{G}^n(x, y) = \frac{1}{\Delta t} \int_{t^n}^{t^{n+1}}\mathbf{g}(\mathbf{q}(t, x, y))\,dt.
\end{equation}

Given a domain~$\Omega = [a_x, b_x] \times [a_y, b_y]$, the point-wise approximations $\mathbf{q}^n_{i, j} \approx \mathbf{q}(t^n, x_i, y_j)$ are placed at
\begin{align}
    x_i &= a_x + \left(i - \frac{1}{2}\right) \Delta x, &
    \Delta x &= \frac{b_x - a_x}{m_x}, &
    i &\in \{1, \ldots, m_x\}, \\
    y_i &= a_y + \left(i - \frac{1}{2}\right) \Delta y, &
    \Delta y &= \frac{b_y - a_y}{m_y}, &
    y &\in \{1, \ldots, m_y\},
\end{align}
and time $t = t^n$\newtext{, where $m_x$ and $m_y$ are positive integers}.  The PIF-WENO scheme solves~\eqref{eq:2DHyperbolicConservationLaw} by setting
\begin{equation}
    \label{eq:2DPIFWENOUpdate}
    \mathbf{q}^{n+1}_{i, j} = \mathbf{q}^{n}_{i, j} - \frac{\Delta t}{\Delta x} \left(\hat{\mathbf{F}}^{n}_{i+1/2,j} - \hat{\mathbf{F}}^{n}_{i-1/2,j}\right) - \frac{\Delta t}{\Delta y} \left(\hat{\mathbf{G}}^{n}_{i, j+1/2} - \hat{\mathbf{G}}^{n}_{i,j-1/2}\right),
\end{equation}
where $\hat{\mathbf{F}}_{i\pm 1/2, j}$ and $\hat{\mathbf{G}}_{i, j\pm 1/2}$ are obtained by
applying the classical WENO reconstruction to the time-averaged fluxes
$\mathbf{F}^n_{i, j}$ and $\mathbf{G}^n_{i, j}$ instead of to the ``frozen-in-time'' fluxes, as
is traditional in a method of lines (MOL) formulation.  Here, $\mathbf{F}^n_{i, j}$ and
$\mathbf{G}^n_{i, j}$ are approximated using a third-order Taylor expansion in time,
followed by an application of Cauchy-Kovalevskaya procedure to replace the
temporal derivatives with spatial derivatives.   The resulting scheme on the
conserved quantities is third-order accurate in time.  We omit the details here
and refer the reader to~\cite{Seal2014b}.  \newtext{Extensions to three dimensions follow by including a third component for the flux function.}

\section{The evolution of the magnetic potential equation}
\label{sec:DA}

\newtext{Recall that the evolution equation for the magnetic potential is defined in \eqref{eq:AEquationWeylGauge}.}
If we expand the curl and cross-product operator, we see that $\Avec$ satisfies
\begin{gather}
    \label{eq:AQuasilinear}
    \partial_t \Avec + N^x \partial_x \Avec + N^y \partial_y \Avec + N^z \partial_z \Avec = 0, \\
\intertext{where}
    \label{eq:NMatrices}
    N^x =
        \begin{bmatrix}
            0 &    -u^y &    -u^z \\
            0 &    u^x &    0 \\
            0 &    0    &    u^x
        \end{bmatrix}, \quad
    N^y =
        \begin{bmatrix}
            u^y  &     0 &    0 \\
            -u^x &     0 &    -u^z \\
            0    &     0 &    u^y
        \end{bmatrix}, \quad
    N^z =
        \begin{bmatrix}
            u^z  &    0   &    0 \\
            0    &    u^z &    0 \\
            -u^x &    -u^y &    0
        \end{bmatrix}
\end{gather}
\newtext{are matrices defined by components of the velocity field $\mathbf{u} = \left( u^x, u^y, u^z \right)$.}
The system~{\eqref{eq:AQuasilinear}--\eqref{eq:NMatrices}} has the following properties that we take into account when making a discretization:
\begin{enumerate}
    \item This system is weakly hyperbolic. \newtext{This means that all of the eigenvalues of \eqref{eq:AQuasilinear} are real, but in contrast to being strictly hyperbolic, the system is not necessarily diagonalizable.}  While this weak hyperbolicity is an artifact of how the magnetic potential is evolved, it must be treated carefully to avoid \newtext{creating} numerical instabilities \cite{Helzel2011};
    \item Each equation of this system can be viewed as a Hamilton-Jacobi equation with a source term~\cite{Christlieb2014}.
\end{enumerate}
\newtext{One consequence of this second observation is that we can use Hamilton-Jacobi technology in order to define magnetic fields that have 
non-oscillatory derivatives.  That is,
a single derivative of ${\bf B}$ is a second derivative of ${\bf A}$, in which case care need be taken in order to retain a non-oscillatory property.}

In the next two subsections, we will describe how we discretize the
system~{\eqref{eq:AQuasilinear}--\eqref{eq:NMatrices}}, with these two
properties in mind.  Before we do that, we make several remarks.  The first is
that in the discretization, $\Avec$ sits on the same mesh as the conserved
quantities, hence the name ``\emph{unstaggered} constrained transport''.  The
second is that the curl operator $\nabla\times$ is discretized in the same way
as in~\cite{Christlieb2014}, \newtext{which defines a divergence free magnetic field (at the discrete level).}  \newtext{A} final remark is that the discretization we
present here is third-order accurate \newtext{in time}, which is in accordance with the
discretization we use for the conserved quantities.

\subsection{The 2D magnetic potential equation}
\label{sec:2DAEquation}
We begin with the two-dimensional case, where the conserved quantities and the magnetic potential do not depend on~$z$.  In this case, the divergence-free condition is
\begin{equation}
    \divg \Bvec = \partial_x B^x + \partial_y B^y = 0.
\end{equation}
Thus $B^z$ has no impact on the divergence of~$\Bvec$.  It therefore suffices to correct only the first two components of the magnetic field.  A computationally cost-efficient way of doing this is to discard the first two components of the magnetic potential, evolve only the third component according to the evolution equation
\begin{equation}
    \label{eq:AEquation2D}
    \partial_t  A^z + u^x \partial_x A^z + u^y \partial_y A^z = 0,
\end{equation}
and for \newtext{the} magnetic field correction, \newtext{only correct} $B^x$ and $B^y$ \newtext{through}
\begin{equation}
\label{eq:BCorrection2D}
    B^x = \partial_y A^z \quad \text{and} \quad B^y = -\partial_x A^z.
\end{equation}
This is equivalent to imposing $\partial_z A^x = 0$ and $ \partial_z A^y = 0$ on the system~{\eqref{eq:AQuasilinear}--\eqref{eq:NMatrices}}, and 
skipping $B^z$ in the magnetic field correction in Eqn.~\eqref{eq:BCorrection}.

Equation~\eqref{eq:AEquation2D} has the favorable property of being \emph{strongly hyperbolic}~\cite{Rossmanith2006}, in contrast to the weak hyperbolicity in the 3D case.  This equation is also a Hamilton-Jacobi equation, with Hamilton principal function $A^z$, and Hamiltonian
\begin{equation}
    H(t, x, y, \partial_x A^z, \partial_y A^z ) = u^x(t, x, y) \partial_x A^z + u^y(t, x, y) \partial_y A^z.
\end{equation}
It is therefore possible to discretize Equation~\eqref{eq:AEquation2D} using a numerical scheme for Hamilton-Jacobi equations \cite{Christlieb2014}.  In line with the single-stage single-step theme of the proposed scheme, we use a method that is inspired by the \newtext{Lax-Wendroff} 
WENO schemes for Hamilton-Jacobi equations from \cite{Qiu2007}.

A \newtext{Lax-Wendroff} 
WENO scheme for Hamilton-Jacobi
equations is based on the Taylor expansion in time of the Hamilton principal
function.  In order to retain third-order accuracy in time, 
we require a total of three time derivatives of $A^z_{i,j}$:
\begin{equation}
    \label{eq:A3Taylor}
    A^{z\; n+1}_{i, j} = A^{z\; n}_{i, j} + \Delta t \, \partial_t  A^{z\; n}_{i, j} + \frac{\Delta t^2}{2!} \, \partial_{t}^2 A^{z\; n}_{i, j}
    + \frac{\Delta t^3}{3!} \, \partial_{t}^3 A^{z\; n}_{i, j},
\end{equation}
where the temporal derivatives on the right-hand-side are computed in the following manner:
\begin{enumerate}
    \item The first derivative is approximated by the Lax-Friedrichs type numerical Hamiltonian, with high-order reconstruction of $\partial_{x} A^{z}$ and $\partial_{y} A^{z}$ \cite{Osher1991}.  Namely, we define
    \begin{equation}
    \label{eq:A3t}
    \begin{aligned}
        \partial_t  A^{z}_{i, j} & := -\hat{H}( \partial_x A^{z-}_{i, j},  \partial_x A^{z+}_{i, j}, 
        						\partial_y A^{z-}_{i, j},  \partial_y A^{z+}_{i, j}) \\
                         & = -u^x_{i, j}\left(\frac{\partial_x A^{z-}_{i, j} + \partial_x A^{z+}_{i, j}}{2}\right)
                             -u^y_{i, j}\left(\frac{\partial_y A^{z-}_{i, j} + \partial_y A^{z+}_{i, j}}{2}\right) \\
                         & \quad +\alpha^x \left(\frac{\partial_x A^{z+}_{i, j} - \partial_x A^{z-}_{i, j}}{2}\right)
                             +\alpha^y \left(\frac{\partial_y A^{z+}_{i, j} - \partial_y A^{z-}_{i, j}}{2}\right),
    \end{aligned}
    \end{equation}
where
    \begin{equation}
        \alpha^x = \max_{i, j} \lvert u^x_{i, j} \rvert
        \quad\text{and}\quad
        \alpha^y = \max_{i, j} \lvert u^y_{i, j} \rvert
    \end{equation}
are the maximum velocities taken over the entire grid, (i.e., the numerical flux is the global Lax-Friedrichs flux), 
and $\partial_{x}  A^{z \pm}_{ i, j}$, $\partial_{y} A^{z\pm}_{i, j}$ are defined by WENO reconstructions through
    \begin{equation}
        \label{eq:ReconstructA3SpatialDerivatives}
        \begin{aligned}
       \partial_x A^{z-}_{i, j}
        &=
        \PhiWENO\left(
            \frac{\Delta^{+}_{x}A^z_{i-3, j}}{\Delta x},
            \frac{\Delta^{+}_{x}A^z_{i-2, j}}{\Delta x},
            \frac{\Delta^{+}_{x}A^z_{i-1, j}}{\Delta x},
            \frac{\Delta^{+}_{x}A^z_{i, j}}{\Delta x},
            \frac{\Delta^{+}_{x}A^z_{i+1, j}}{\Delta x},
            \right), \\
         \partial_x  A^{z+}_{i, j}
        &=
        \PhiWENO\left(
            \frac{\Delta^{+}_{x}A^z_{i+2, j}}{\Delta x},
            \frac{\Delta^{+}_{x}A^z_{i+1, j}}{\Delta x},
            \frac{\Delta^{+}_{x}A^z_{i, j}}{\Delta x},
            \frac{\Delta^{+}_{x}A^z_{i-1, j}}{\Delta x},
            \frac{\Delta^{+}_{x}A^z_{i-2, j}}{\Delta x},
            \right),  \\
       \partial_y  A^{z-}_{i, j}
        &=
        \PhiWENO\left(
            \frac{\Delta^{+}_{y}A^z_{i, j-3}}{\Delta y},
            \frac{\Delta^{+}_{y}A^z_{i, j-2}}{\Delta y},
            \frac{\Delta^{+}_{y}A^z_{i, j-1}}{\Delta y},
            \frac{\Delta^{+}_{y}A^z_{i, j}}{\Delta y},
            \frac{\Delta^{+}_{y}A^z_{i, j+1}}{\Delta y},
            \right), \\
        \partial_y A^{z+}_{i, j}
        &=
        \PhiWENO\left(
            \frac{\Delta^{+}_{y}A^z_{i, j+2}}{\Delta y},
            \frac{\Delta^{+}_{y}A^z_{i, j+1}}{\Delta y},
            \frac{\Delta^{+}_{y}A^z_{i, j}}{\Delta y},
            \frac{\Delta^{+}_{y}A^z_{i, j-1}}{\Delta y},
            \frac{\Delta^{+}_{y}A^z_{i, j-2}}{\Delta y},
            \right),
        \end{aligned}
    \end{equation}
and $\Delta^{+}_{x}A^z_{i, j} := A^z_{i+1, j} - A^z_{i, j}$ and
$\Delta^{+}_{y}A^z_{i, j} := A^z_{i, j+1} - A^z_{i, j}$.  
The function $\PhiWENO$ is the classical fifth-order WENO reconstruction whose coefficients can be found in many sources (e.g., ~\cite{Jiang1996,Shu1998, Christlieb2014}).
\newtext{Note that this difference operator is designed for Hamilton-Jacobi
problems, and does not produce the typical flux difference form that most
hyperbolic solvers produce.}

    \item The higher derivatives $\partial_{t}^2 A^{z}$ and $\partial_{t}^3 A^{z}$ are converted into spatial derivatives by way of the Cauchy-Kovalevskaya procedure, and the resulting spatial derivatives are approximated using central differences.  For example, $\partial_{t}^2 A^{z}$ is converted into spatial derivatives by
    \begin{align} 
    	\nonumber  
         \partial_{t}^2  A^z & =\partial_{t}  (\partial_{t} A^z) \\ \nonumber
                 & = \partial_t (-u^x \partial_{x} A^z -u^y \partial_{y} A^z) \\ \nonumber
                 & = -\partial_t u^x \, \partial_{x} A^z - u^x \partial_t
                 \partial_{x} A^z - \partial_t u^y \, \partial_{y} A^z - u^y \partial_t \partial_{y} A^z \\
        \label{eq:StoreTemp}
                 & = -\partial_t u^x \, \partial_{x} A^z - u^x \partial_{x} (\partial_{t} A^z) - \partial_t u^y \, \partial_{y} A^z - u^y \partial_{y} (\partial_{t} A^z) \\
        \nonumber
                 & = -\partial_t u^x \, \partial_{x} A^z - u^x \, \partial_x (-u^x \partial_{x} A^z - u^y \partial_{y} A^z)
                     -\partial_t u^y \, \partial_{y} A^z - u^y \, \partial_y (-u^x \partial_{x} A^z - u^y \partial_{y} A^z) \\
        \nonumber
                 & = -\partial_t u^x \, \partial_{x} A^z - u^x (-\partial_{x}
                 u^x \, \partial_{x} A^z - u^x \partial_{x}^2 A^z  - \partial_{x} u^y \, \partial_{y} A^z - u^y \partial_x \partial_{y} A^z ) \\
        \nonumber
                 & \quad -\partial_t u^y \, \partial_{y} A^z - u^y (-\partial_{y} u^x \, \partial_{x} A^z - u^x \partial_x \partial_{y} A^z  - \partial_{y} u^y \, \partial_{y} A^z - u^y \partial_{y}^2 A^z ),
    \end{align}
    where
    \begin{equation}
        \label{eq:u1t}
        \partial_t u^x =  \partial_t \left(\frac{\rho u^x}{\rho}\right)
              = \frac{ \partial_t (\rho u^x) \rho - (\rho u^x) \partial_t \rho}{\rho^2},
    \end{equation}
    and
    \begin{equation}
        \label{eq:u2t}
        \partial_t u^y  = \partial_t \left(\frac{\rho u^y}{\rho}\right)
              = \frac{ \partial_t (\rho u^y) \rho - (\rho u^y)  \partial_t \rho}{\rho^2},
    \end{equation}
    with $\partial_t (\rho u^x)$, $\partial_t (\rho u^y)$, and $\partial_t \rho$ converted into spatial derivatives by way of~\eqref{eq:MHD}. 
\newtext{Note that we do not find it necessary to use a FD WENO discretization for $\rho_t$.}
    There are a total of 49 distinct spatial derivatives that need to be
    approximated in {\eqref{eq:StoreTemp}--\eqref{eq:u2t}}.  Similar expressions exist for $\partial_t^3 A^z$.  
    \newtext{These} are done by using the central differencing formulae
    \begin{align}
	    \label{eq:FirstOfSpatialDerivatives}
       \partial_x U_{i, j} &= \frac{- U_{i-1, j} + U_{i+1, j} }{2\Delta x},    \\
       \partial_{x}^2 U_{i, j} &= \frac{U_{i-1, j} - 2 U_{i, j} + U_{i+1, j}}{\Delta x^2}, \\
       \partial_{x} \partial_{y} U_{i, j} &= \frac{U_{i-1, j-1} - U_{i-1, j+1} - U_{i+1, j-1} + U_{i+1, j+1}}{4\Delta x \Delta y}, \\
       \partial_{x}^3 U_{ i, j} &= \frac{-U_{i-2, j} + 2 U_{i-1, j} - 2 U_{i+1, j} + U_{i+2, j}}{2\Delta x^3}, \\
       \label{eq:LastOfSpatialDerivatives}
       \partial_{x}^2 \partial_y U_{i, j} &= \frac{2 U_{i, j-1} - 2 U_{i, j+1} - U_{i-1, j-1} + U_{i-1, j+1} - U_{i+1, j-1} + U_{i+1, j+1}}{2\Delta x^2 \Delta y},
    \end{align}
    and similar ones for the $y$-, $yy$-, $xyy$- and $yyy$-derivatives, where $U$ is any of $\rho$, $\rho u^x$, $\rho u^y$, $\rho u^z$, $\En$, $B^x$, $B^y$, $B^z$, or $A^z$.
 \end{enumerate}
  
\newtext{  
\begin{remark}
In the current work, the magnetic field correction in Eqn. \eqref{eq:BCorrection2D} is discretized to fourth-order accuracy with central difference formulas.  For example,
we define the $y$-derivative of $A^z$ to be
\begin{equation}
	B^x_{i,j} = \partial_y A^z \approx \frac{ A^z_{i,j-2} - 8 A^z_{i,j-1} + 8 A^z_{i,j+1} - A^z_{i, j+2 } }{12 \Delta y},
\end{equation}
and others similarly.
Therefore in the discretization of equations~\eqref{eq:MHD} and~\eqref{eq:AQuasilinear}, we aim for fourth-order accuracy in space.  For the fluid 
equations in~\eqref{eq:MHD}, we use the fifth-order accurate spatial
discretization defined in ~\cite{Seal2014b}.  In the case
of~\eqref{eq:AQuasilinear}, Eqn.~\eqref{eq:A3Taylor} is a fifth-order
approximation to $\partial_t  A^z(t^n, x_i, y_j)$ since the
reconstructions~\eqref{eq:ReconstructA3SpatialDerivatives} give, for example, 
\begin{equation}
	\partial_x A^{z+}_{ i,j} =  \partial_x A^{z}(t^n, x_i, y_j) + \bigO(\Delta x^5),
\end{equation}
whereas Eqns.~{\eqref{eq:FirstOfSpatialDerivatives}--\eqref{eq:LastOfSpatialDerivatives} are chosen so that Eqn.~\eqref{eq:A3Taylor}} is fourth-order accurate in space when $\Delta t = \bigO(\Delta x)$.  This fourth-order accuracy is verified in our numerical results.
\end{remark}
}

\newtext{
\begin{remark}
In each dimension, the stencil needed to compute $\mathbf{B}$ from $\mathbf{A}$ is 5 points wide, and an inspection of~ \eqref{eq:ReconstructA3SpatialDerivatives} and~{\eqref{eq:FirstOfSpatialDerivatives}--\eqref{eq:LastOfSpatialDerivatives}} indicates that the stencil needed to compute $\mathbf{A}$ is 7 points wide.  This results in a stencil that is 11 points wide in each dimension.  A more careful analysis of~{\eqref{eq:FirstOfSpatialDerivatives}--\eqref{eq:LastOfSpatialDerivatives}} shows that this stencil is indeed contained within the stencil given by the PIF-WENO discretization of Eqn.~\eqref{eq:2DHyperbolicConservationLaw} (which is sketched in Fig.~2 of~\cite{Seal2014b}).  We present a sketch of the stencil required to discretize the magnetic potential in Fig.~\ref{fig:stencil}.
\end{remark}
}

    
\begin{figure}
\begin{center}
                \includegraphics[width=0.8\textwidth]{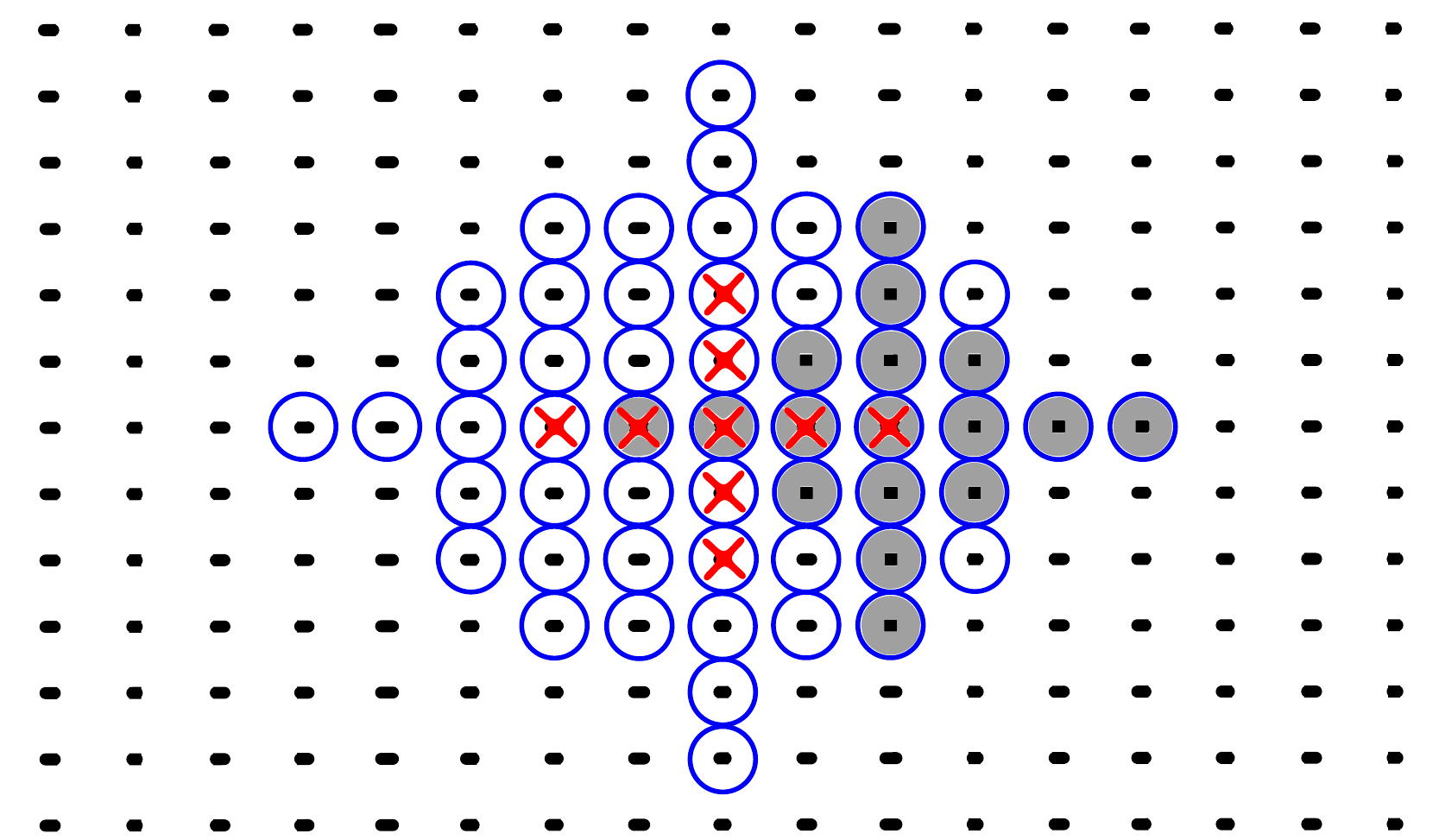}
\caption{\newtext{
                    Stencil needed for magnetic potential.  
                    The values of $A^z$ at each ``$\times$'' (save the centermost value) is needed to correct the value of $\Bvec$ at the center.  The shaded points are needed in 
                    Eqns.~\eqref{eq:ReconstructA3SpatialDerivatives} and~{\eqref{eq:FirstOfSpatialDerivatives}--\eqref{eq:LastOfSpatialDerivatives}} to compute the value of $A^z$ at the rightmost 
                    ``$\times$'' point.  Note that this stencil is contained in the stencil shown in~\cite{Seal2014b}, Fig.~2. }
\label{fig:stencil} }
\end{center}
\end{figure}
    
We note that our discretization differs from that found in~\cite{Qiu2007} in
that we do not store any lower temporal derivatives for use in computation of
higher temporal derivatives.  That is, the method proposed in~\cite{Qiu2007} stores the values of $\partial_t  A^z$ computed in Equation~\eqref{eq:StoreTemp}, and makes use of the central differences of $\partial_t  A^z$. \newtext{Although} that approach saves the trouble of expanding Eqn. \eqref{eq:LastOfSpatialDerivatives}, 
our approach has the following advantages that are in line with our goals:
\begin{description}
        \item[Minimal storage.]  Our approach avoids the necessity of storing
        the lower temporal derivatives \newtext{(at the expense of a more
        complicated code).}  Note here that it would be not only $\partial_{t}
        A^{z}$ and $\partial_{tt} A^{z}$, but also the temporal derivatives of
        $\rho$, $\rho u^x$, $\rho u^y$, $\rho u^z$, $\En$, $B^x$, $B^y$, and
        $B^z$ that must be stored if the approach in~\cite{Qiu2007} were
        \newtext{to be} taken.  This would lead to a formidable amount of temporary storage
        when applied to the magnetic potential evolution equation in 3D;
        \item[Smaller stencils.]  \newtext{The method in~\cite{Qiu2007} would
        require a stencil that is 15 points wide in each dimension in order to
        provide the necessary $\Avec$ values for correcting the magnetic
        field.   Our method only needs a stencil that is 11 points wide.}
        This \newtext{somewhat} simplifies the future work of embedding the
        scheme into an AMR framework\newtext{, but the biggest improvement is
        the reduction of temporary storage.}
    \end{description}

\subsection{The 3D magnetic potential equation}
The 3D magnetic potential equation~{\eqref{eq:AQuasilinear}--\eqref{eq:NMatrices}} is weakly hyperbolic, and therefore special attention is needed.  This can be treated by introducing artificial resistivity terms into the system~\cite{Christlieb2014,Helzel2011,Helzel2013}.  To illustrate the technique, we consider the first row of the system which is
\begin{equation}
    \label{eq:A1Equation}
    \partial_{t} A^x - u^y \partial_{x} A^y - u^z \partial_{x} A^z + u^y \partial_{y} A^x + u^z \partial_{z} A^x = 0.
\end{equation}
A discretization of $\partial_{t} A^{x}$ in the spirit of~\eqref{eq:A3t} would give the following numerically problematic formulation
\begin{align}
    \nonumber
    \partial_t A^{x}_{i, j, k}     
        & = -u^y_{i, j, k}\left(\frac{\partial_y A^{x-}_{i, j, k}  + \partial_y A^{x+}_{i, j, k} }{2}\right)
        -u^z_{i, j, k}\left(\frac{\partial_z A^{x-}_{i, j, k} + \partial_z A^{x+}_{i, j, k} }{2}\right) \\
    \nonumber
        & \quad +\alpha^y \left(\frac{\partial_y A^{x+}_{i, j, k}  - \partial_y A^{x-}_{i, j, k} }{2}\right)
            +\alpha^z \left(\frac{\partial_z A^{x+}_{i, j, k} - \partial_z A^{x-}_{i, j, k}}{2}\right) \\
    \label{eq:A1tNoArtificialViscosity}
        & \quad +u^y_{i, j, k}\left(\frac{\partial_{x} A^{y-}_{i, j, k} + \partial_{x} A^{y+}_{i, j, k}}{2}\right)
                +u^z_{i, j, k}\left(\frac{\partial_{x} A^{z-}_{i, j, k} + \partial_{x} A^{z+}_{i, j, k}}{2}\right)              
\end{align}
where
\begin{equation}
    \alpha^y = \max_{i, j, k} \lvert u^y_{i, j, k} \rvert
    \quad\text{and}\quad
    \alpha^z = \max_{i, j, k} \lvert u^z_{i, j, k} \rvert,
\end{equation}
and $\partial_{y} A^{x-}_{i, j, k}$, $\partial_{y} A^{x+}_{i, j, k}$, $\partial_{z} A^{x-}_{i, j, k}$, $\partial_{z} A^{x+}_{i, j, k}$, $A^{2-}_{x\; i, j, k}$, $A^{2+}_{x\; i, j, k}$, $A^{3-}_{x\; i, j, k}$, and $A^{3+}_{x\; i, j, k}$ are reconstructed in a manner similar to~\eqref{eq:ReconstructA3SpatialDerivatives}.  The problem with~\eqref{eq:A1tNoArtificialViscosity} is that this formulation lacks numerical resistivity in the $x$-direction.  

\newtext{As is the case in~\cite{Christlieb2014}, we find that the
\newtext{addition} of an artificial resistivity term
to~\eqref{eq:A1tNoArtificialViscosity} yields satisfactory numerical results.}
With the artificial resistivity term, the evolution equation~\eqref{eq:A1Equation} becomes
\begin{equation}
    \partial_{t} A^x - u^y \partial_{x} A^y - u^z \partial_{x} A^z + u^y \partial_{y} A^x + u^z \partial_{z} A^x
    = \epsilon^x \partial_{x}^2 A^x,
\end{equation}
where $\epsilon^x$ ideally satisfies
\begin{equation}
    \epsilon^x = \bigO(\Delta x^6), \text{when $\partial_{x} A^x$ is smooth},
\end{equation}
and
\begin{equation}
    \epsilon^x = \bigO(\Delta x), \text{when $\partial_{x} A^x$ is non-smooth}.
\end{equation}

We define the artificial resistivity $\epsilon^x$ in the following manner.  
Define a smoothness indicator~$\gamma^x$ for~$\partial_{x} A^x$ as follows:
\begin{equation}
    \gamma^x_{i, j, k} = \left| \frac{a^-}{a^- + a^+} - \frac{1}{2} \right|,
\end{equation}
where
\begin{equation}
    a^- = \left(\epsilon + (\Delta x \partial_{x} A^{x-}_{i, j, k})^2\right)^{-2}
    \quad \text{and} \quad
    a^+ = \left(\epsilon + (\Delta x \partial_{x} A^{x+}_{i, j, k})^2\right)^{-2}.
\end{equation}
\newtext{Here, $\epsilon$ is a small positive number introduced to avoid dividing by 
	a number close to 0 when the potential is smooth
	($\epsilon = 10^{-8}$ in all our numerical simulations).}
Now we define $\epsilon^x$ to be
\begin{equation}
    \label{eq:epsilon1}
    \epsilon^x = 2\nu \gamma^x \frac{\Delta x^2}{\Delta t},
\end{equation}
where $\nu$ is a positive constant that controls the magnitude of the
artificial resistivity.   
For a more detailed discussion of the weak hyperbolicity and the reasoning 
leading up to~\eqref{eq:epsilon1}, \newtext{we refer the reader to~\cite{Helzel2011,Christlieb2014}}.
\newtext{In \cite{Helzel2011}, it is shown that 
$\nu$ has to be in the range of $[0, 0.5]$ to maintain the stability up to CFL
one for their finite volume method.
Through numerical experimentation, we find that $\nu$ in the vicinity of
$0.01$ is sufficient to control potential oscillations in the magnetic field.}

We thus use the following formulation as our discretization of~\eqref{eq:A1Equation}
\begin{align}
    \nonumber
    \partial_t A^{x}_{i, j, k}     
        & = -u^y_{i, j, k}\left(\frac{\partial_y A^{x-}_{i, j, k} + \partial_y A^{x+}_{i, j, k}}{2}\right)
        -u^z_{i, j, k}\left(\frac{\partial_z A^{x-}_{i, j, k} + \partial_z A^{x+}_{i, j, k} }{2}\right) \\
    \nonumber
        & \quad +\alpha^y \left(\frac{\partial_y A^{x+}_{i, j, k}  - \partial_y A^{x-}_{i, j, k} }{2}\right)
            +\alpha^z \left(\frac{\partial_z A^{x+}_{i, j, k} - \partial_z A^{x-}_{i, j, k}}{2}\right) \\
        \nonumber
        & \quad +u^y_{i, j, k}\left(\frac{\partial_x A^{y-}_{i, j, k} + \partial_x A^{y+}_{i, j, k}}{2}\right)
                +u^z_{i, j, k}\left(\frac{\partial_x A^{z-}_{i, j, k} + \partial_x A^{z+}_{i, j, k}}{2}\right)  \\      
        & \quad +2 \nu \gamma^x_{i, j, k} \left(\frac{A^x_{i-1, j, k} - 2 A^x_{i, j, k} + A^x_{i+1, j, k}}{\Delta t}\right). 
\end{align}
The discretization of $\partial_{t} A^{y}$ and $\partial_t  A^z$ are similar and we omit them for brevity.

For our Lax-Wendroff formulation, it remains to discretize the higher temporal
derivatives of the components of $\Avec$.  We find it suffices to apply the
same techniques as in Section~\ref{sec:2DAEquation} for $\partial_{t}^2 A^{z}$
and $\partial_t^3 A^z$.  Namely, we convert all the temporal derivatives to
spatial derivatives via the Cauchy-Kovalevskaya procedure, and approximate the
\newtext{resulting}
spatial derivatives with central differences.  We note that the artificial
resistivity term is only added to the first temporal derivative.

\section{Positivity preservation}
\label{sec:PP}

\newtext{The scheme presented thus far can be applied to a large class of problems.  However, for problems where the plasma density or pressure
are near zero, Gibb's phenomenon can cause these values to become negative, and hence the numerical simulation will instantly fail.  As a final ingredient to the solver,
we introduce an additional option for retaining positivity of the solution.  Given the lack of number of degrees of freedom, this limiter comes at the expense of energy conservation, but these regions only occur in small areas where the density or pressure become negative.}

A positivity-preserving scheme can be \newtext{constructed} 
by modifying the fluxes $\hat{\mathbf{F}}$ and $\hat{\mathbf{G}}$ in~\eqref{eq:2DPIFWENOUpdate}.  
%
Let $\hat{\mathbf{f}}_{i+1/2,j}$ and $\hat{\mathbf{g}}_{i,j+1/2}$ be the (global) Lax-Friedrichs fluxes defined by
\begin{equation}
    \label{eq:LaxFriedrichsFluxes}
    \begin{aligned}
        \hat{\mathbf{f}}_{i+1/2,j}
        &= \frac{1}{2} (\mathbf{f}(\mathbf{q}^n_{i+1, j} ) + \mathbf{f}(\mathbf{q}^n_{i, j}) - \alpha^x (\mathbf{q}^n_{i+1, j} - \mathbf{q}^n_{i, j})), \\
        \hat{\mathbf{g}}_{i, j+1/2}
        &= \frac{1}{2} (\mathbf{g}(\mathbf{q}^n_{i, j+1} ) + \mathbf{g}(\mathbf{q}^n_{i, j}) - \alpha^y (\mathbf{q}^n_{i, j+1} - \mathbf{q}^n_{i, j})),
    \end{aligned}
\end{equation}
where $\alpha^x$ and $\alpha^y$ are the maximal wave speeds in the $x$ and $y$ directions, respectively.   
\newtext{This type of formulation is commonly referred to as 
Lax-Friedrich's flux splitting in the ENO/WENO literature \cite{Shu1988,Shu1989,Jiang1999}.}  

The update of $\mathbf{q}$ using these Lax-Friedrichs fluxes is 
\begin{equation}
	\label{eq:qLF}
	\mathbf{q}^{\text{LF}}_{i, j} = 
    \mathbf{q}^{n}_{i, j} - \frac{\Delta t}{\Delta x} \left(\hat{\mathbf{f}}^{n}_{i+1/2,j} - \hat{\mathbf{f}}^{n}_{i-1/2,j}\right) - \frac{\Delta t}{\Delta y} \left(\hat{\mathbf{g}}^{n}_{i, j+1/2} - \hat{\mathbf{g}}^{n}_{i,j-1/2}\right).
\end{equation}
For an 8-component state vector $\mathbf{q} = (\rho, \rho \uvec, \En, \Bvec)$, we also introduce the notation $\rho(\mathbf{q})$ and $p(\mathbf{q})$ to represent the density and the thermal pressure of $\mathbf{q}$.  Also, let $\epsilon_{\rho}$ and $\epsilon_{p}$ be small positive numbers.
  
The following claim was conjectured in \cite{Cheng2013}.
\begin{claim}
    \label{claim:LaxFriedrichsPreservesPositivity}
    If $\rho(\mathbf{q}^{n}_{i, j}) > \epsilon_{\rho}$ and $p(\mathbf{q}^{n}_{i, j}) > \epsilon_p$ for all $i$, $j$,
    and the CFL number is less than or equal to $0.5$, we then have
    $\rho(\mathbf{q}^{\text{LF}}_{i, j}) > \epsilon_{\rho}$ and 
    $p(\mathbf{q}^{\text{LF}}_{i, j}) > \epsilon_{p}$ for all $i$, $j$.
\end{claim}

Though not proven, this claim is verified in \cite{Cheng2013} using a fairly large number of random values of~$\mathbf{q}$.  Our positivity limiter assumes this claim is true.  However, as is noted in \cite{Cheng2013}, if a different flux can be found such that it satisfies a property similar to that of the Lax-Friedrichs fluxes stated in Claim~\ref{claim:LaxFriedrichsPreservesPositivity}, we can then use \newtext{these} different fluxes in place of $\hat{\mathbf{f}}_{i+1/2,j}$ and $\hat{\mathbf{g}}_{i,j+1/2}$ in the construction of our positivity limiter.

The modified fluxes take the form
\begin{align}
    \label{eq:ModifiedFluxesFirst}
    \tilde{\mathbf{F}}_{i+1/2,j} &= \theta_{i+1/2,j} (\hat{\mathbf{F}}_{i+1/2, j} - \hat{\mathbf{f}}_{i+1/2, j}) + \hat{\mathbf{f}}_{i+1/2,j}, \\
    \label{eq:ModifiedFluxesLast}
    \tilde{\mathbf{G}}_{i,j+1/2} &= \theta_{i,j+1/2} (\hat{\mathbf{G}}_{i, j+1/2} - \hat{\mathbf{g}}_{i, j+1/2}) + \hat{\mathbf{g}}_{i,j+1/2}, 
\end{align}
where $\theta_{i+1/2,j}$ and $\theta_{i,j+1/2}$ are chosen such that
\begin{itemize}
    \item $0 \leq \theta_{i+1/2, j} \leq 1$, $0 \leq \theta_{i,j+1/2} \leq 1$,
    \item the update \eqref{eq:2DPIFWENOUpdate} with $\hat{\mathbf{F}}$ and $\hat{\mathbf{G}}$ replaced by the modified fluxes $\tilde{\mathbf{F}}$ and $\tilde{\mathbf{G}}$ defined in {\eqref{eq:ModifiedFluxesFirst}--\eqref{eq:ModifiedFluxesLast}} yields positive density and pressure, and
    \item while subject to the positivity requirement just stated, $\theta_{i+1/2, j}$ and $\theta_{i,j+1/2}$ should be as close to $1$ as possible, 
    so that the high-order fluxes are used in regions where violation of positivity is unlikely to happen.
\end{itemize}

Following \cite{Christlieb2015}, we choose each $\theta$ in a series of two steps:
\begin{enumerate}[\StepTheta (i)]
    \item \label{steptheta:one} For each $i$, $j$, find ``large'' candidate limiting parameters $\Lambda_{L,\,I_{i,j}}$, $\Lambda_{R,\,I_{i,j}}$, $\Lambda_{D,\,I_{i,j}}$, and 
    $\Lambda_{U,\,I_{i,j}} \in [0,1]$ such that for all 
\begin{equation}
    	(\theta_L, \theta_R, \theta_D, \theta_U) \in [0, \Lambda_{L,\,I_{i,j}}] \times [0, \Lambda_{R,\,I_{i,j}}] \times [0, \Lambda_{D,\,I_{i,j}}] \times [0, \Lambda_{U,\,I_{i,j}}],
\end{equation}
	the update defined by
    \begin{equation}
    \label{eq:qAsFunctionOfTheta}
    \begin{aligned}
        \mathbf{q}
        & := \mathbf{q}^n_{i,j}
        -\frac{\Delta t}{\Delta x} \left((\theta_R (\hat{\mathbf{F}}_{i+1/2, j} - \hat{\mathbf{f}}_{i+1/2, j}) + \hat{\mathbf{f}}_{i+1/2,j})
        -(\theta_L (\hat{\mathbf{F}}_{i-1/2, j} - \hat{\mathbf{f}}_{i-1/2, j}) + \hat{\mathbf{f}}_{i-1/2,j})
        \right) \\
        & \quad -\frac{\Delta t}{\Delta y} \left((\theta_U (\hat{\mathbf{G}}_{i, j+1/2} - \hat{\mathbf{f}}_{i, j+1/2}) + \hat{\mathbf{f}}_{i,j+1/2})
                -(\theta_D (\hat{\mathbf{F}}_{i, j-1/2} - \hat{\mathbf{f}}_{i, j-1/2}) + \hat{\mathbf{f}}_{i,j-1/2})
                \right)
    \end{aligned}
    \end{equation}
    satisfies $\rho(\mathbf{q}) > \epsilon_{\rho}$ and $p(\mathbf{q}) > \epsilon_p$;
    
    \item For each $i$, $j$, set $\theta_{i+1/2,j} := \min\{\Lambda_{R, I_{i,j}}, \Lambda_{L, I_{i+1,j}}\}$ and $\theta_{i, j+1/2} := \min\{\Lambda_{U, I_{i,j}}, \Lambda_{D, I_{i, j+1}}\}$.
\end{enumerate}
We note that with $\mathbf{q}^n_{i,j}$, $\hat{\mathbf{F}}$, $\hat{\mathbf{f}}$, $\hat{\mathbf{G}}$, and $\hat{\mathbf{g}}$ already computed, Equation~\eqref{eq:qAsFunctionOfTheta} expresses the update $\mathbf{q}$ as an affine function of $\theta_L$, $\theta_R$, $\theta_D$, $\theta_U$, which, by abuse of notation, is denoted by $\mathbf{q}(\theta_L, \theta_R, \theta_D, \theta_U)$.  The coefficients of the $\theta$'s in $\mathbf{q}(\theta_L, \theta_R, \theta_D, \theta_U)$ are denoted by $\mathbf{C}_L$, $\mathbf{C}_R$, $\mathbf{C}_D$, and $\mathbf{C}_U$.  Thus we have
\begin{equation}
\begin{aligned}
     &\mathbf{C}_L = \frac{\Delta t}{\Delta x} (\hat{\mathbf{F}}_{i-1/2, j} - \hat{\mathbf{f}}_{i-1/2, j}),
    &\mathbf{C}_R = -\frac{\Delta t}{\Delta x} (\hat{\mathbf{F}}_{i+1/2, j} - \hat{\mathbf{f}}_{i+1/2, j}), \\
    &\mathbf{C}_D = \frac{\Delta t}{\Delta y} (\hat{\mathbf{G}}_{i, j-1/2} - \hat{\mathbf{g}}_{i, j-1/2}),
    &\mathbf{C}_U = -\frac{\Delta t}{\Delta y} (\hat{\mathbf{G}}_{i, j+1/2} - \hat{\mathbf{g}}_{i, j+1/2}),
\end{aligned}
\end{equation}
and, with \eqref{eq:qLF} in mind, \newtext{the limited solution is}
\begin{equation}
	\mathbf{q}(\theta_L, \theta_R, \theta_D, \theta_U) = \mathbf{q}^{\text{LF}}_{i, j} + \mathbf{C}_L \theta_L + \mathbf{C}_R \theta_R + \mathbf{C}_D \theta_D + \mathbf{C}_U \theta_U.
\end{equation}
With this notation, \StepTheta(\ref{steptheta:one}) reduces to solving the following problem for each $i$, $j$:
\begin{problem}
\label{pr:Optimization}
Given constant real vectors $\mathbf{C}_L$, $\mathbf{C}_R$, $\mathbf{C}_D$ and $\mathbf{C}_U$, and constant 8-component state vector $\mathbf{q}^{\text{LF}}_{i,j}$ such that $\rho(\mathbf{q}^{\text{LF}}_{i,j}) > \epsilon_{\rho}$ and $p(\mathbf{q}^{\text{LF}}_{i,j}) > \epsilon_p$, find $(\Lambda_{L,\,I_{i,j}}, \Lambda_{R,\,I_{i,j}}, \Lambda_{D,\,I_{i,j}}, \Lambda_{U,\,I_{i,j}})$ in $[0,1]\times [0,1] \times [0,1]\times [0,1]$ such that for all $(\theta_L, \theta_R, \theta_D, \theta_U)$ in $[0, \Lambda_{L,\,I_{i,j}}] \times [0, \Lambda_{R,\,I_{i,j}}] \times [0, \Lambda_{D,\,I_{i,j}}] \times [0, \Lambda_{U,\,I_{i,j}}]$, the expression
\begin{equation}
    \mathbf{q}(\theta_L, \theta_R, \theta_D, \theta_U) = \mathbf{q}^{\text{LF}}_{i, j} + \mathbf{C}_L \theta_L + \mathbf{C}_R \theta_R + \mathbf{C}_D \theta_D + \mathbf{C}_U \theta_U
\end{equation}
satisfies
\begin{gather}
    \rho(\mathbf{q}(\theta_L, \theta_R, \theta_D, \theta_U)) > \epsilon_{\rho}\\
    \intertext{and}
    p(\mathbf{q}(\theta_L, \theta_R, \theta_D, \theta_U)) > \epsilon_{p}.
\end{gather}
The region $[0, \Lambda_{L,\,I_{i,j}}] \times [0, \Lambda_{R,\,I_{i,j}}] \times [0, \Lambda_{D,\,I_{i,j}}] \times [0, \Lambda_{U,\,I_{i,j}}]$ should be ``as big as possible''.
\end{problem}

We note the following fact.
\begin{claim}
	The sets $S_{\rho}$ and $S$ defined by
	\begin{align}
		S_{\rho} &= \{ (\theta_L, \theta_R, \theta_D, \theta_U) \in [0,1]^4
						\, \vert \,
						\rho(\mathbf{q}(\theta_L, \theta_R, \theta_D, \theta_U)) > \epsilon_{\rho} \} \\
		\intertext{and}
		S &= \{ (\theta_L, \theta_R, \theta_D, \theta_U) \in [0,1]^4
				\, \vert \,
				\rho(\mathbf{q}(\theta_L, \theta_R, \theta_D, \theta_U)) > \epsilon_{\rho}
				\text{ and }
				p(\mathbf{q}(\theta_L, \theta_R, \theta_D, \theta_U)) > \epsilon_{p}	 \}
	\end{align}
	are both convex.
\end{claim}
\begin{proof}
	Since $\rho$ is a linear function of $\mathbf{q}$ and $\mathbf{q}$ is a affine function of $(\theta_L, \theta_R, \theta_D, \theta_U)$, we see that $\rho(\mathbf{q}(\theta_L, \theta_R, \theta_D, \theta_U))$ is an affine function of $(\theta_L, \theta_R, \theta_D, \theta_U)$.  Thus $S_{\rho}$ is the part of $[0,1]^4$ that lies on one side of a hyperplane.  This shows $S_{\rho}$ is convex.
	
	To see the convexity of $S$, note that by the equation of state \ref{eq:EquationOfState}, the pressure $p$ is a concave function of the components of $\mathbf{q}$, whenever $\rho > 0$.  Combined with the fact that $\mathbf{q}(\theta_L, \theta_R, \theta_D, \theta_U)$ is an affine function of $(\theta_L, \theta_R, \theta_D, \theta_U)$,  we see that $p(\mathbf{q}(\theta_L, \theta_R, \theta_D, \theta_U))$ is a concave function of $(\theta_L, \theta_R, \theta_D, \theta_U)$, if $\rho(\mathbf{q}(\theta_L, \theta_R, \theta_D, \theta_U)) > 0$.  This shows the convexity of $S_p$.
\end{proof}

Note that Problem \ref{pr:Optimization} is not well-defined, since the notion of ``big'' is not defined.  The algorithm we are about to describe gives a solution that is satisfactory, in the sense that this algorithm yields a positivity-preserving limiter that behaves well in our numerical tests.

We now describe this algorithm.
The first step of this algorithm is to find a ``big'' rectangular subset $R_{\rho} := [0, \Lambda^{\rho}_{L}] \times [0, \Lambda^{\rho}_{R}] \times [0, \Lambda^{\rho}_{D}] \times [0, \Lambda^{\rho}_{U}]$ of $S_{\rho}$.  The $\Lambda^{\rho}$'s are computed by
\begin{equation}
	\Lambda^{\rho}_{\weirdI} =
	\begin{cases}
		\min \left\lbrace 1, \; \frac{\rho(\mathbf{q}^{\text{LF}}_{i, j}) - \epsilon_{\rho}}{\epsilon + \mathlarger{\mathlarger{\sum}}\limits_{\substack{\weirdJ, \\ C^{(1)}_{\weirdJ} < 0}}{\left| C^{(1)}_{\weirdJ} \right|}} \right\rbrace
		&\text{if $C^{(1)}_{\weirdI} < 0$,} \\
		1
		&\text{if $C^{(1)}_{\weirdI} \geq 0$,}
	\end{cases}
\end{equation}
where $\epsilon$ is a small fixed positive number ($10^{-12}$ in all our simulations) and the subscript letters $\weirdI$ and $\weirdJ$ take values in $L$, $R$, $D$, and $U$. \newtext{The value $C^{(1)}_\weirdI$ denotes the first component of the $\mathbf{C}_\weirdI$ vector.}

The second step is to shrink this rectangular subset $R_{\rho}$ to fit into $S_p$.  The vertices of $R_{\rho}$ are denoted by $\mathbf{A}^{k_L, k_R, k_D, k_U}$, where $k_{\weirdI} = 0\text{ or } 1$, such that the $\weirdI$-th component of $\mathbf{A}^{k_L, k_R, k_D, k_U}$ is
\begin{equation}
    A^{k_L, k_R, k_D, k_U}_{\weirdI} = 
    \begin{cases}
        \Lambda^{\rho}_{\weirdI} & \text{if $k_{\weirdI} = 1$,} \\
        0    & \text{if $k_{\weirdI} = 0$.}
    \end{cases}
\end{equation}
Now for each $(k_L, k_R, k_D, k_U)$, shrink $\mathbf{A}^{k_L, k_R, k_D, k_U}$
to get $\mathbf{B}^{k_L, k_R, k_D, k_U}$ in the following way.  If
$p(\mathbf{A}^{k_L, k_R, k_D, k_U}) \geq \epsilon_p$, set $\mathbf{B}^{k_L,
k_R, k_D, k_U} = \mathbf{A}^{k_L, k_R, k_D, k_U}$.  Otherwise, we solve for
the smallest positive $r$ such that $p(r \mathbf{A}^{k_L, k_R, k_D, k_U}) \geq
\epsilon_p$, and set $\mathbf{B}^{k_L, k_R, k_D, k_U} = r \mathbf{A}^{k_L,
k_R, k_D, k_U}$.  \newtext{In order to solve for this value, we apply a total
of 10 iterations of the bisection method. A more efficient (approximate) solver could 
easily replace this step (e.g., a single step of the method of false
position).}  Note that the different
vertices $\mathbf{A}^{k_L, k_R, k_D, k_U}$ are in general shrunk by different
factors $r$.  Now we set
\begin{equation}
    \label{eq:RectangularInsidePolygon}
    \Lambda_{\weirdI, \, I_{i,j}} = \min\limits_{\substack{(k_L, k_R, k_U, k_D), \\ k_{\weirdI = 1}}}{B}^{k_L, k_R, k_U, k_D}_{\weirdI},
\end{equation}
where the subscript $\weirdI$ indicates the $\weirdI$-th component.  Note that this is equivalent to finding a rectangular subset inside the convex polygon with vertices $\mathbf{B}^{k_L, k_R, k_D, k_U}$, $k_{\weirdI} = 0, 1$.

This completes the description of our positivity-preserving limiter in 2D.  The 3D case is similar.

\section{Numerical results}
\label{sec:NR}

In this section, we present the results of numerical simulations using the proposed method.  Unless otherwise specified, constrained transport is turned on, the gas constant $\gamma = 5/3$, the CFL number is $0.5$, and for 3D simulations the artificial viscosity coefficient is \newtext{$\nu = 0.01$}.

\subsection{Smooth \Alfven wave}
\label{sec:Alfven}

The smooth \Alfven wave problem is often used for convergence studies of
numerical schemes for ideal MHD
equations~\cite{Helzel2011,Rossmanith2006,Toth2000}.  This problem is a
one-dimensional problem (computed in multiple dimensions) that has a known
smooth solution.  In 1D, the initial conditions for this problem are
\begin{equation}
    \label{eq:1DAlfvenInit}
    \begin{split}
        &(\rho, u^x, u^y, u^z, u^z, p, B^x, B^y, B^z)(0, x) \\
        &\quad = (1, 0, 0.1\sin(2\pi x), 0.1\cos(2\pi x), 0.1, 1, 0.1\sin(2\pi x), 0.1\cos(2\pi x)).
    \end{split}
\end{equation}
\newtext{The exact solution to~\eqref{eq:1DAlfvenInit} propagates with the \Alfven speed that is unity (i.e., $\mathbf{q}(t,x) = \mathbf{q}(0,x+t)$).}
The 2D and 3D smooth \Alfven wave problems are obtained from the 1D problem by rotating the direction of wave propagation.

\subsubsection{Smooth \Alfven wave\newtext{: The 2D problem}}

The 2D version of the smooth \Alfven wave problem is obtained by rotating the
direction of propagation by an angle of $\phi$, so that the wave now
propagates in direction $\mathbf{n} = \langle -\cos \phi, -\sin \phi, 0
\rangle$.  Identical to ~\cite{Helzel2011,Christlieb2014}, the computational
domain we use is $[0, 1/\cos \phi] \times [0, 1/\sin \phi]$, \newtext{where $\phi = \tan^{-1}(0.5)$}. 
\newtext{Periodic boundary conditions are applied on all four sides.}

\newtext{For this problem, we present numerical results with and without the
energy correction.  In Table \ref{tab:2DAlfvenThird}, we observe the overall
third-order accuracy of the method, and in Table \ref{tab:2DAlfvenFour}, we
refine $\Delta t$ faster than the mesh spacing as well as run the solution to
a shorter final time in order to extract the spatial order of accuracy.  
For these test cases, we observe the predicted fourth-order accuracy in space. 
of convergence in space, third-order in time, and little difference between
the results obtained with and without the energy correction turned on.  When
the flag for the positivity-preserving limiter is turned on in the code, we
see identical results as without it, because this problem does not have
density or pressure that is near zero. 
The choice of $m_y = 2 m_x$ allows for $\Delta x = \Delta y$.  
}

\begin{table}
\begin{center}
\caption{2D smooth \Alfven wave.  Here, we show $L^{\infty}$-errors at a final
time of $t=1.0$  for the solution with and without the energy ``correction'' step.  
Left two columns have the energy correction turned off, and the rightmost
columns have the energy correction turned on. Because time is only discretized
to third-order accuracy, we observe the predicted third-order accuracy of the
solver here. \label{tab:2DAlfvenThird} }
{\footnotesize
\begin{tabular}{|l|l||l|l|l|l||l|l|l|l|}
\hline
Mesh & CFL & Error in $\mathbf{B}$ & Order & Error in $A^z$ & Order & Error in $\mathbf{B}$ & Order & Error in $A^z$ & Order\\
\hline\hline
$32 \times 64$ & 0.5 & $3.842 \times 10^{-5}$ & --- & $5.356 \times 10^{-6}$ & --- 
                     & $3.848 \times 10^{-5}$ & --- & $5.320 \times 10^{-6}$ & ---\\
$64 \times 128$ & 0.5 & $4.940 \times 10^{-6}$ & $2.96$ & $7.530 \times 10^{-7}$ & $2.83$
                      & $4.938 \times 10^{-6}$ & $2.96$ & $7.469 \times 10^{-7}$ & $2.83$\\
$128 \times 256$ & 0.5 & $6.324 \times 10^{-7}$ & $2.97$ & $9.697 \times 10^{-8}$ & $2.96$ 
                      & $6.318 \times 10^{-7}$ & $2.97$ & $9.628 \times 10^{-8}$ & $2.96$\\
$256 \times 512$ & 0.5 & $8.020 \times 10^{-8}$ & $2.98$ & $1.218 \times 10^{-8}$ & $2.99$
                       & $8.009 \times 10^{-8}$ & $2.98$ & $1.210 \times 10^{-8}$ & $2.99$
\\
\hline\hline
\end{tabular}
}
\end{center}
\end{table}

\begin{table}
\begin{center}
\caption{2D smooth \Alfven wave.  Here, we show $L^{\infty}$-errors at a short final
time of $t=0.01$  for the solution with and without the energy ``correction'' step.  
Left two columns have the energy correction turned off, and the rightmost
columns have the energy correction turned on.  Here, we refine $\Delta t$
faster than $\Delta x$ in order to expose the spatial order of accuracy of the
solver.  Because we only use a fourth-order accurate spatial discretization
for $\nabla \times {\bf A}$, we only observe fourth-order accuracy despite the
fact that the fluid variables are discretized to fifth-order accuracy.
\label{tab:2DAlfvenFour} }
{\footnotesize
\begin{tabular}{|l|l||l|l|l|l||l|l|l|l|}
\hline
Mesh & CFL & Error in $\mathbf{B}$ & Order & Error in $A^z$ & Order & Error in $\mathbf{B}$ & Order & Error in $A^z$ & Order\\
\hline\hline
$32 \times 64$ & 0.5 & $3.852 \times 10^{-6}$ & --- & $1.078 \times 10^{-7}$ & ---
                     & $3.852 \times 10^{-6}$ & --- & $1.078 \times 10^{-7}$ & ---\\
$64 \times 128$ & 0.25 & $2.356 \times 10^{-7}$ & $4.03$ & $8.121 \times 10^{-9}$ & $3.73$
                       & $2.356 \times 10^{-7}$ & $4.03$ & $8.121 \times 10^{-9}$ & $3.73$\\
$128 \times 256$ & 0.125 & $1.466 \times 10^{-8}$ & $4.01$ & $5.190 \times 10^{-10}$ & $3.97$
                         & $1.466 \times 10^{-8}$ & $4.01$ & $5.190 \times 10^{-10}$ & $3.97$\\
$256 \times 512$ & 0.0625 & $9.117 \times 10^{-10}$ & $4.01$ & $3.291 \times 10^{-11}$ & $3.98$
                          & $9.117 \times 10^{-10}$ & $4.01$ & $3.291 \times 10^{-11}$ & $3.98$
\\
\hline\hline
\end{tabular}
}
\end{center}
\end{table}

\subsubsection{Smooth \Alfven wave\newtext{: The 3D problem}}

The setup we use here is the same as that used in~\cite{Helzel2011}, Section 6.2.1.  The direction of propagation is
\begin{equation}
\mathbf{n} = \langle-\cos\phi\cos\theta, -\sin\phi\cos\theta, \sin\theta\rangle,
\end{equation}
and the computational domain is
\begin{equation}
\left[0, \frac{1}{\cos\phi\cos\theta}\right] \times
\left[0, \frac{1}{\sin\phi\cos\theta}\right] \times
\left[0, \frac{1}{\sin\theta}\right],
\end{equation}
\newtext{where $\phi = \theta = \tan^{-1} (0.5)$.}
\newtext{Periodic boundary conditions are imposed on all directions.}

\newtext{We again seek to numerically investigate the spatial and temporal
orders of accuracy, with and without the energy correction step.  The errors in $\Bvec$ and $\Avec$ are presented in {Tables
\ref{tab:3DAlfvenLong}}--\ref{tab:3DAlfvenShort}.  Here we choose 
$m_y = m_z = 2 m_x$ so that $\Delta x = \Delta y = \Delta z / \cos \theta$.  Similar to the 2D
case, we observe fourth-order of convergence in space, third-order in time,
and little difference between the results obtained with and without
the energy correction step turned on.}

\begin{table}
\begin{center}
\caption{3D smooth \Alfven wave.  In this table we show the
$L^{\infty}$-errors at a moderate time of $t=1.0$.  In the left
columns the positivity-preserving limiter is off, and the
the positivity-preserving limiter (and the energy correction step) is turned on
for the results in the right columns.  Because time is discretized to
third-order accuracy, the final method is formally only third-order accurate
in time.  In Table \ref{tab:3DAlfvenShort} we run the solver to a short final
time in order to expose the spatial order of accuracy.
\label{tab:3DAlfvenLong}}
{\footnotesize
\begin{tabular}{|l|l||l|l|l|l||l|l|l|l|}
\hline
Mesh & CFL & Error in $\Bvec$ & Order & Error in $\Avec$ & Order & Error in $\Bvec$ & Order & Error in $\Avec$ & Order\\
\hline\hline
$16 \times 32 \times 32$ & 0.5 & $4.784 \times 10^{-4}$ & --- & $5.116 \times 10^{-5}$ & ---
                               & $4.882 \times 10^{-4}$ & --- & $5.176 \times 10^{-5}$ & ---\\
$32 \times 64 \times 64$ & 0.5 & $2.452 \times 10^{-5}$ & $4.29$ & $3.181 \times 10^{-6}$ & $4.01$
                               & $2.485 \times 10^{-5}$ & $4.30$ & $3.191 \times 10^{-6}$ & $4.02$\\
$64 \times 128 \times 128$ & 0.5 & $3.093 \times 10^{-6}$ & $2.99$ & $4.612 \times 10^{-7}$ & $2.79$
                                 & $3.105 \times 10^{-6}$ & $3.00$ & $4.621 \times 10^{-7}$ & $2.79$\\
$128 \times 256 \times 256$ & 0.5 & $3.969 \times 10^{-7}$ & $2.96$ & $6.133 \times 10^{-8}$ & $2.91$
                                  & $3.977 \times 10^{-7}$ & $2.96$ & $6.147 \times 10^{-8}$ & $2.91$
\\
\hline\hline
\end{tabular}
}
\end{center}
\end{table}

\begin{table}
\begin{center}
\caption{3D smooth \Alfven wave.  Here, we show the $L^{\infty}$-errors at a
short final time of $t=0.01$.  In addition, we refine $\Delta t$ faster than
the mesh spacing in order to extract the spatial order of accuracy.  Left two
columns have the turned off, and the right two columns have the
positivity-preserving limiter (as well as the correction step) turned on. 
The results are almost identical.
\label{tab:3DAlfvenShort}}
{\footnotesize
\begin{tabular}{|l|l||l|l|l|l||l|l|l|l|}
\hline
Mesh & CFL & Error in $\Bvec$ & Order & Error in $\Avec$ & Order & Error in $\Bvec$ & Order & Error in $\Avec$ & Order\\
\hline\hline
$16 \times 32 \times 32$ & 0.5 & $6.752 \times 10^{-5}$ & --- & $5.715 \times 10^{-7}$ & ---
                               & $6.752 \times 10^{-5}$ & --- & $5.715 \times 10^{-7}$ & ---\\
$32 \times 64 \times 64$ & 0.25 & $4.280 \times 10^{-6}$ & $3.98$ & $3.856 \times 10^{-8}$ & $3.89$
                                & $4.280 \times 10^{-6}$ & $3.98$ & $3.856 \times 10^{-8}$ & $3.89$\\
$64 \times 128 \times 128$ & 0.125 & $2.666 \times 10^{-7}$ & $4.00$ & $2.613 \times 10^{-9}$ & $3.88$
                                   & $2.666 \times 10^{-7}$ & $4.00$ & $2.613 \times 10^{-9}$ & $3.88$\\
$128 \times 256 \times 256$ & 0.0625 & $1.652 \times 10^{-8}$ & $4.01$ & $1.711 \times 10^{-10}$ & $3.93$
                                     & $1.652 \times 10^{-8}$ & $4.01$ & $1.712 \times 10^{-10}$ & $3.93$
\\
\hline\hline
\end{tabular}
}
\end{center}
\end{table}

\subsection{2D rotated shock tube problem}

Similar to the smooth \Alfven problems, the rotated shock tube problem is a 1D
problem, with direction of wave propagation rotated a certain angle.  The
setup we use in the current work is the same as that in~\cite{Christlieb2014},
Section 7.2, which we repeat here \newtext{for completeness}.

The initial conditions consist of a shock
\begin{equation}
(\rho, u_{\perp}, u_{\parallel}, u^z, p, B_{\perp}, B_{\parallel}, B^z)
=
    \begin{cases}
        (1, -0.4, 0, 0, 1, 0.75, 1, 0) &
        \text{if $\xi<0$,} \\
        (0.2, -0.4, 0, 0, 0.1, 0.75, -1, 0) &
        \text{if $\xi\geq 0$,}
    \end{cases}
\end{equation}
where $\xi = x \cos \phi + y \sin \phi$, 
and $u_{\perp}$ and $B_{\perp}$ are vector components perpendicular to the shock interface, and $u_{\parallel}$ and $B_{\parallel}$  are the vector components parallel to the shock interface.  Namely
\begin{gather}
u^x = u_{\perp} \cos \phi - u_{\parallel} \sin \phi, \quad u^y = u_{\perp} \sin \phi + u_{\parallel} \cos \phi , \\
B^x = B_{\perp} \cos \phi - B_{\parallel} \sin \phi, \quad B^y = B_{\perp} \sin \phi + B_{\parallel} \cos \phi .
\end{gather}
The initial condition for magnetic potential is
\begin{equation}
A^z(0, x, y) = 
\begin{cases}
    0.75 \eta - \xi & \text{if $\xi\leq 0$, } \\
    0.75 \eta + \xi & \text{if $\xi > 0$, }
\end{cases}
\end{equation}
where $\eta = -x \sin \phi + y \cos \phi$.

The computational domain is $[-1.2, 1.2]\times[-1, 1]$ with a $180\times 150$
mesh.  The boundary conditions used are zeroth order extrapolation on the
conserved quantities and first-order extrapolation on the magnetic potential.
\newtext{That is, we set the conserved quantities at the ghost points to be
identical to the last value on the interior of the domain, and we define
values for the magnetic potential at ghost points through repeated
extrapolation of two point stencils starting with two interior points.}
On the top and bottom boundaries, the direction of extrapolation is parallel
to the shock interface.

\begin{figure}
\begin{center}
    \begin{tabular}{cc}
        \includegraphics[width=0.4\textwidth]{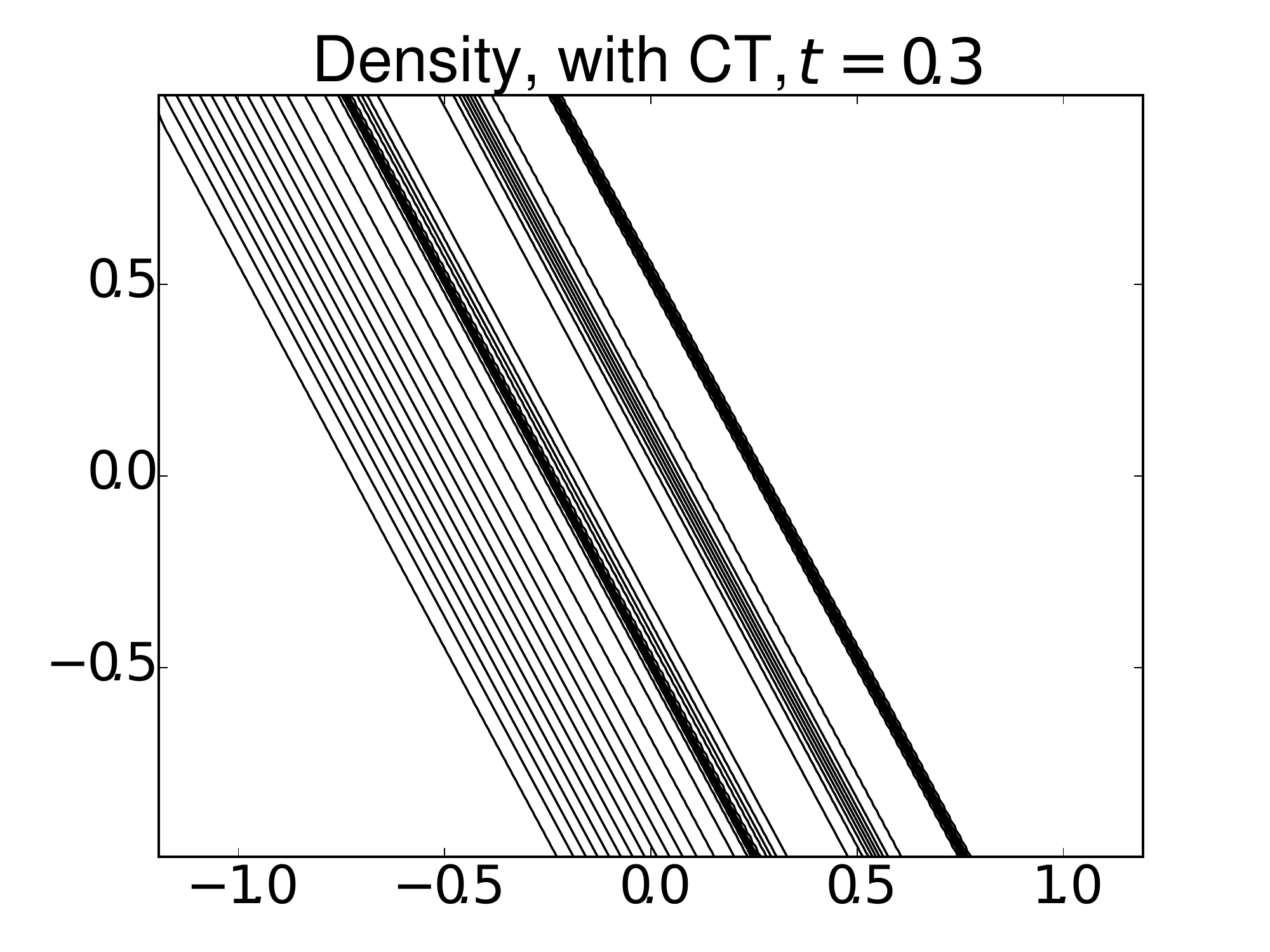} &
        \includegraphics[width=0.4\textwidth]{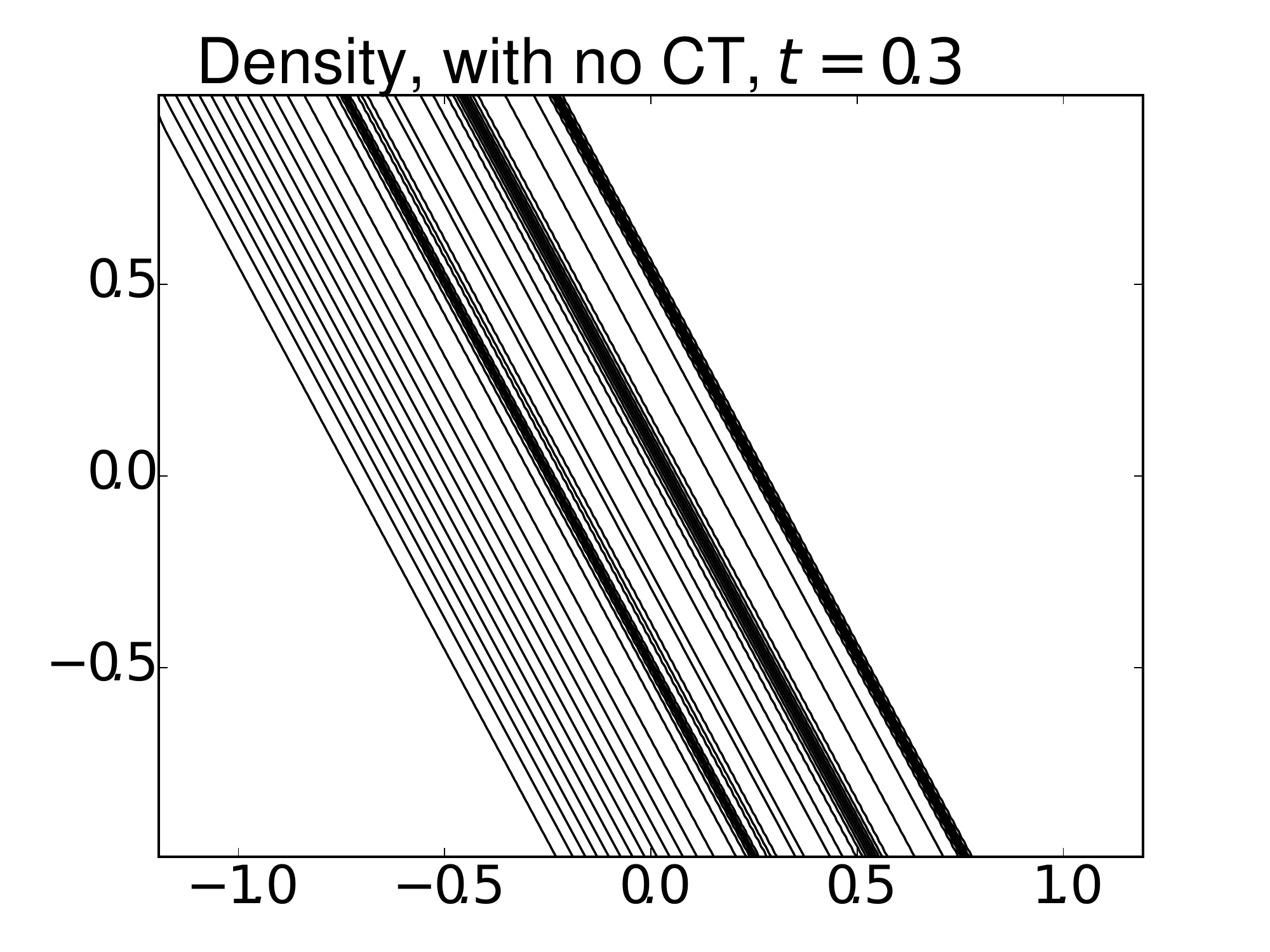} \\
        \includegraphics[width=0.4\textwidth]{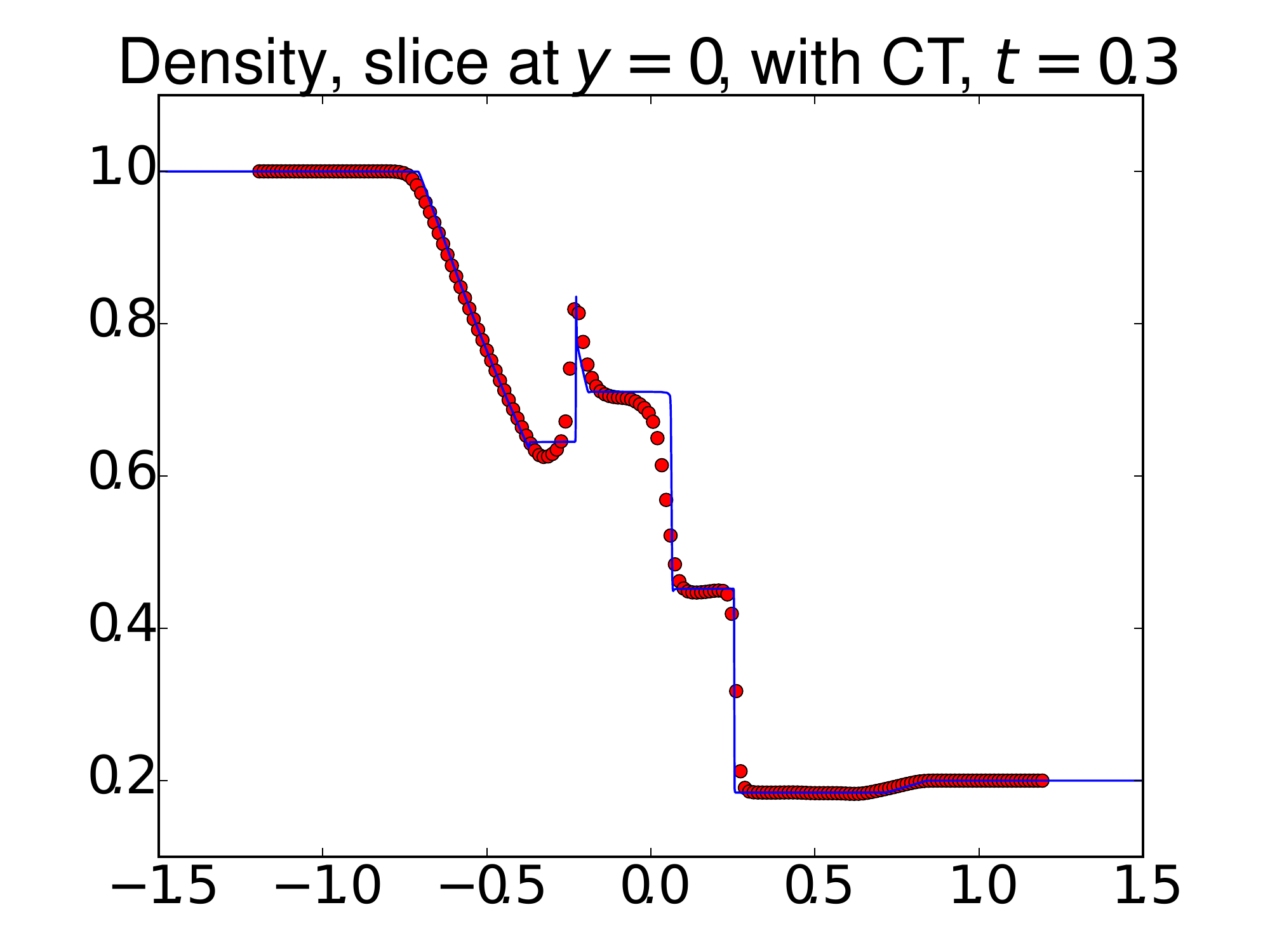} &
        \includegraphics[width=0.4\textwidth]{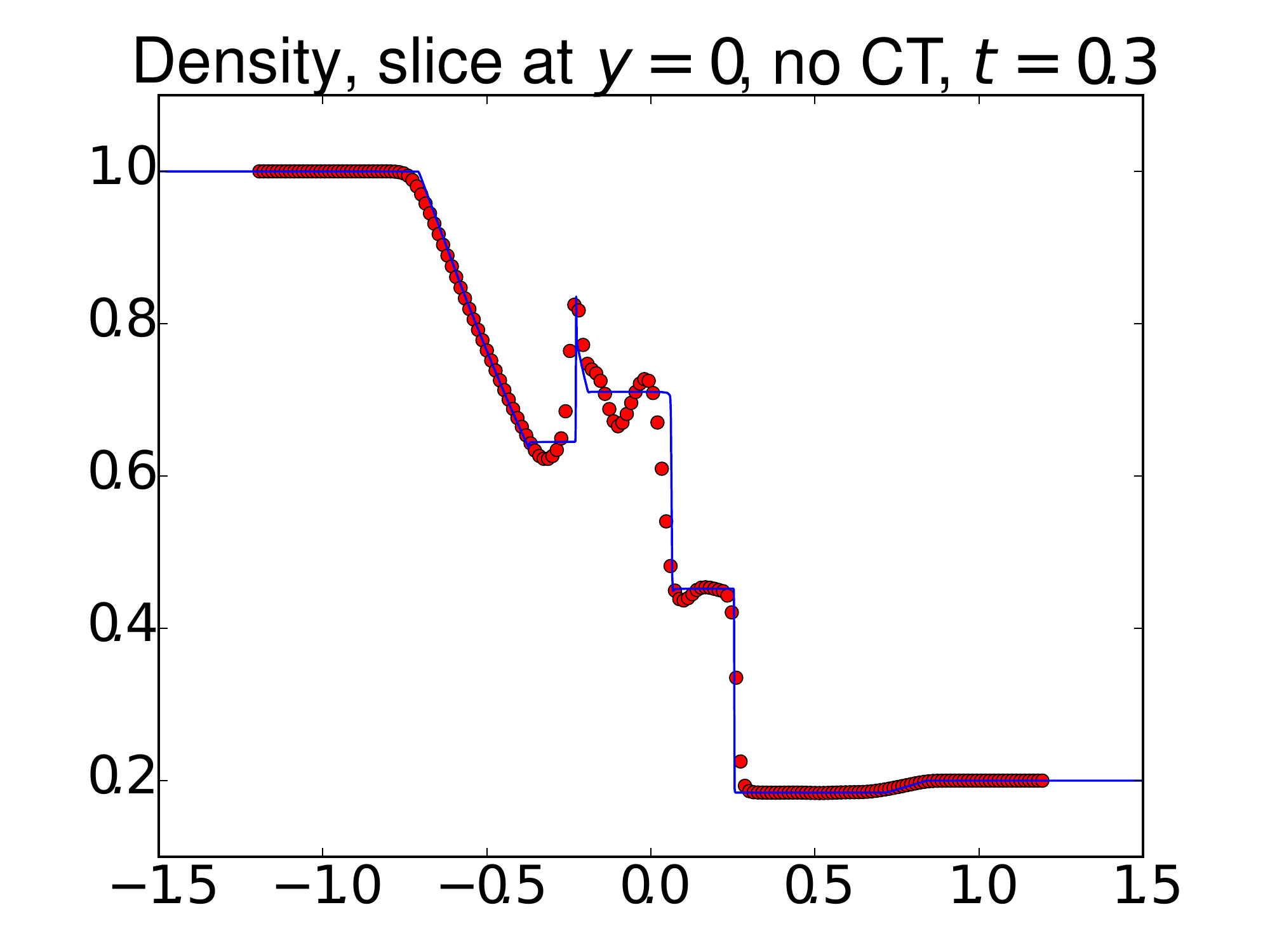} \\
    \end{tabular}
\caption{The rotated shock tube problem.  Here, we compare the solver
with (left panels) and without (right panels) constrained transport turned on.
A uniform mesh of size $180\times 150$ is used for both simulations.
The angle of rotation for the initial conditions is $\phi = \tan^{-1}(0.5)$,
and 30 equally spaced contours \newtext{ranging from the minumum to the maximum of each function} are used for the top two panels.
The contour plot for the solution without CT contains small wiggles that are
much more pronounced when slices of the solution are sampled.  To this end,
a slice of the solution along $y=0$ is presented in the bottom two panels.
Further evolution produces a solution that causes the code to fail in the case where CT is turned off.
The positivity-preserving limiter is turned off in order to exercise the code.
The solid lines in the bottom images are reference solutions that are computed
by solving the equivalent 1D shock problem on a uniform mesh with 50,000
points with the fifth-order finite difference WENO method.
\label{fig:2DShockTubeDensity}}
\end{center}
\end{figure}

\begin{figure}
\begin{center}
    \begin{tabular}{cc}
        \includegraphics[width=0.4\textwidth]{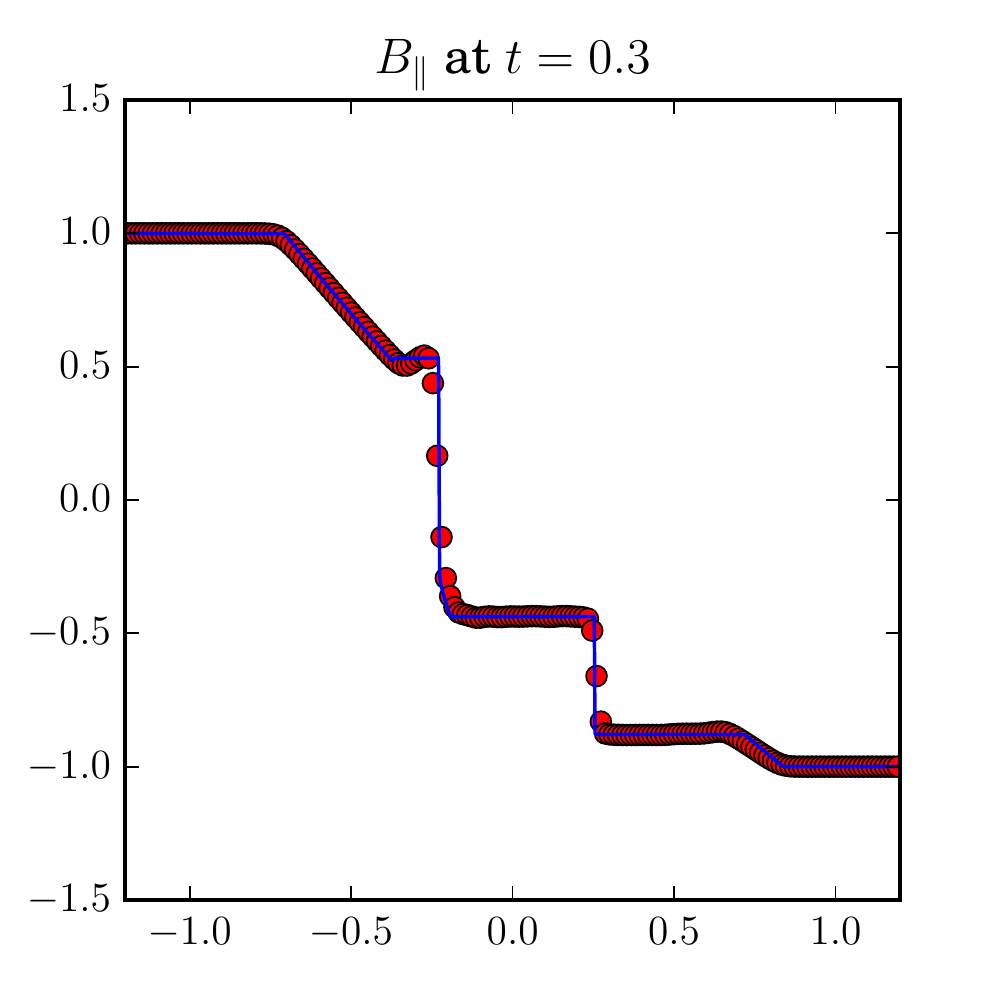} &
        \includegraphics[width=0.4\textwidth]{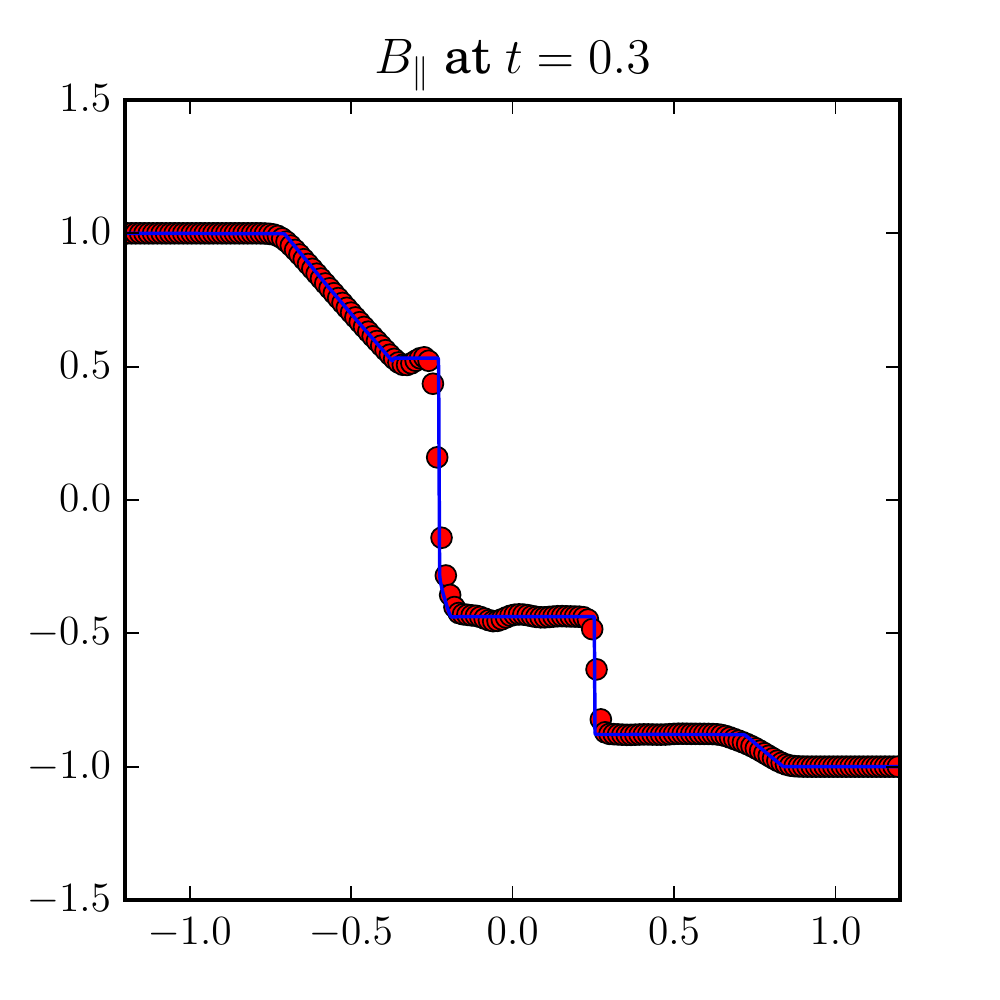} \\
        \includegraphics[width=0.4\textwidth]{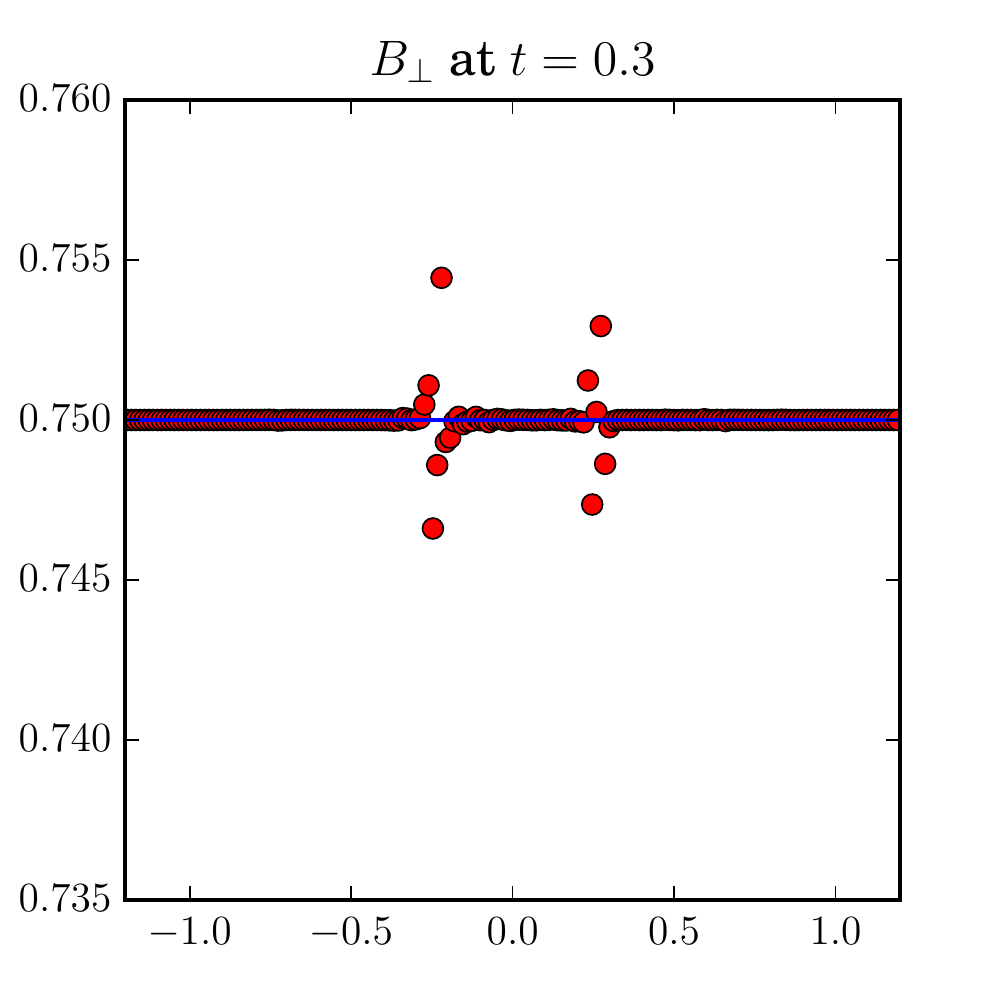} &
        \includegraphics[width=0.4\textwidth]{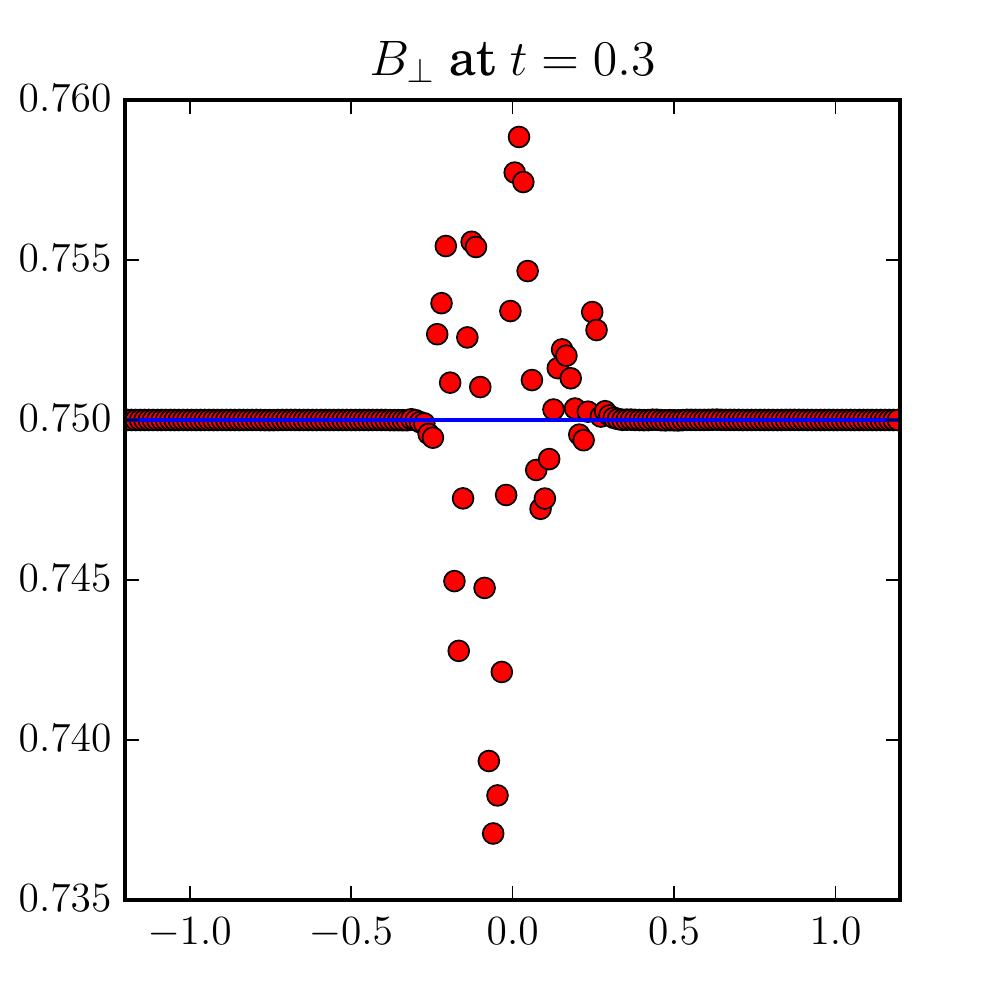} \\
        With Constrained Transport & No Constrained Transport
    \end{tabular}
\caption{2D rotated shock tube problem. 
The left panels have constrained transport turned on, and the right panels have constrained transport turned off.
Components of the magnetic field at $t = 0.3$ along the slice $y=0$.  The
mesh size is $180\times 150$.
The positivity-preserving limiter is turned off in order to exercise the code.
Each solid line is the reference solution described in Fig.  \ref{fig:2DShockTubeDensity}.
We observe that constrained transport (with the Hamilton-Jacobi solver)
allows us to numerically compute magnetic fields with far fewer
oscillations than would otherwise be obtainable.
    \label{fig:2DShockTubeBSlice}}
\end{center}
\end{figure}

In Figure~\ref{fig:2DShockTubeDensity} we present results for the density of
solutions computed using PIF-WENO with and without constrained-transport.  We
note that the contour plot of the solution obtained without constrained
transport does not exhibit the unphysical \newtext{wiggles} 
as is the case
in Fig 2(b) of~\cite{Christlieb2014}.  However, as can be seen from the plots
of the slice at $y=0$, 
unphysical oscillations appear in the PIF-WENO scheme that has constrained
transport turned off.  It is also clear from these plots that the constrained
transport method we propose in the current work is able to suppress the
unphysical oscillations satisfactorily.  As a side note, we find it helps to
use a global, as opposed to a local value for $\alpha$ in the Lax-Friedrichs
flux splitting for the high-order WENO reconstruction in order to further
reduce undesirable spurious oscillations.

\subsection{2D Orszag-Tang vortex}

In this section, we investigate the Orszag-Tang vortex problem.  A notable
feature of this problem is that shocks and cortices emerge from smooth initial
conditions as time evolves.  This is a standard test problem for numerical
schemes for MHD
equations~\cite{Christlieb2014,Dai1998a,Rossmanith2006,Toth2000, Zachary1994}.
The initial conditions are
\begin{equation}
    \rho = \gamma^2, \quad {\bf u} = \left( -\sin(y), \sin(x), 0 \right),
    \quad {\bf B} = \left( 
    -\sin(y), \sin(2x), 0 \right), \quad 
    p = \gamma,
\end{equation}
with an initial magnetic potential of 
\begin{equation}
    A^z(0, x, y) = 0.5\cos(2x) + \cos(y).
\end{equation}
The computational domain is $[0, 2\pi] \times [0, 2\pi]$, with double-periodic boundary conditions.

\begin{figure}
\begin{center}
    \begin{tabular}{cc}
        \includegraphics[width=0.4\textwidth]{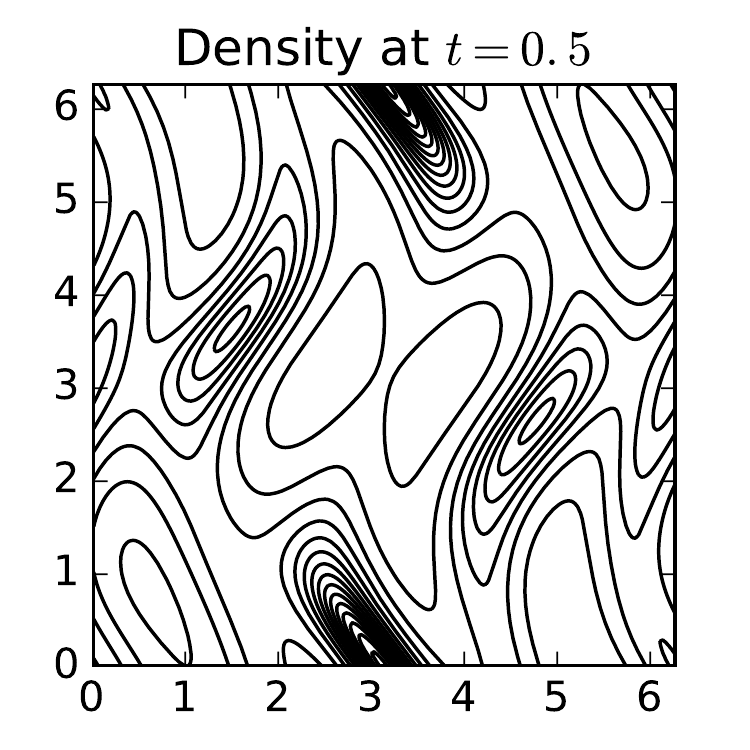} &
        \includegraphics[width=0.4\textwidth]{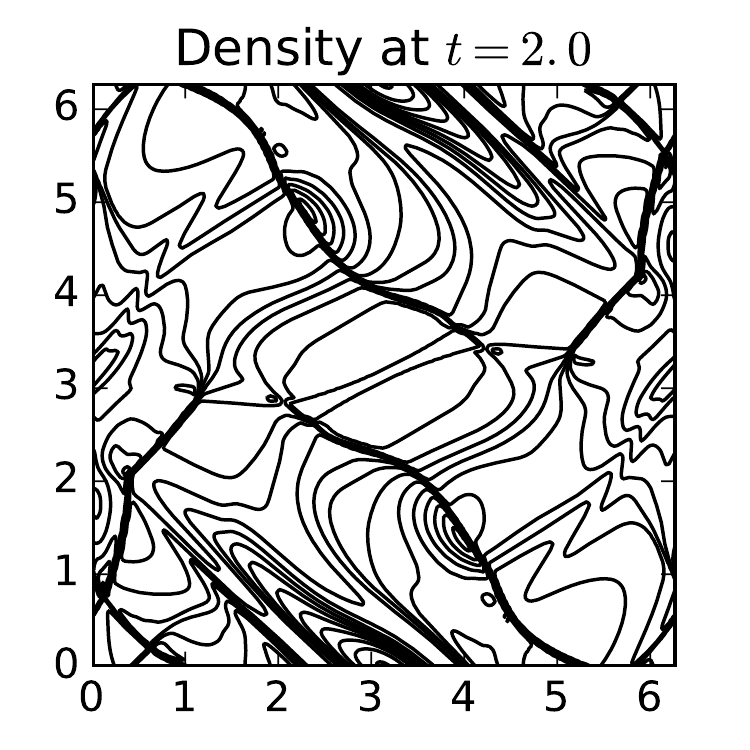} \\
        (a) & (b) \\
        \includegraphics[width=0.4\textwidth]{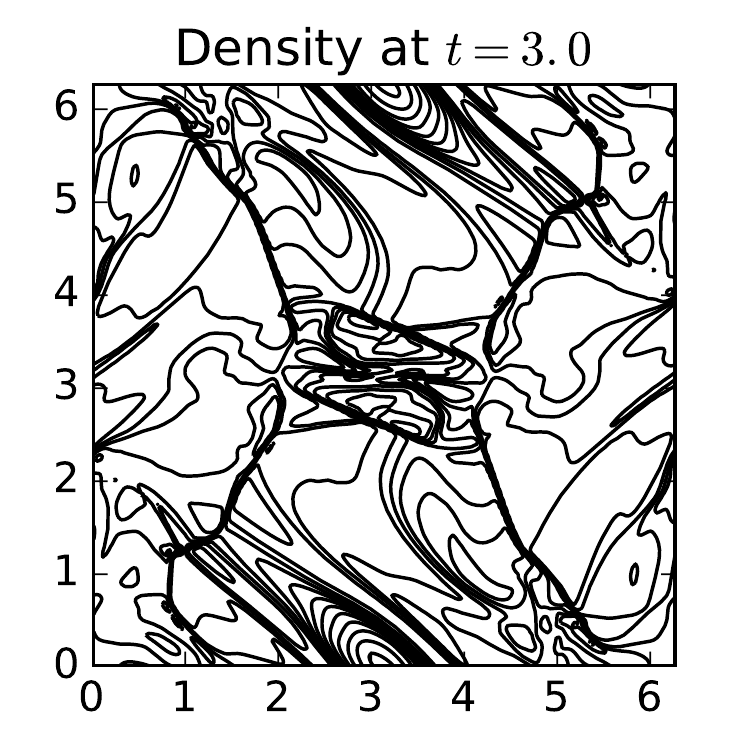} &
        \includegraphics[width=0.4\textwidth]{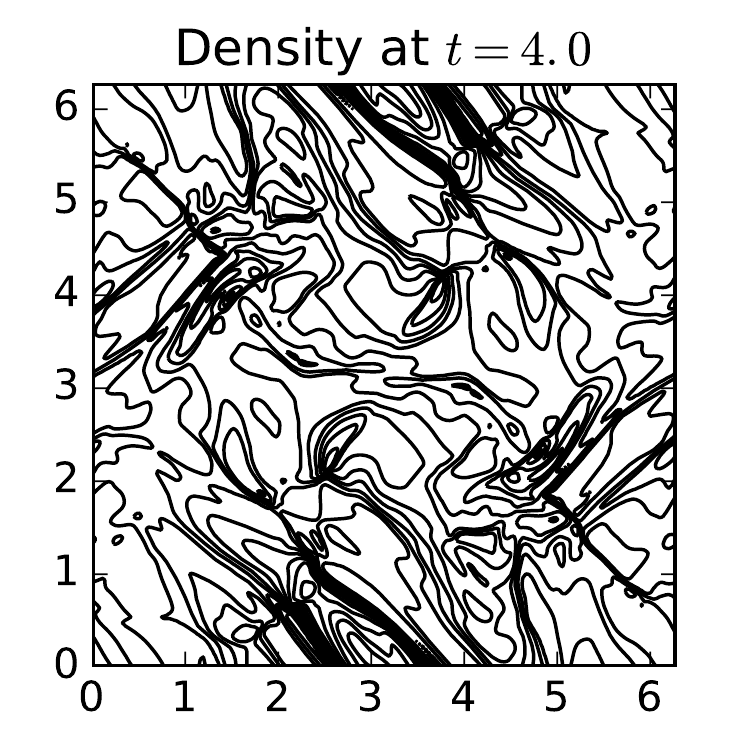} \\
        (c) & (d)
    \end{tabular}
    \caption{The Orszag-Tang problem. We show density contour plots at (a) $t = 0.5$, (b) $t=2$, (c) $t=3$, (d) $t=4$.  The solution is computed using PIF-WENO with constrained transport on and positivity-preserving limiter \newtext{on}, on a $192\times 192$ mesh.  A total of 15 equally spaced contours are used for each graph. \label{fig:2DOrszagTang}}
\end{center}
\end{figure}

We present in Figure~\ref{fig:2DOrszagTang} the density contour plots at
$t=0.5$, $t=2$, $t=3$, and $t=4$, computed by PIF-WENO with constrained
transport on and positivity-preserving limiter \newtext{on}.  

Control of divergence error of the magnetic field is critical for this test
problem.  If we turn off the constrained transport, the simulation crashes at
$t = 1.67$.  In Figure~\ref{fig:2DOrszagTangNoCTvsNoPPvsPP} we present plots
of the density at time $t = 1.5$ obtained \newtext{with different
configurations of the numerical schemes}.  Note the development of the
nonphysical features in the solution obtained without constrained transport.

\newtext{We also note that in this problem, where the positivity-preserving
limiter is not needed for the simulation, the limiter leads to little
difference in the solution.   We present in {Figures
\ref{fig:2DOrszagTangNoCTvsNoPPvsPP} and \ref{fig:2DOrszagTangPPvsNoPP}} plots
that demonstrate this.}
\newtext{The plots of the type in {Figures \ref{fig:2DOrszagTang} and
\ref{fig:2DOrszagTangPPvsNoPP}} are presented in many sources.  Our plots
agree with results from the literature~\cite{Christlieb2014,Dai1998a,Rossmanith2006,Toth2000,Zachary1994}.
}

\begin{figure}
\begin{center}
    \begin{tabular}{ccc}
        \includegraphics[width=0.32\textwidth]{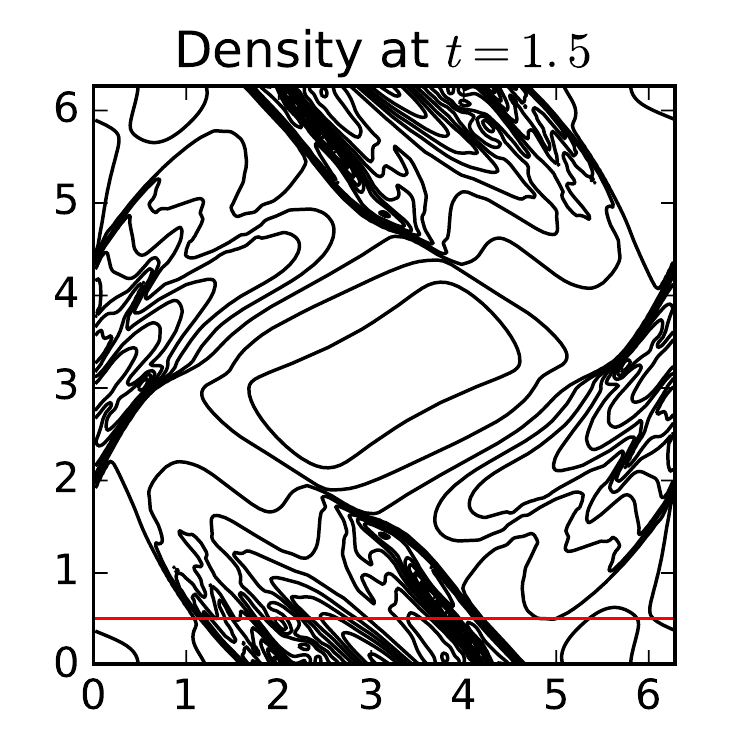} &
        \includegraphics[width=0.32\textwidth]{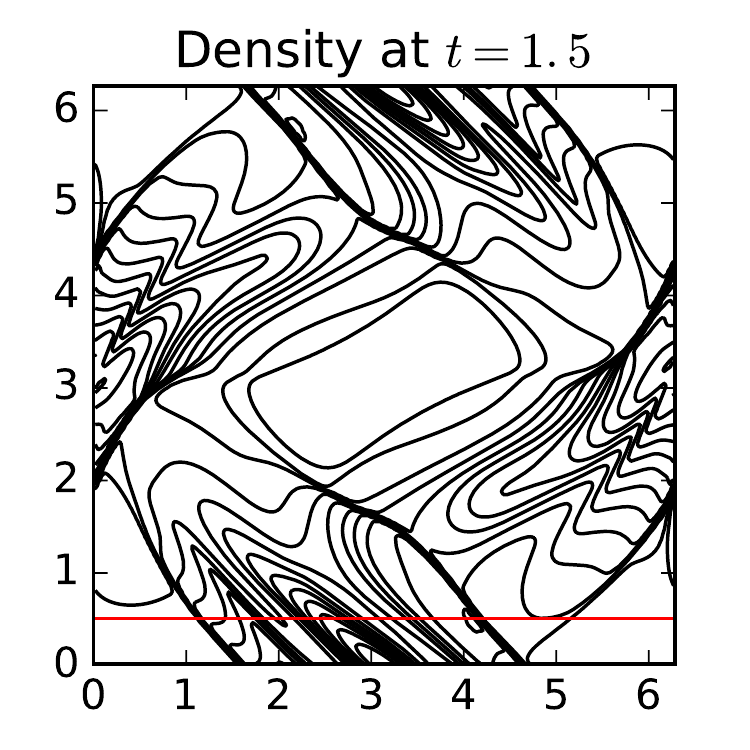} &
        \includegraphics[width=0.32\textwidth]{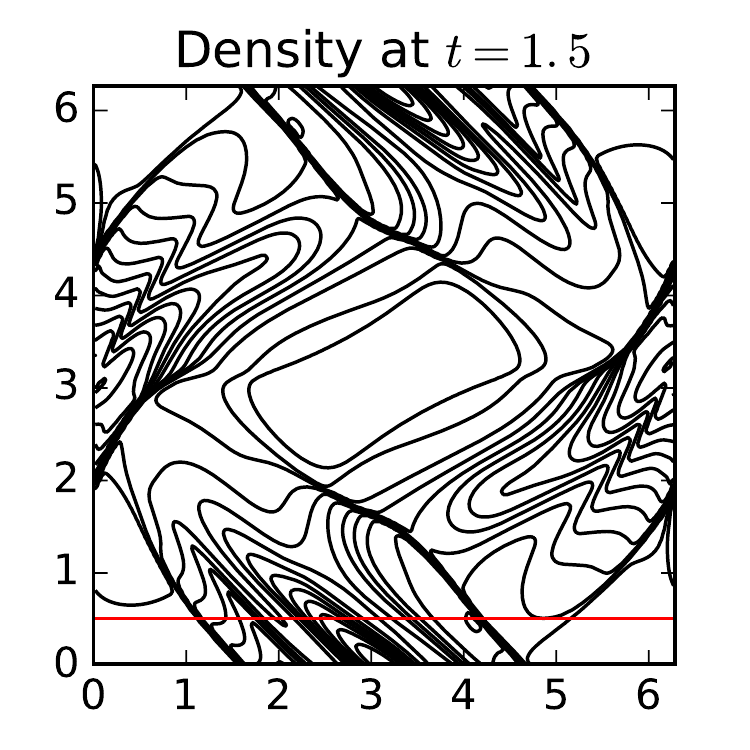} \\
        (a) & (b) & (c)\\
        \multicolumn{3}{c}{\includegraphics[width=0.9\textwidth]{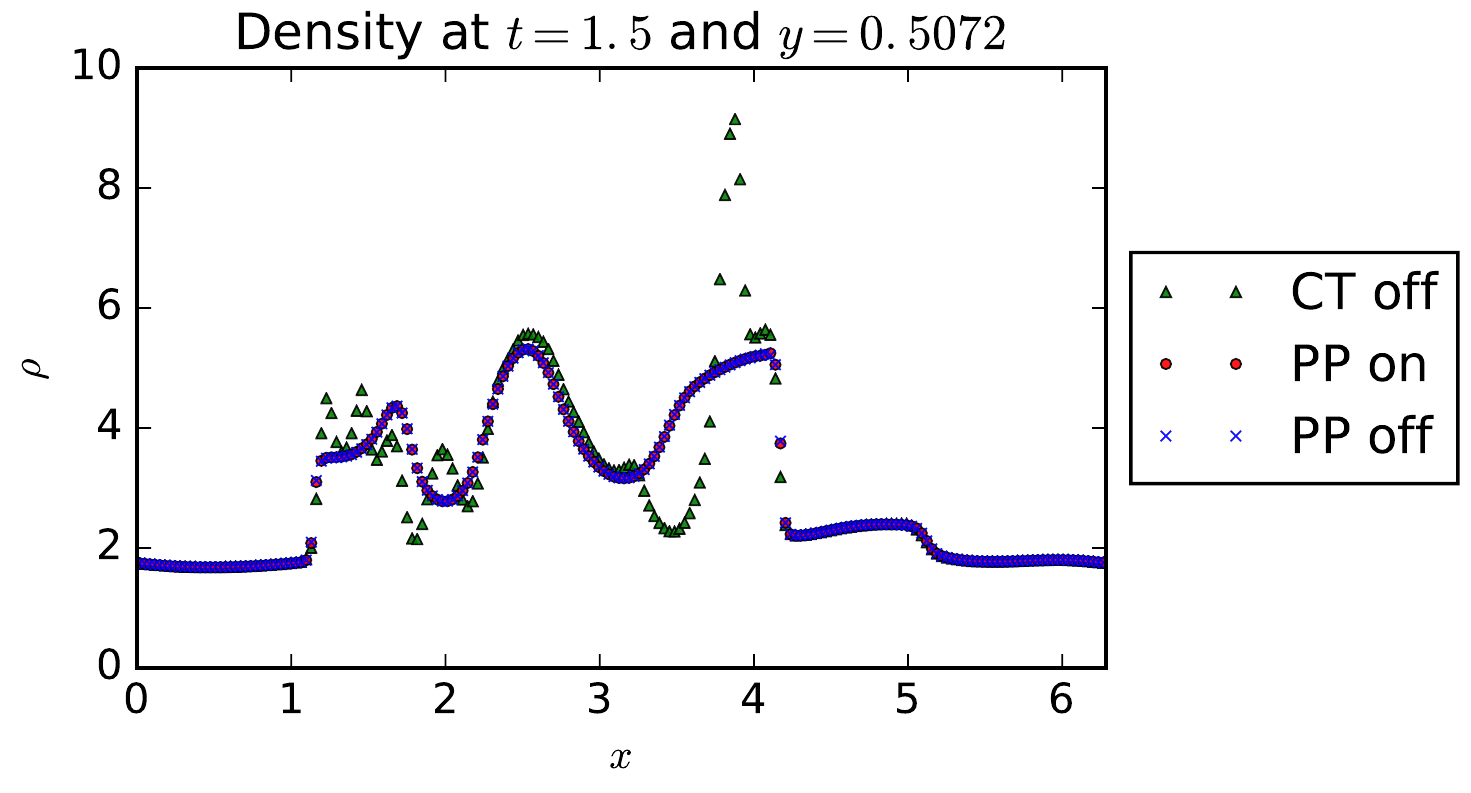}} \\
        \multicolumn{3}{c}{(d)}        
    \end{tabular}
    \caption{\newtext{The Orszag-Tang problem. Density plots at $t=1.5$ for solutions computed using different configurations in the numerical scheme.  (a) PIF-WENO, with constrained transport turned off; (b) PIF-WENO with constrained transport and positivity preserving limiter turned on; (c) PIF-WENO with constrained transport turned on and positivity preserving limiter turned off; (d) The slice along $y = 0.5072$.  The solutions are computed with a $192\times 192$ mesh.  A total of 15 equally spaced contours are used for each of (a), (b), and (c).} \label{fig:2DOrszagTangNoCTvsNoPPvsPP}}
\end{center}
\end{figure}

\begin{figure}
\begin{center}
    \begin{tabular}{cc}
        \includegraphics[width=0.45\textwidth]{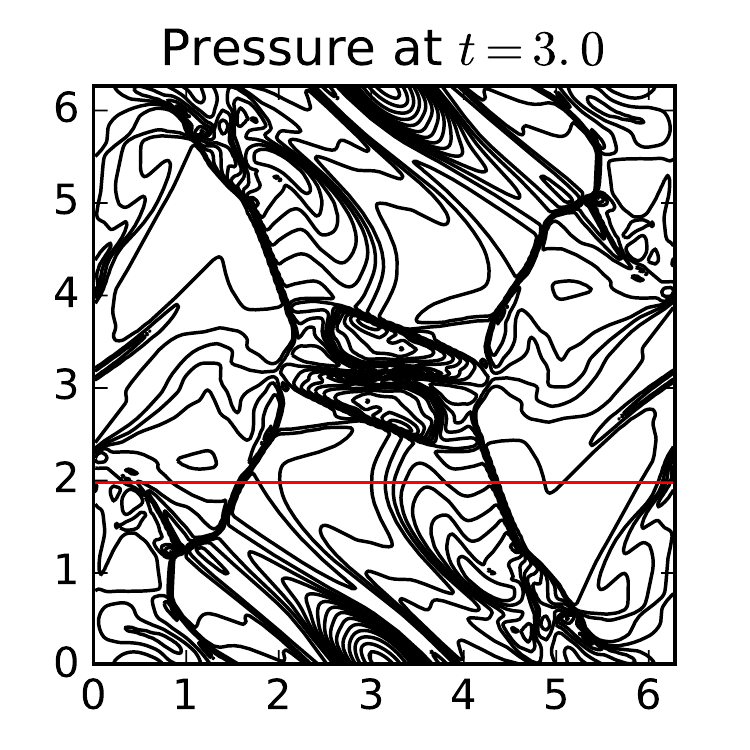} &
        \includegraphics[width=0.45\textwidth]{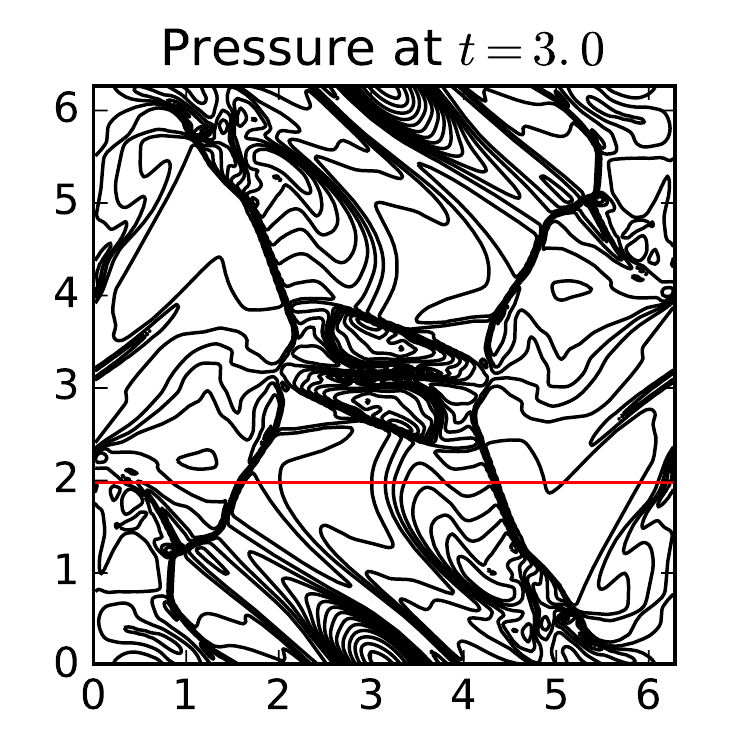} \\
        (a) & (b) \\
        \multicolumn{2}{c}{\includegraphics[width=0.9\textwidth]{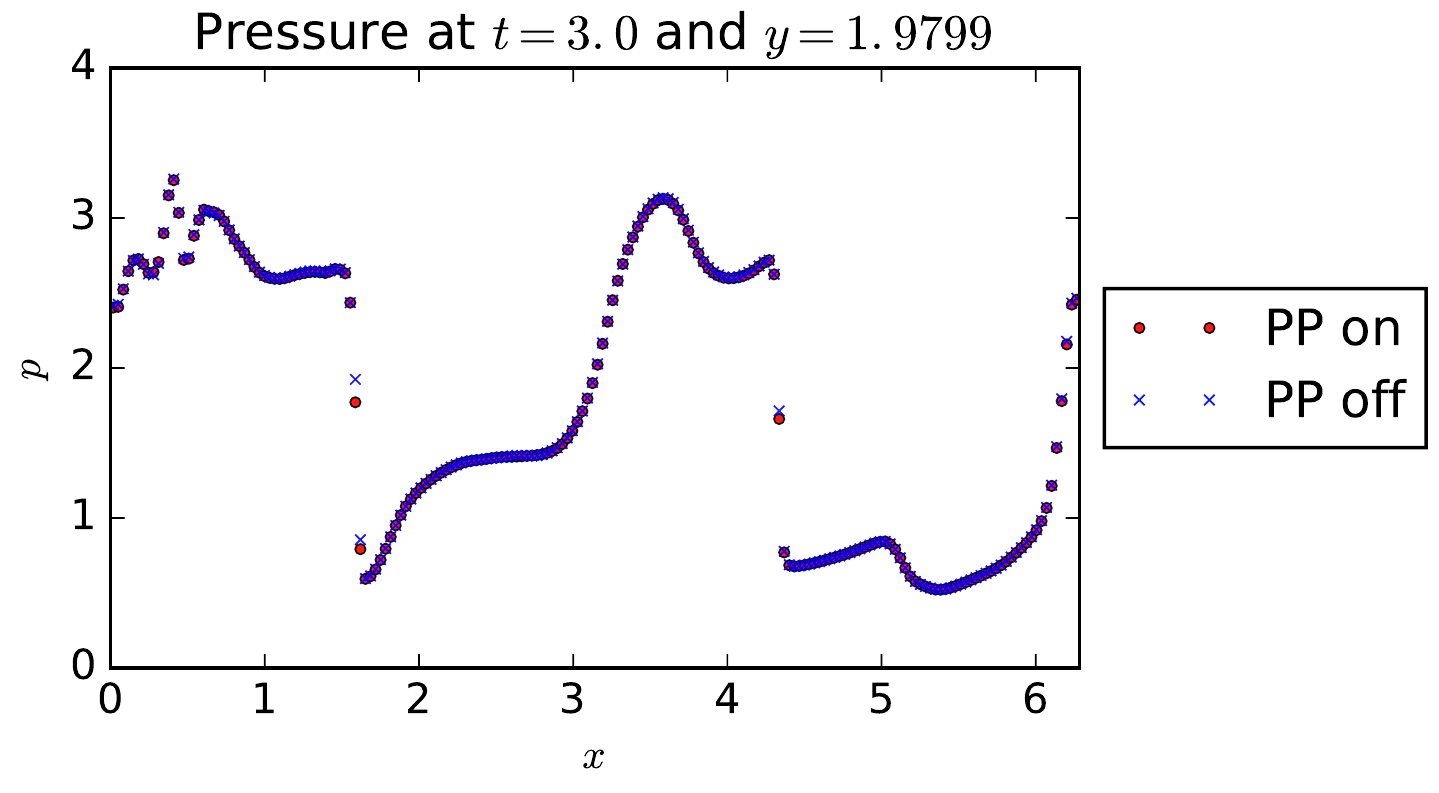}} \\
        \multicolumn{2}{c}{(c)}        
    \end{tabular}
    \caption{\newtext{The Orszag-Tang problem. Pressure plots at $t=3.0$ for solutions computed using different configurations in the numerical scheme.  (a) PIF-WENO with constrained transport turned on and positivity preserving limiter turned off; (b) PIF-WENO with constrained transport and positivity preserving limiter turned on;  (c) The slice along $y = 1.9799$.  The solutions are computed with a $192\times 192$ mesh.  A total of 15 equally spaced contours are used for each of (a) and (b).} \label{fig:2DOrszagTangPPvsNoPP}}
\end{center}
\end{figure}

\subsection{2D rotor problem}

This is a two-dimensional test problem that involves low pressure values
\cite{Balsara1999a}.  
The setup we use here is the same as
the very low $\beta$ version found in~\cite{Waagan2009a}.  These initial
conditions are
\begin{gather}
	\rho = 
		\begin{cases}
			10 & \text{if $r \leq 0.1$,} \\
			1 + 9 \tilde{{f}}(r) & \text{if $r\in (0.1, 0.115)$,} \\
			1 & \text{if $r\geq 0.115$,}
		\end{cases} \\
	u^x = 
		\begin{cases}
			-10y + 5 & \text{if $r \leq 0.1$,} \\
			(-10y + 5)\tilde{f}(r) & \text{if $r \in (0.1, 0.115)$,} \\
			0 & \text{if $r \geq 0.115$,}
		\end{cases} \\
	u^y =
		\begin{cases}
			10x - 5 & \text{if $r \leq 0.1$,} \\
			(10x - 5)\tilde{f}(r) & \text{if $r \in (0.1, 0.115)$,} \\
			0 & \text{if $r \geq 0.115$,}
		\end{cases} \\
	u^z = 0, \quad B^x = \frac{2.5}{\sqrt{4\pi}}, \quad B^y = 0,
	\quad B^z = 0, \quad p = 10^{-8}, \qquad A^z = \frac{2.5}{\sqrt{4\pi}} y,
\end{gather}
where
\begin{equation}
	r = \sqrt{(x-0.5)^2 + (y - 0.5)^2}, \quad \tilde{{f}}(r) = \frac{1}{3}(23 - 200r).
\end{equation}

The computation domain we use is $[0, 1]\times [0, 1]$, \newtext{with}
zeroth order extrapolation on the conserved quantities and
first order extrapolation on the magnetic potential as the boundary conditions
on all four sides 
\newtext{(i.e., conserved quantities at the ghost points are set equal to the
last interior point, and values for the magnetic potential are defined through repeated
extrapolation of two point stencils).}

We compute the solution \newtext{to a final time of} $t = 0.27$ using a $400\times
400$ mesh and present the plots of the magnetic pressure and the magnetic
field line in Figure~\ref{fig:2DRotor}.  The magnetic field line plot
presented here is the contour plot of $A^z$.  A total of $50$ levels are used
in this contour plot. Our result is consistent with the one presented
in~\cite{Waagan2009a}.   We note that the positivity-preserving limiter is
necessary to complete this test problem, because otherwise the code fails.

\begin{figure}
	\begin{center}
		\begin{tabular}{cc}
			\includegraphics[width=0.45\textwidth]{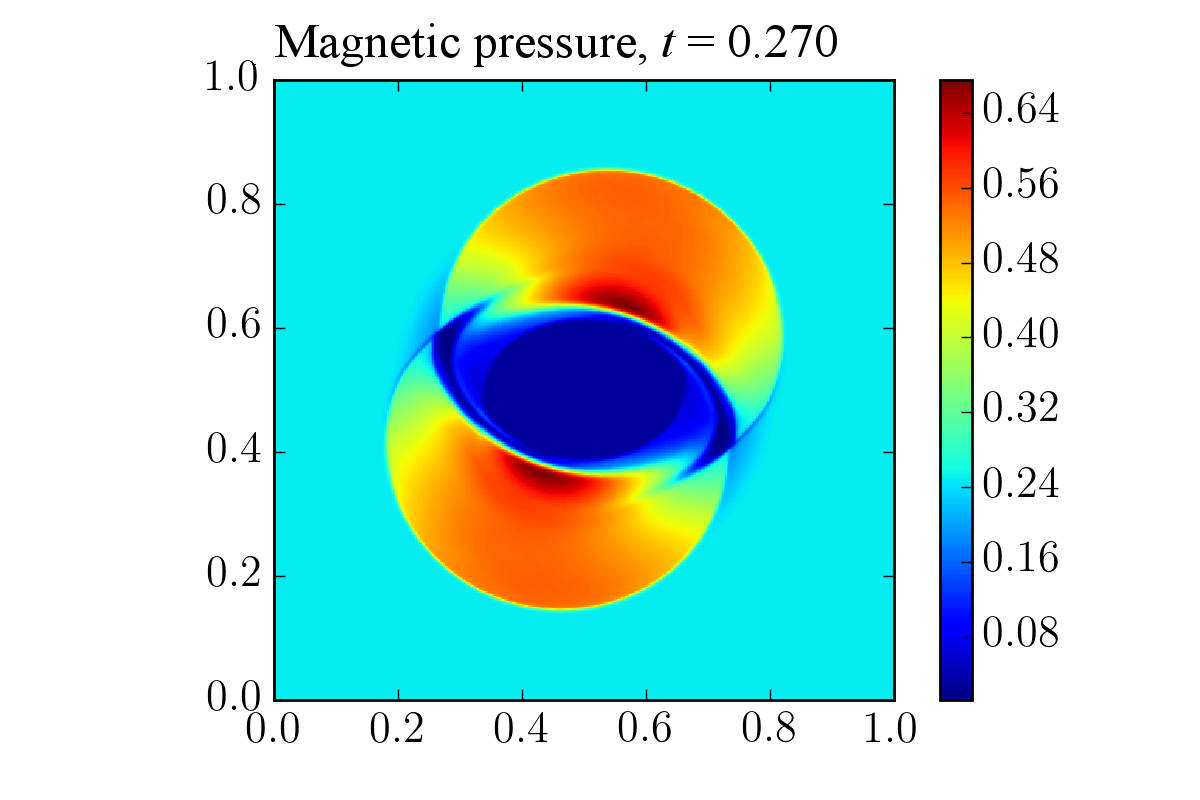} &
			\includegraphics[width=0.45\textwidth]{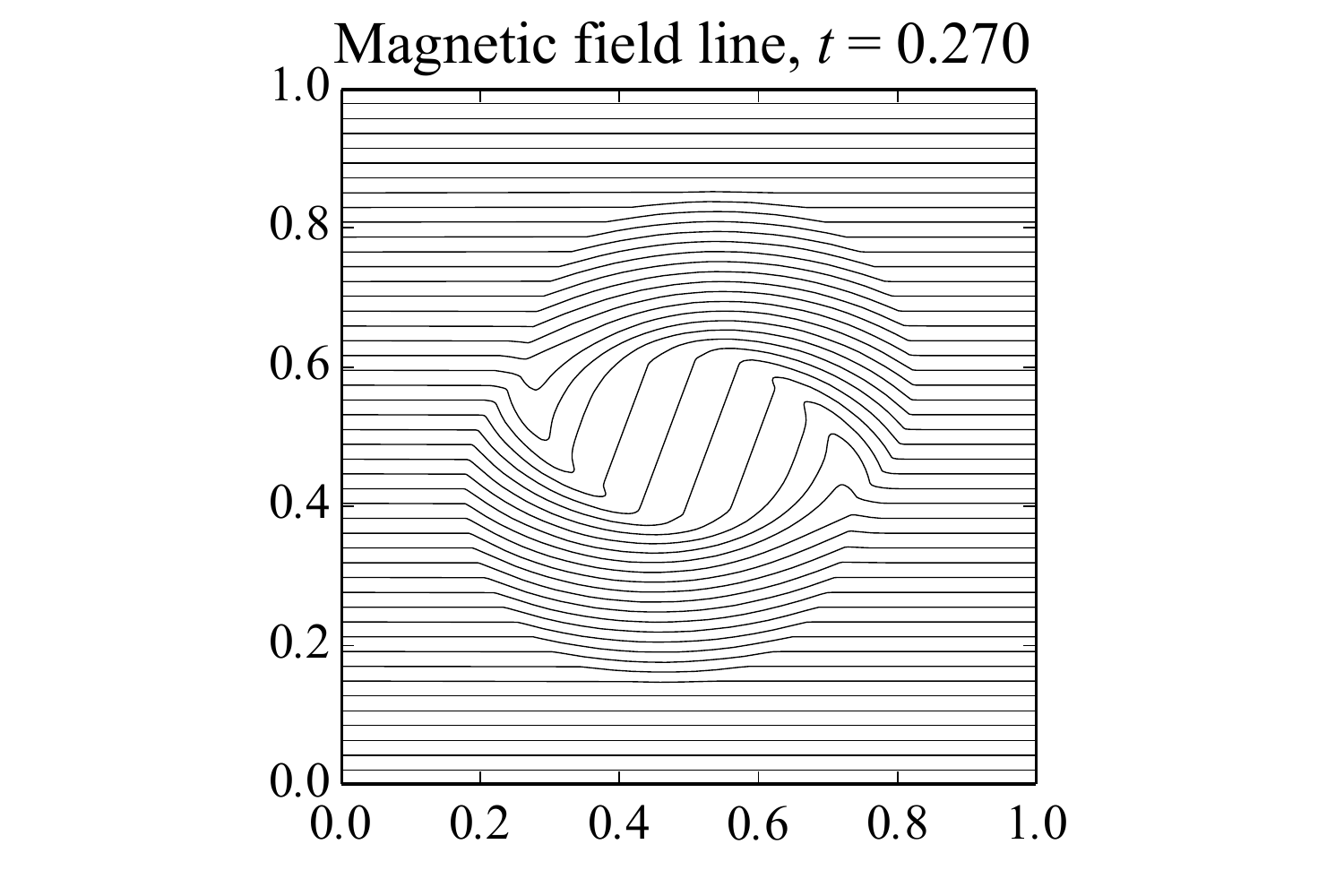} \\
			(a) & (b)
		\end{tabular} 
		\caption{The 2D rotor problem.  Shown here are (a) the pseudocolor plot of the magnetic pressure and (b) the magnetic field line.  This solution is computed on a mesh of size $400 \times 400$.  Constrained transport and positivity-preserving limiters are turned on. \label{fig:2DRotor}}
	\end{center}
\end{figure}

\subsection{Cloud-shock interaction}

The cloud-shock interaction problem is a standard test problem for MHD 
\cite{Christlieb2014,Dai1998a,Helzel2011,Helzel2013,Rossmanith2006}.  The initial conditions are
\begin{align}
    &(\rho, u^x, u^y, u^z, p, B^x, B^y, B^z) \\ \nonumber
    &\quad =
        \begin{cases}
            (3.86859, 11.2536, 0, 0, 167.345, 0, 2.1826182, -2.1826182) &
                \text{if $x < 0.05$}, \\
            (10, 0, 0, 0, 1, 0, 0.56418958, 0.56418958) &
                \text{if $x > 0.05$ and $r < 0.15$,} \\
            (1, 0, 0, 0, 1, 0, 0.56418958, 0.56418958) &
                \text{otherwise,}
        \end{cases}
\end{align}
where $r = \sqrt{(x-0.25)^2 + (y-0.5)^2}$ in 2D, and $r = \sqrt{(x-0.25)^2 + (y-0.5)^2 + (z-0.5)^2}$ in 3D denotes the distance to the center of the stationary cloud.
For both problems, we use the initial magnetic potential
\begin{equation}
    A^x = 0, \qquad
    A^y = 0, \qquad
    A^z =
    \begin{cases}
        -2.1826182 x + 0.080921431 &
            \text{if $x \leq 0.05$}, \\
        -0.56418958 x &
            \text{if $x \geq 0.05$}.
    \end{cases}
\end{equation}
In the 2D case, we only keep track of $A^z$.

\subsubsection{Cloud-shock interaction\newtext{: The 2D problem}}

The computational domain we use is $[0, 1]\times [0, 1]$, with zeroth order
extrapolation on the conserved quantities and first order extrapolation on the
magnetic potential as the boundary conditions on all four sides
\newtext{(i.e., conserved quantities at the ghost points are set equal to the
last interior point, and values for the magnetic potential are defined through repeated
extrapolation of two point stencils).}

\begin{figure}
\begin{center}
    \begin{tabular}{ccc}
        \includegraphics[width=0.30\textwidth]{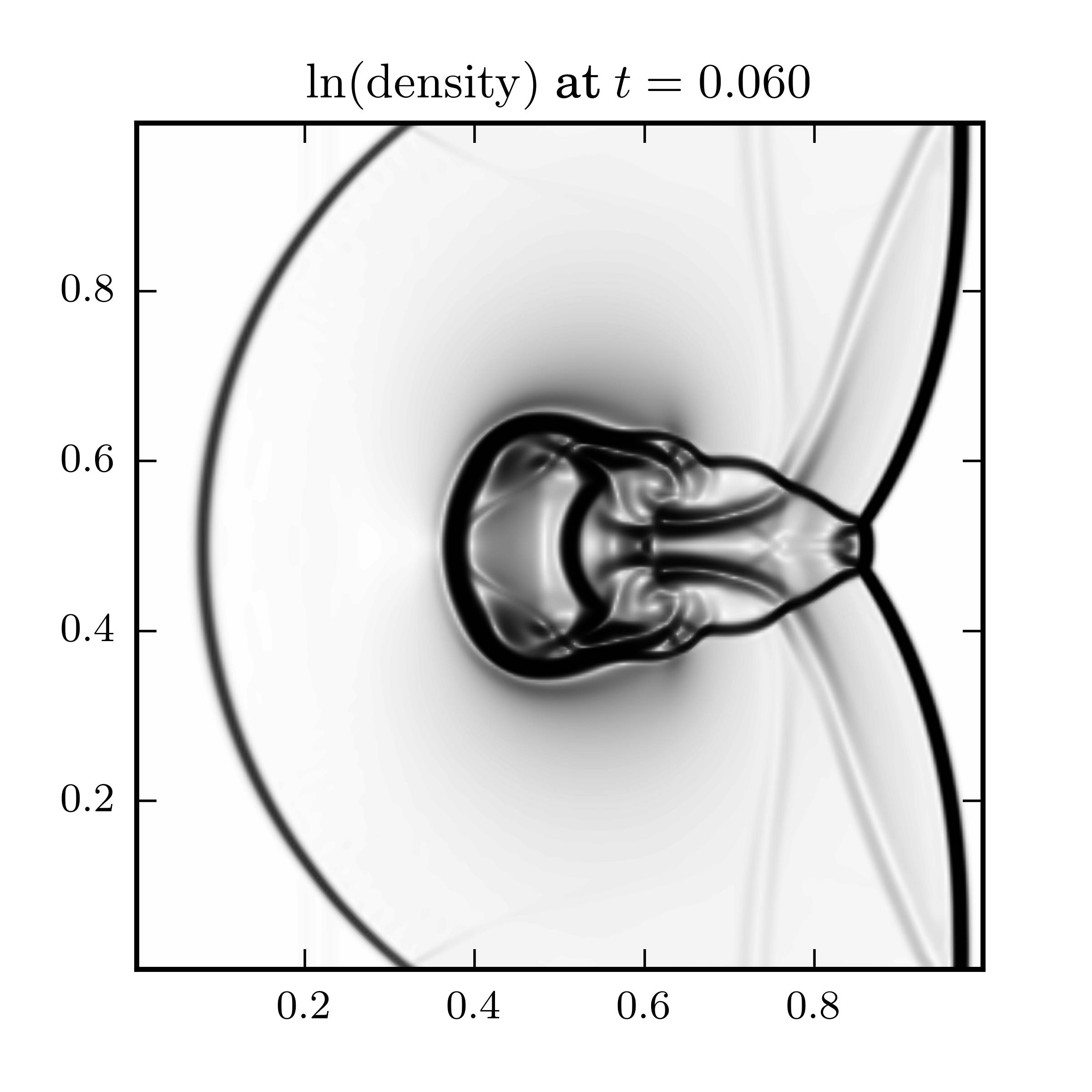} &
        \includegraphics[width=0.30\textwidth]{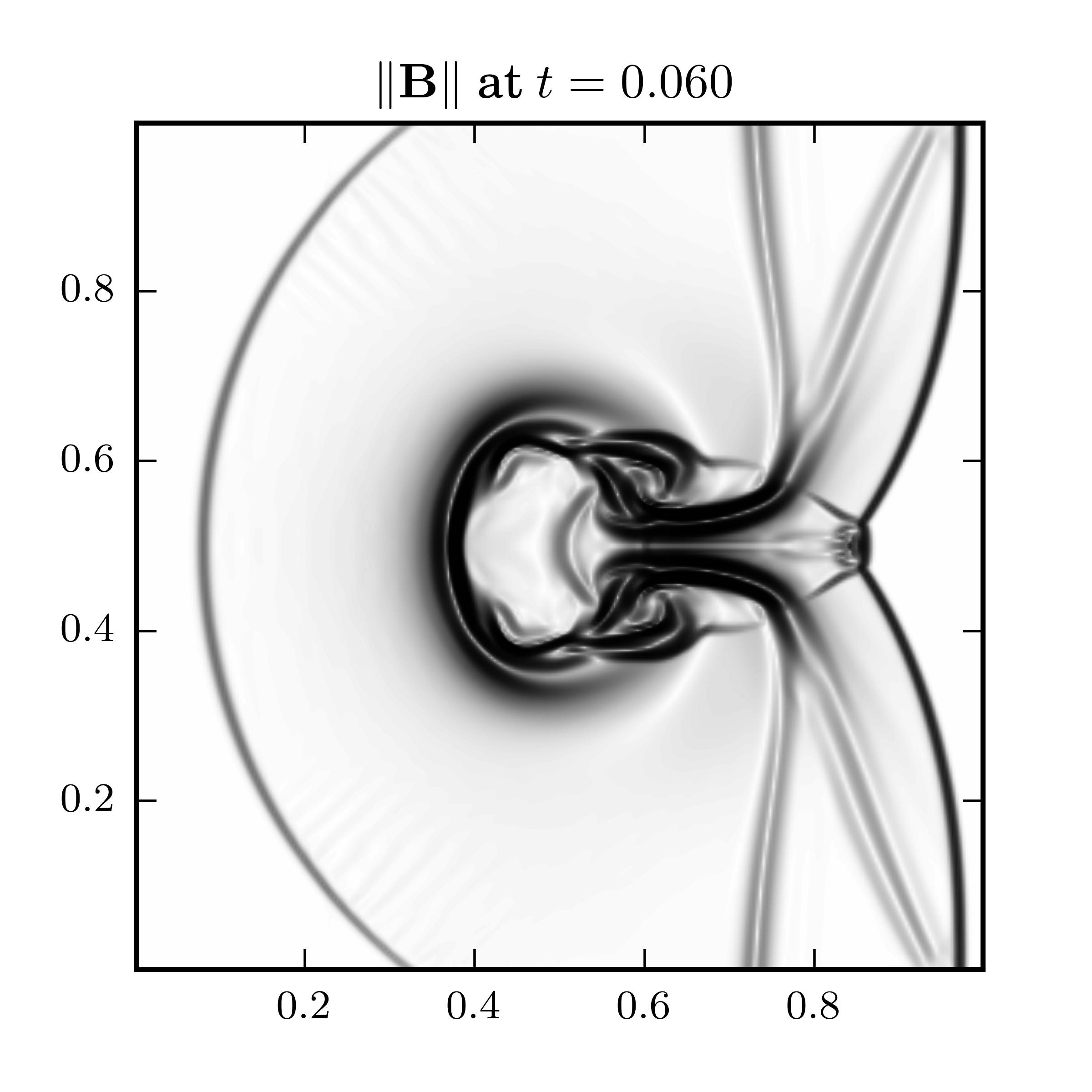} &
        \includegraphics[width=0.30\textwidth]{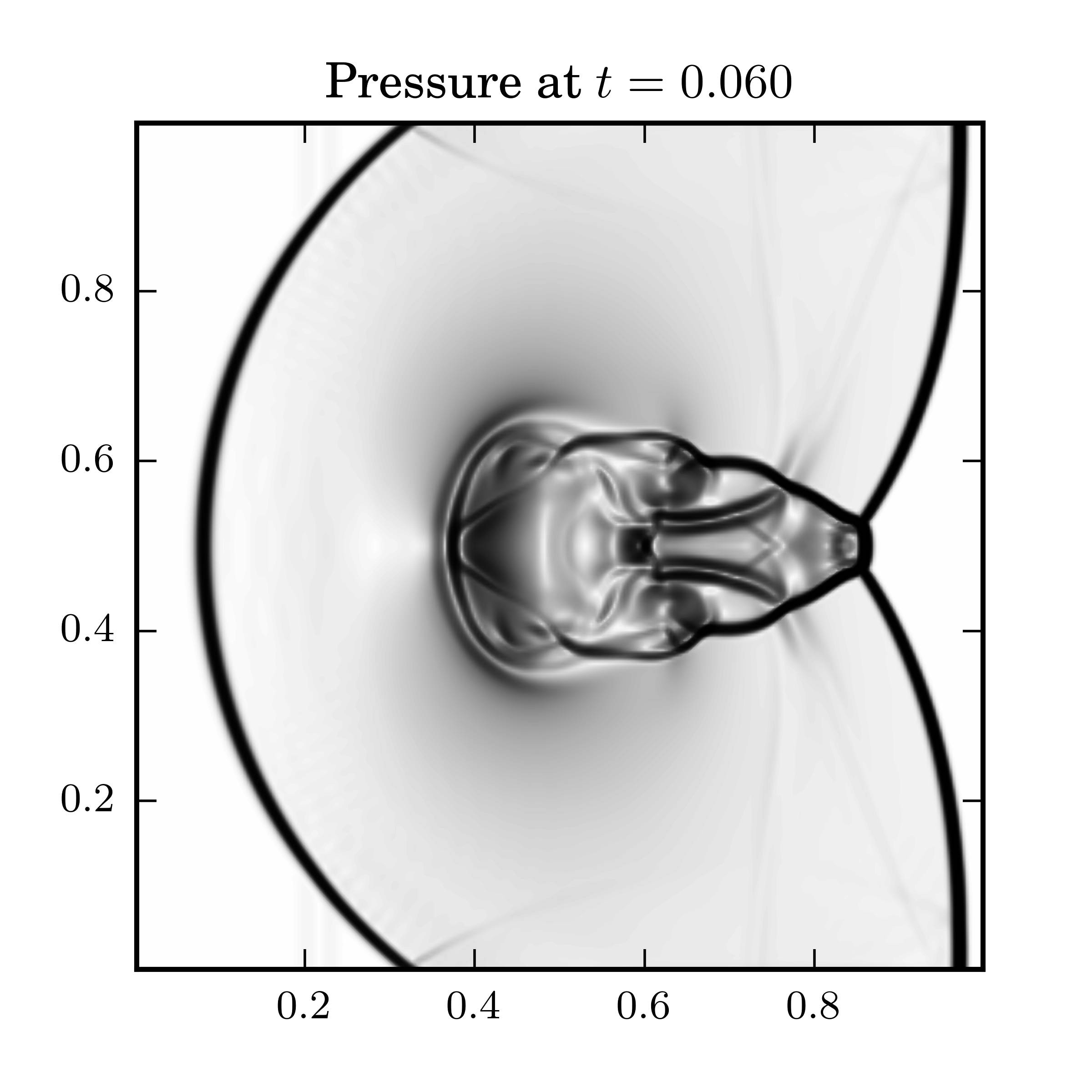} 
        \\
        (a) & (b) & (c) \\
        \includegraphics[width=0.30\textwidth]{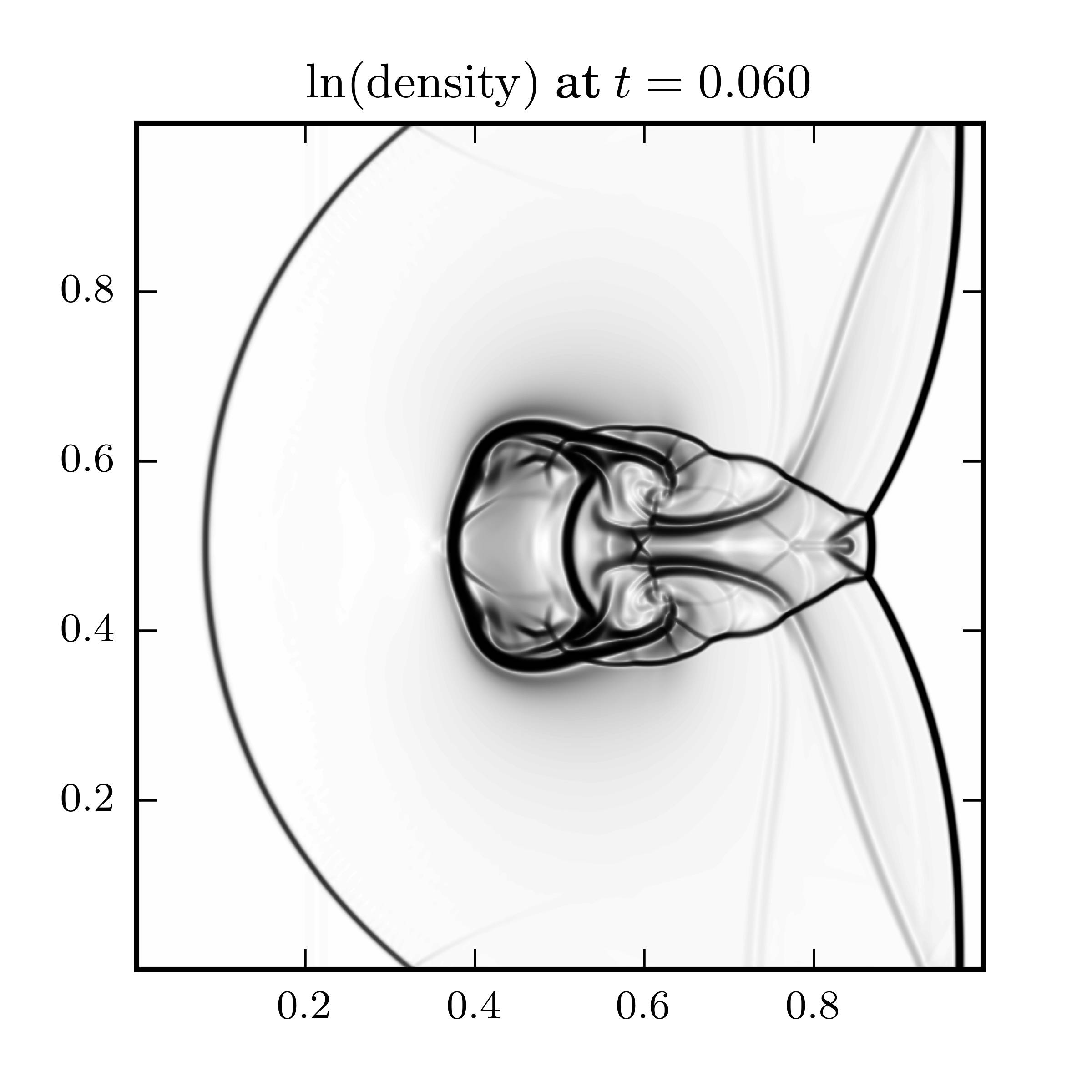} &
        \includegraphics[width=0.30\textwidth]{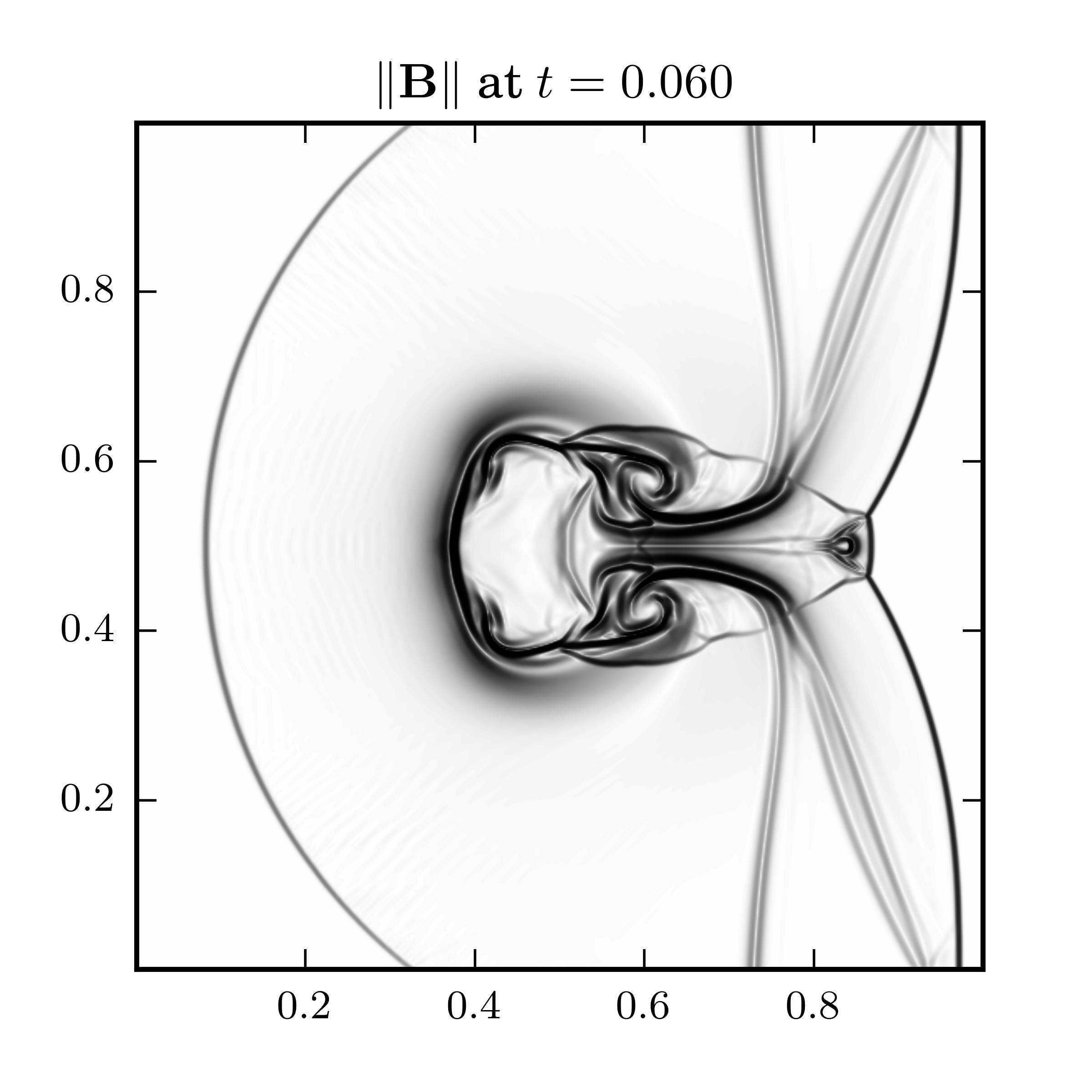} &
        \includegraphics[width=0.30\textwidth]{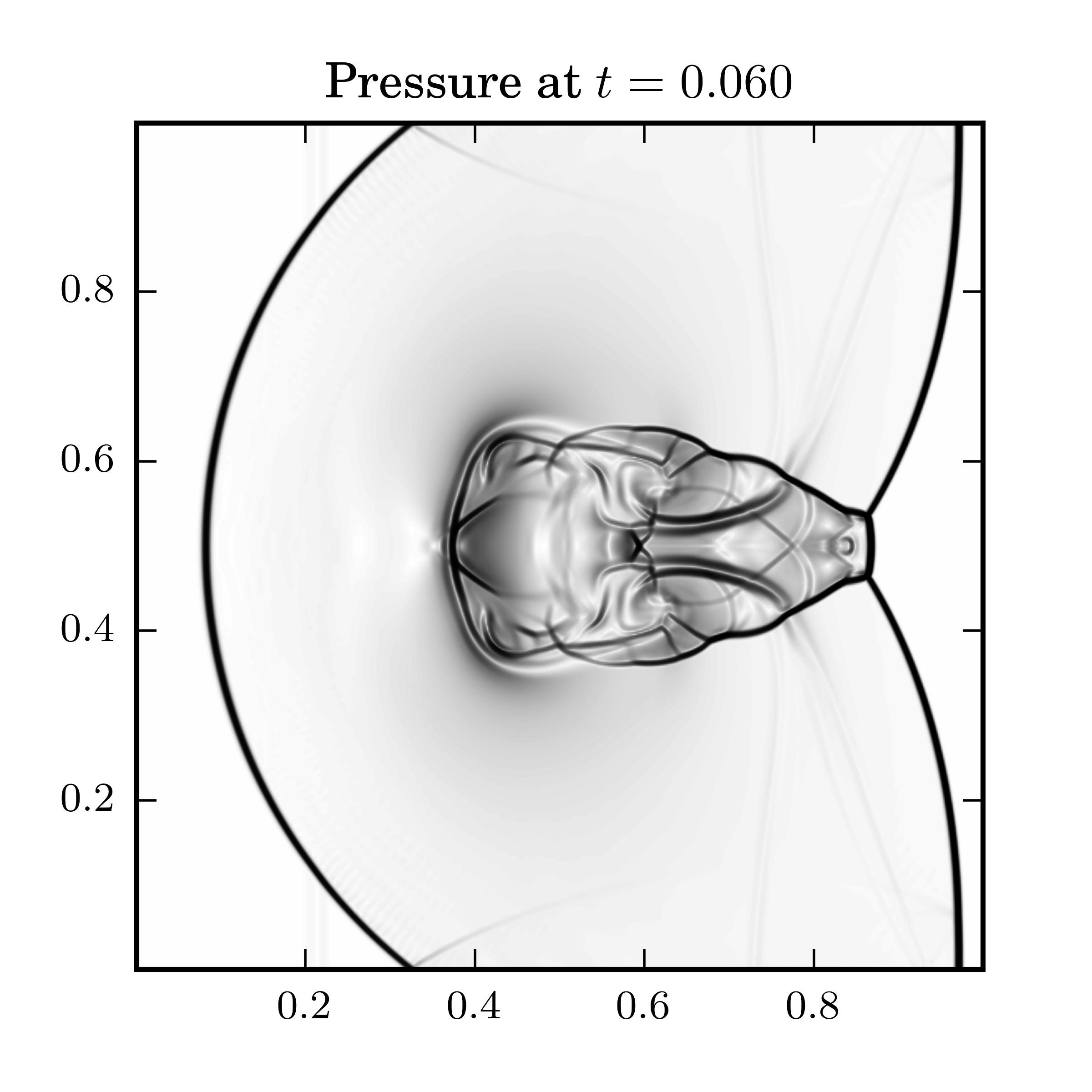}
        \\
        (d) & (e) & (f)
    \end{tabular} 
    \caption{\newtext{The 2D cloud-shock interaction problem.  Here,
    we run the solver to a final time of $t=0.06$.  In the first
    three panels, we show Schlieren plots for (a) the natural log of the
    density, (b) the norm of the magnetic field, and (c) the pressure
    for a mesh of size $256\times256$.  The same results for a mesh
    of size $512\times512$ are presented in panels (d)-(f), where we observe
    much higher resolution for the problem.  Constrained
    transport and positivity-preserving limiter are turned on for both
    simulations.
    \label{fig:2DCloudShock}}}
\end{center}
\end{figure}

We compute the solution at $t = 0.06$ using a $256\times 256$ mesh.  The Schlieren plots of $\ln \rho$ and of $\left|\Bvec\right|$ are presented in Figure~\ref{fig:2DCloudShock}.  We note here that the current scheme is able to capture the shock-wave-like structure near $x=0.75$.  This is consistent with our previous result in~\cite{Christlieb2014}, and an improvement over earlier results in~\cite{Dai1998a,Rossmanith2006}.  We also note that the positivity-preserving limiter is not required for this simulation.  Nonetheless, we present the result here to demonstrate the high resolution of our method, even when the limiter is turned on.

\subsubsection{Cloud-shock interaction\newtext{: The 3D problem}}

The computational domain \newtext{for this problem} is $[0, 1]\times [0, 1] \times [0, 1]$, with zeroth
order extrapolation on the conserved quantities and first order extrapolation
on the magnetic potential as the boundary conditions on all six faces
\newtext{(i.e., conserved quantities at the ghost points are set equal to the
last interior point, and values for the magnetic potential are defined through repeated
extrapolation of two point stencils).}

We compute the solution to a final time of time $t = 0.06$ on a
	$256\times 256 \times 256$ mesh.  In Figure~\ref{fig:3DCloudShock}, we show
the evolution of the density of the solution.
We have two remarks on this result.  The first is that the shock-wave-like
structure near $x = 0.75$ at the final time is also visible when we use a
$128\times 128 \times 128$ mesh.  The second is that the positivity-preserving
limiter is required to run this simulation with $256\times 256 \times 256$
mesh.  The reason is that extra structure that contains very low pressure
shows up in this mesh at time $t = 0.0378$.  This extra structure cannot be
observed on the coarser mesh with the method proposed in
the current work, nor do we observe it with our previous SSP-RK
solver~\cite{Christlieb2014}.  Therefore, it does
not cause trouble for simulations using a $128\times 128 \times 128$ mesh.
\begin{figure}
\begin{center}
\begin{tabular}{ccc}
    \includegraphics[width=0.30\textwidth]{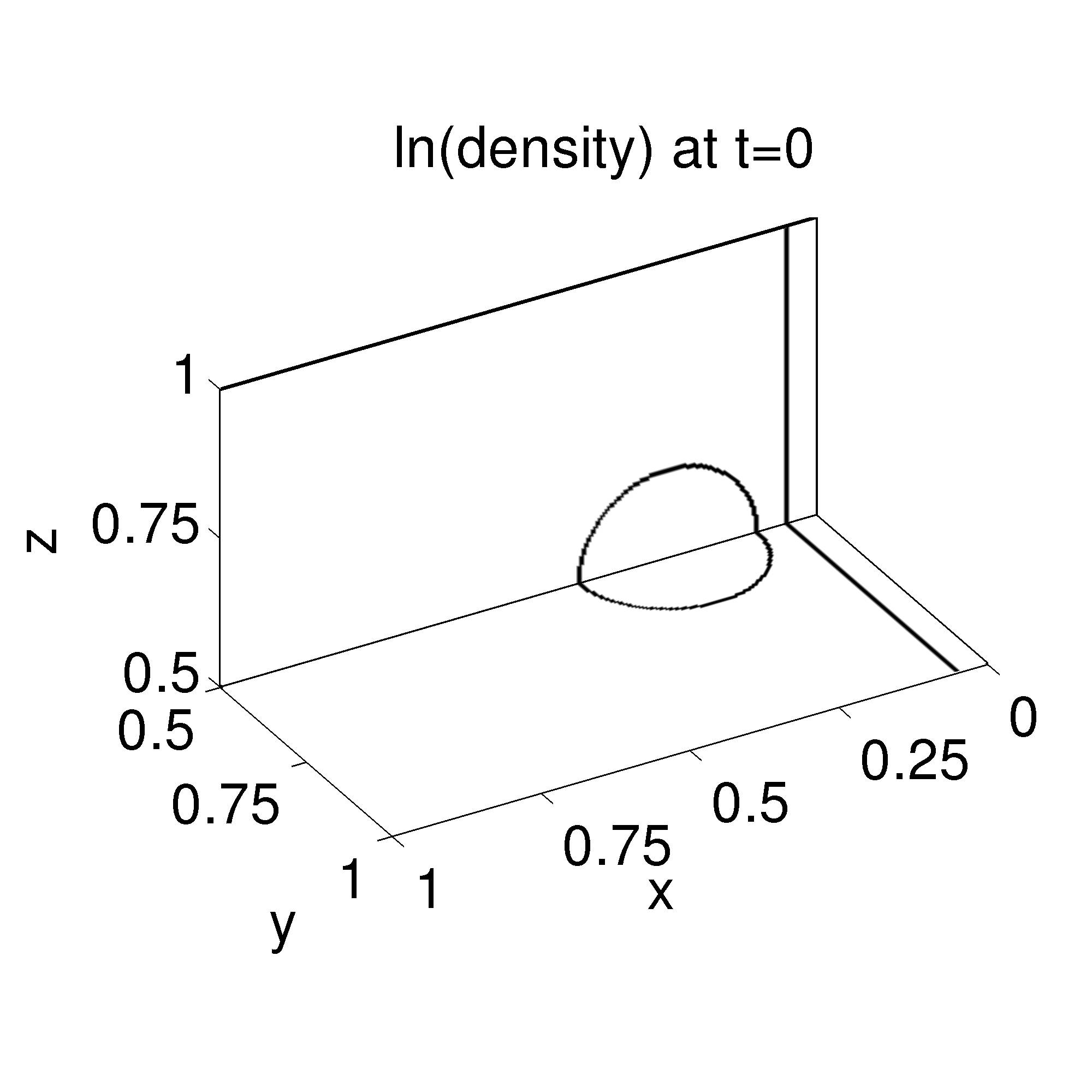} &
    \includegraphics[width=0.30\textwidth]{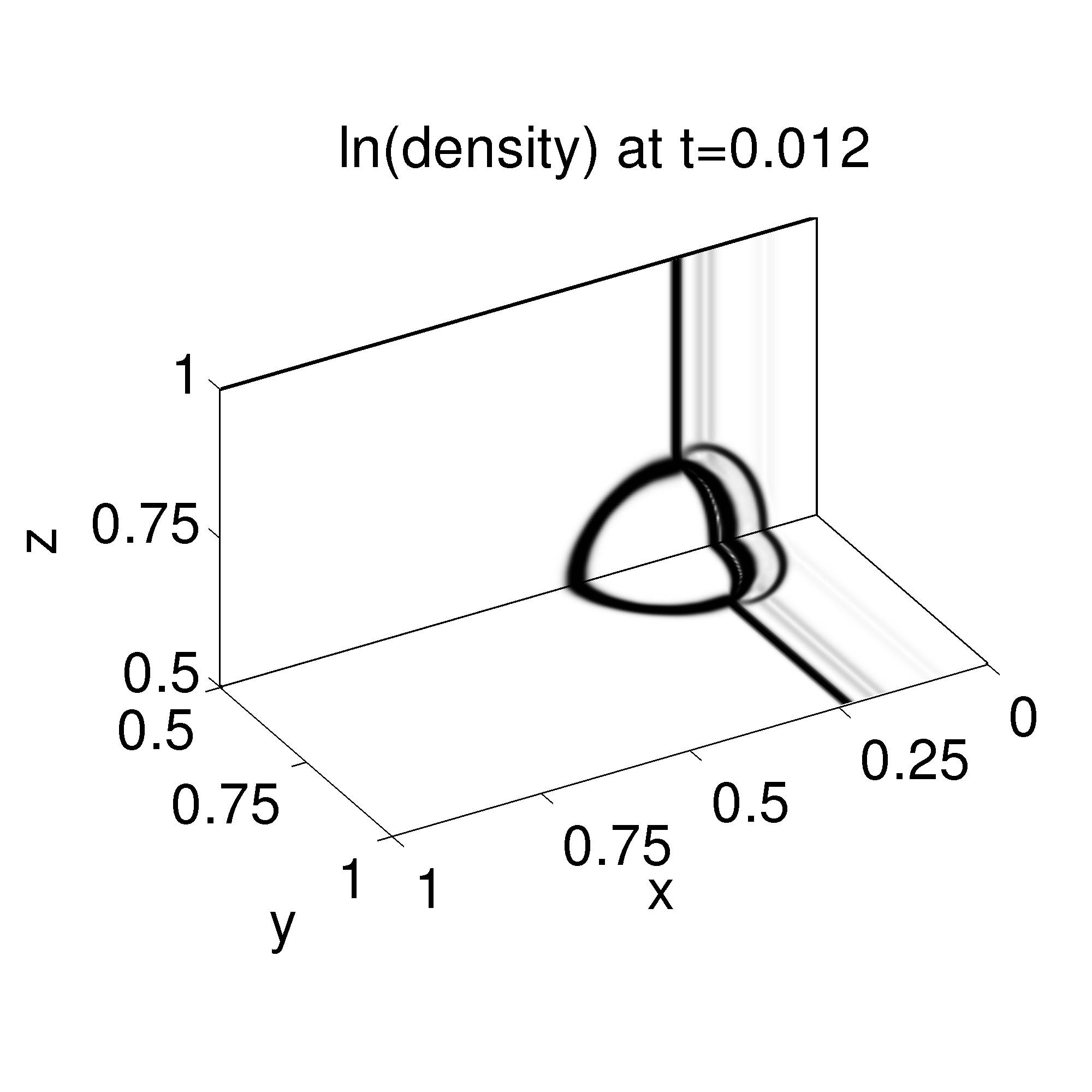} &
    \includegraphics[width=0.30\textwidth]{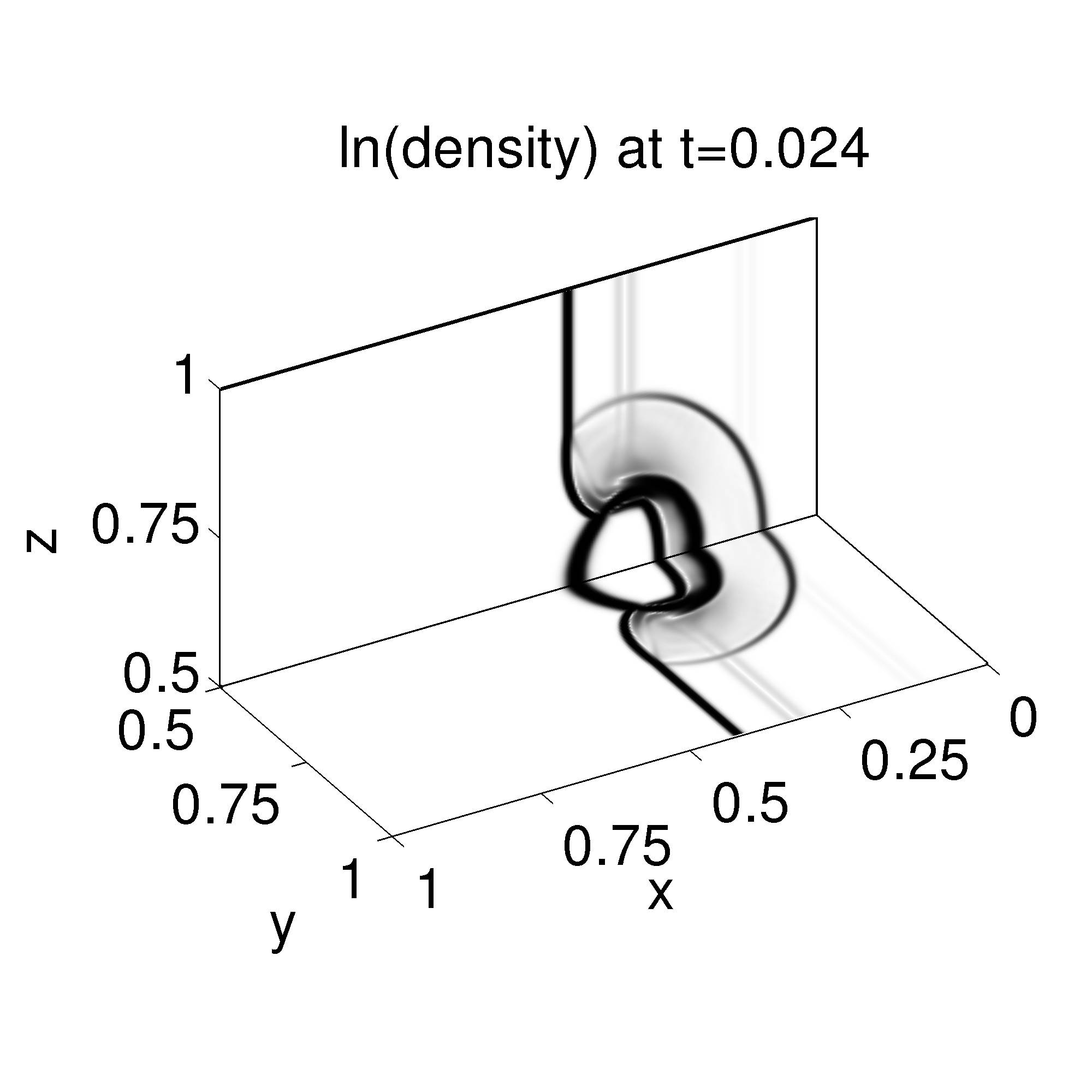} \\
    \includegraphics[width=0.30\textwidth]{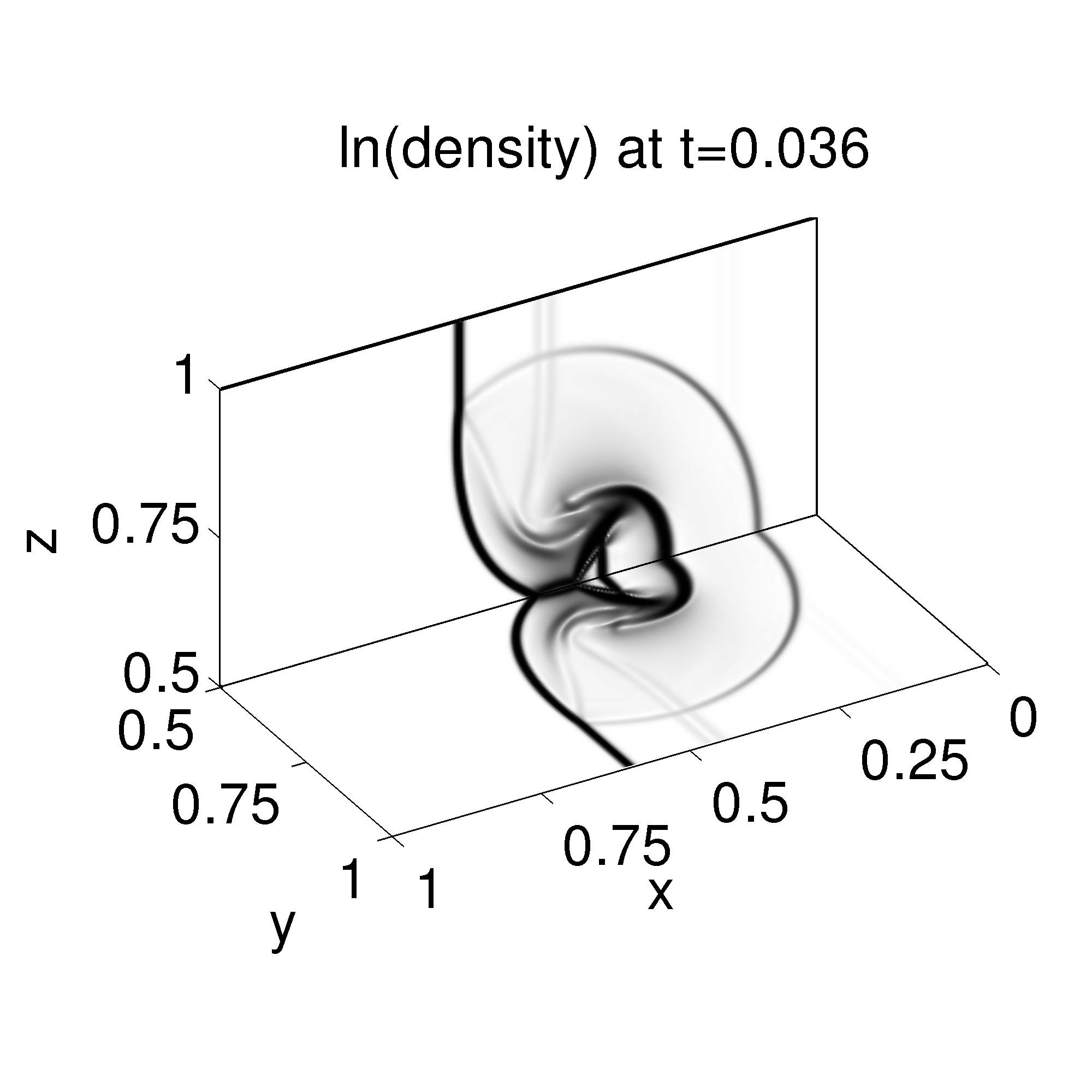} &
    \includegraphics[width=0.30\textwidth]{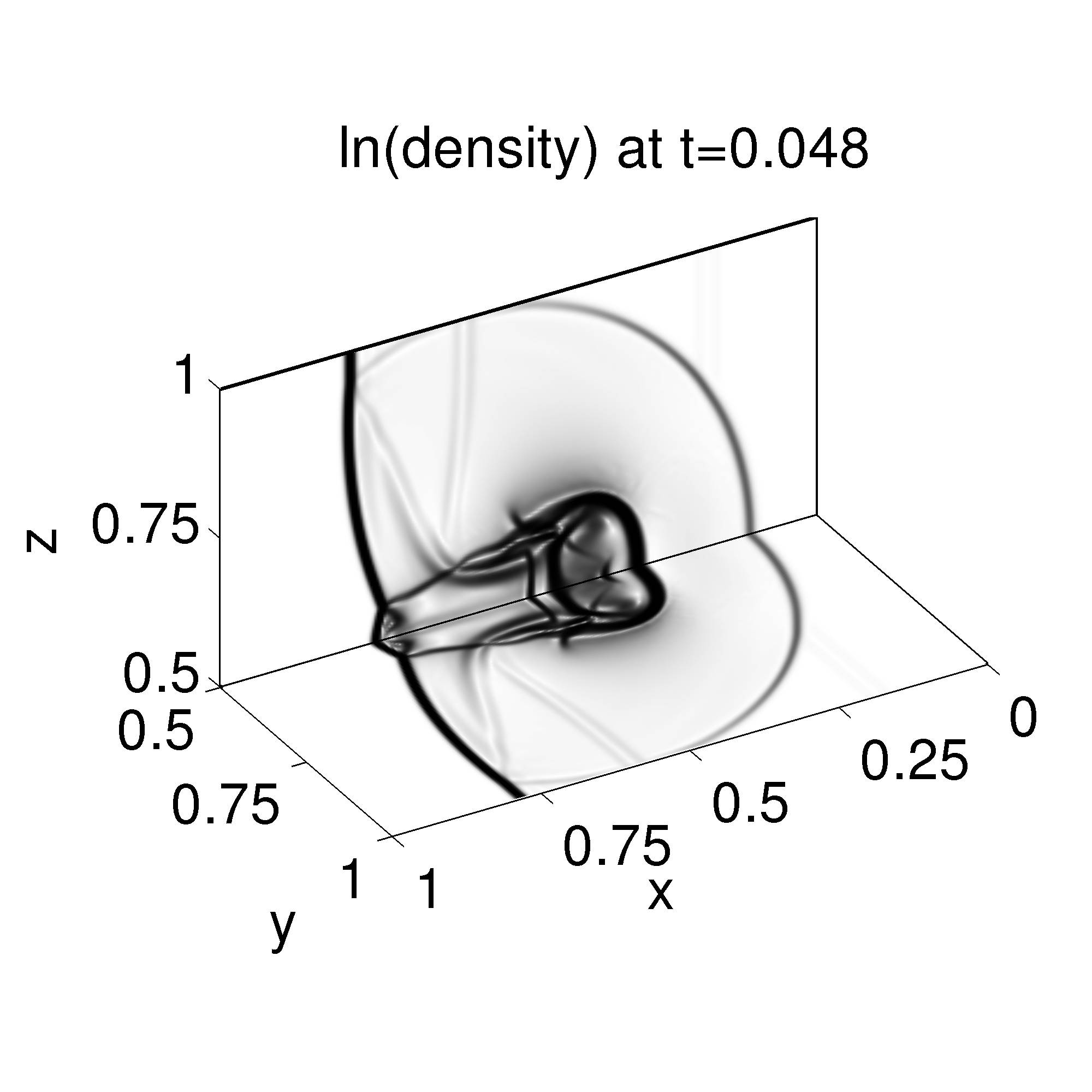} &
    \includegraphics[width=0.30\textwidth]{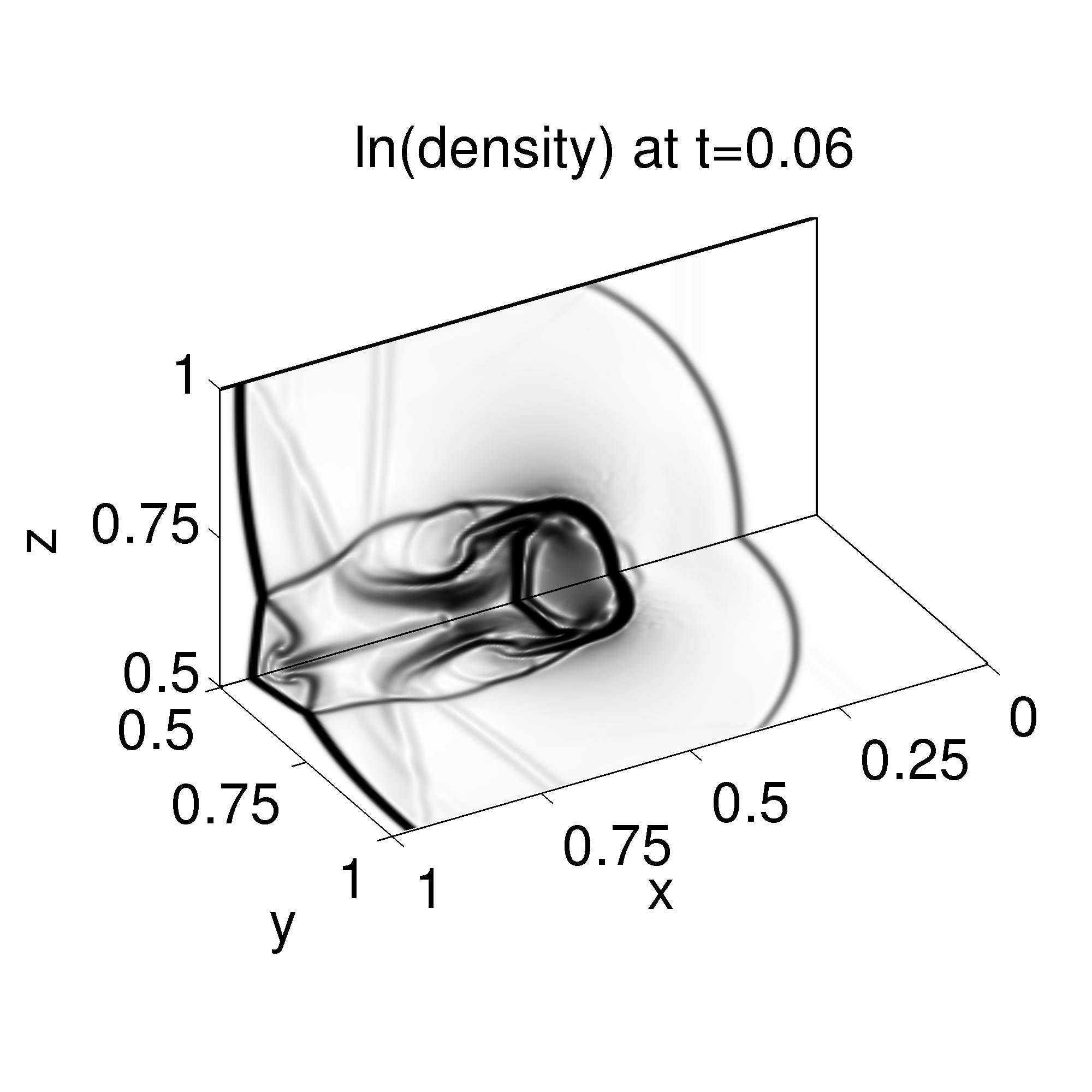} \\
\end{tabular}
    \caption{\newtext{The 3D cloud-shock interaction problem.  Schlieren plots
    of $\ln(\rho)$.  The solution here is computed using a $256\times 256
    \times 256$ mesh.  Cross-sections at $y = 0.5$ and $z = 0.5$ for the
    region $y\geq 0.5$ and $z\geq 0.5$ are shown.  Constrained transport and
    positivity-preserving limiter are turned on. 
    \label{fig:3DCloudShock}}}
\end{center}
\end{figure}

\subsection{Blast wave example}
\label{sec:Blast}

In the blast wave problems, strong shocks interact with a low-$\beta$ background, which can cause negative pressure if not handled properly.  These problems are often used to test the positivity-preserving capabilities of numerical methods for MHD~\cite{Balsara2009,Balsara1999a,Christlieb2015,Gardiner2008,Li2011,Mignone2010,Ziegler2004}.  The initial conditions contain a piecewise defined pressure:
\begin{equation}
p = \begin{cases}
0.1 & r < 0.1, \\
1000 & \text{otherwise},
\end{cases}
\end{equation}
where $r$ is the distance to the origin, and 
a constant density, velocity and magnetic field:
\begin{equation}
	(\rho, u^x, u^y, u^z, B^x, B^y, B^z) = 
            (1, 0, 0, 0, 100/\sqrt{4\pi}/\sqrt{2}, 100/\sqrt{4\pi}/\sqrt{2}, 0).
\end{equation}
The initial magnetic potential is simply
\begin{equation}
    \Avec = (0, 0, 100 y / \sqrt{4\pi} / \sqrt{2} - 100 x / \sqrt{4\pi} / \sqrt{2}).
\end{equation}
In 2D we only keep track of $A^z$, as we do with all the 2D examples.

\subsubsection{\newtext{Blast wave example: The} 2D problem}

In this section, we present our result on the 2D version of the blast wave problem.
The computational domain is $[-0.5, 0.5]\times [-0.5, 0.5]$, with zeroth order
extrapolation on the conserved quantities and first order extrapolation on the
magnetic potential as the boundary conditions on all four sides
\newtext{(i.e., conserved quantities at the ghost points are set equal to the
last interior point, and values for the magnetic potential are defined through repeated
extrapolation of two point stencils).}

\begin{figure}
\begin{center}
    \begin{tabular}{cc}
        \includegraphics[width=0.4\textwidth]{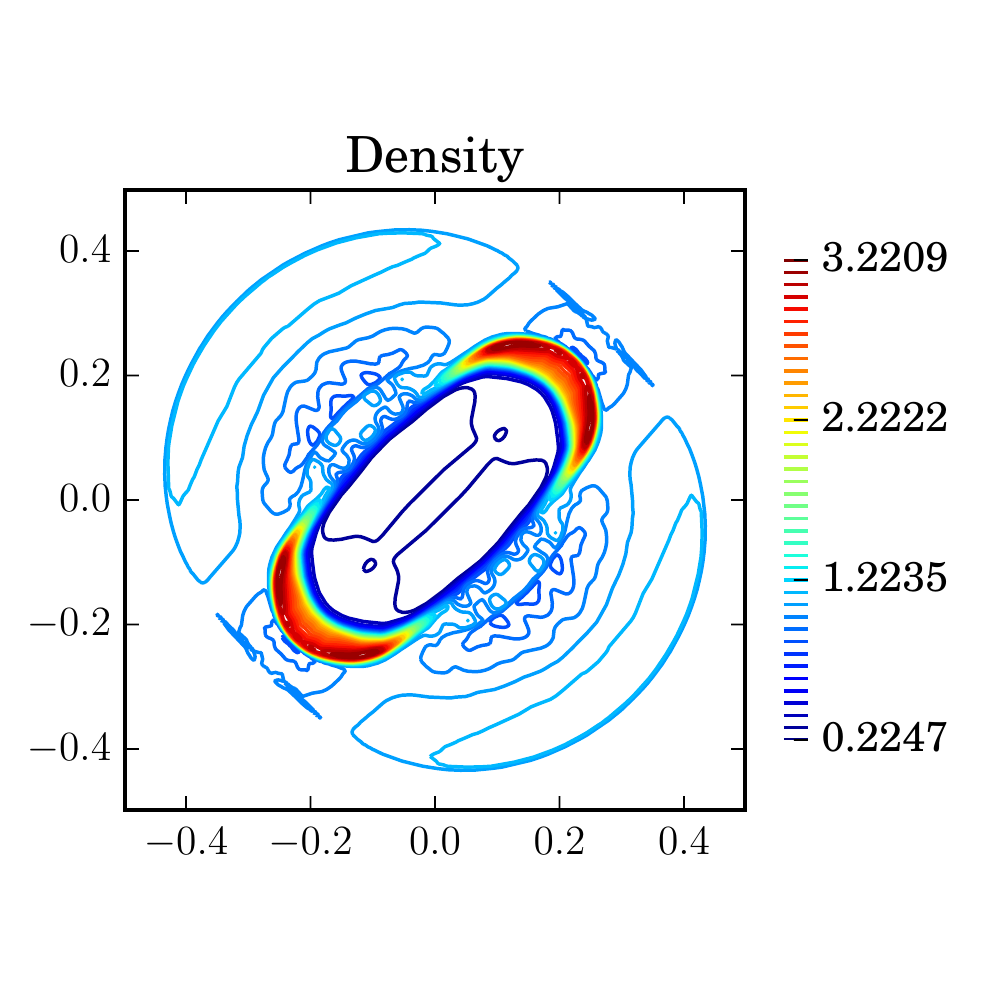} &
        \includegraphics[width=0.4\textwidth]{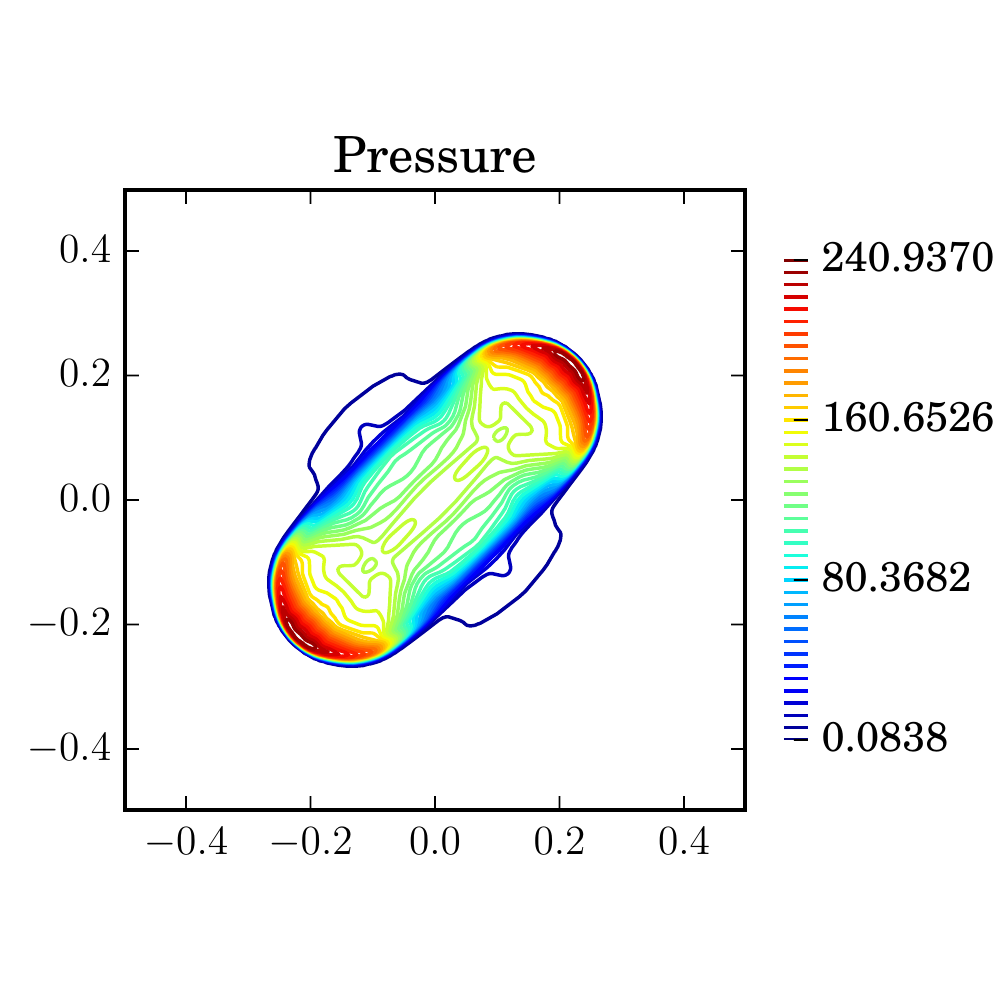} \\
        (a) & (b) \\
        \includegraphics[width=0.4\textwidth]{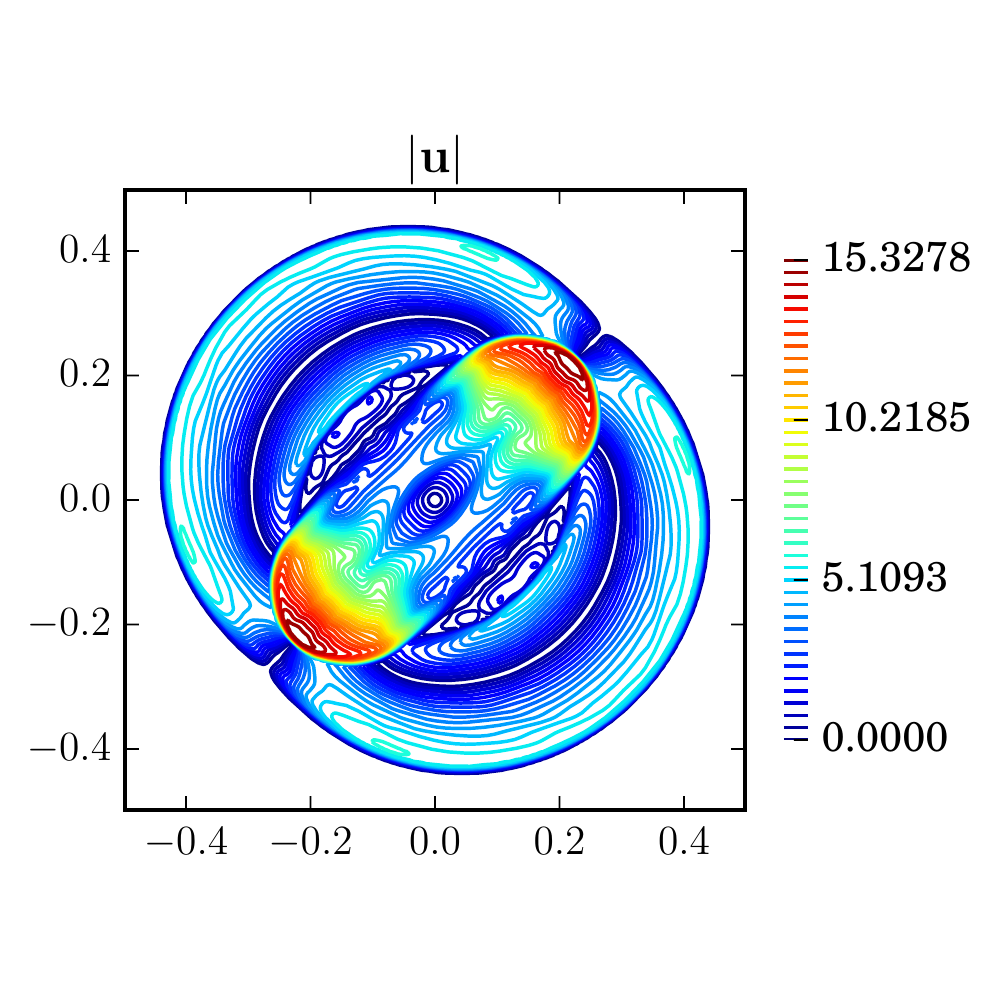} &
        \includegraphics[width=0.4\textwidth]{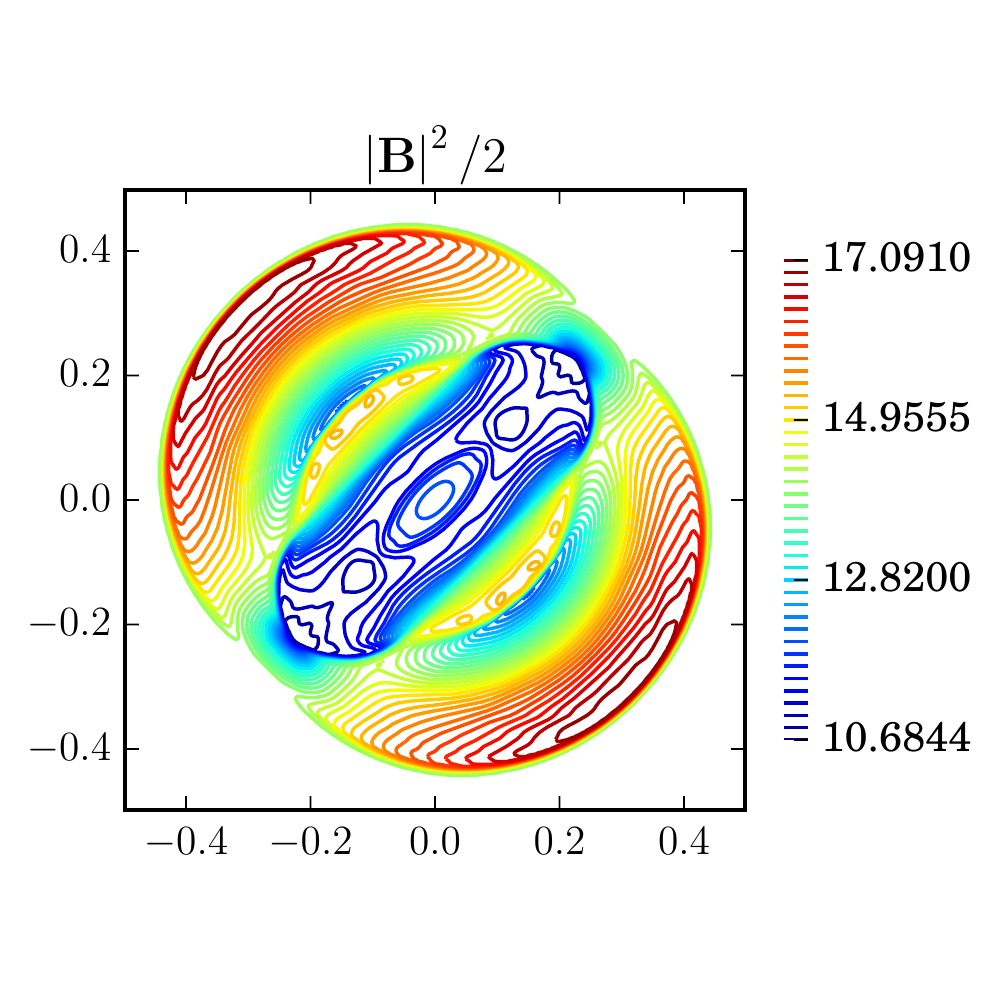} \\
        (c) & (d)
    \end{tabular}
    \caption{2D blast problem.  Shown here are the contour plots at $t = 0.01$
    of (a) density, (b) thermal pressure, (c) magnitude of velocity, and (d)
    magnetic pressure.  \newtext{A total of} $40$ equally spaced contours
    \newtext{ranging from the min to the max of the function} are used for each plot.  The mesh size is $256\times 256$.\label{fig:2DBlast}}
\end{center}
\end{figure}

\newtext{Results for the solution computed to a final time of $t=0.01$ on a
$256\times256$ mesh are presented in Figure~\ref{fig:2DBlast}.
There, we display contour plots of $\rho$, $p$,
$\left|\mathbf{u}\right|$, and $\left|\Bvec\right|$.}  
These plots are
comparable to the previous results in~\cite{Christlieb2015}.  We note that
negative pressure occurs right in the first step if positivity-preserving
limiter is turned off.

\subsubsection{\newtext{Blast wave example: The} 3D problem}

For the 3D version of the blast wave problem, we choose the computational
domain to be $[-0.5, 0.5]\times [-0.5, 0.5] \times [-0.5, 0.5]$ with zeroth
order extrapolation on the conserved quantities and first order extrapolation
on the magnetic potential as the boundary conditions on all six faces 
\newtext{(i.e., conserved quantities at the ghost points are set equal to the
last interior point, and values for the magnetic potential are defined through repeated
extrapolation of two point stencils).}

\begin{figure}
\begin{center}
    \begin{tabular}{cc}
        \includegraphics[width=0.4\textwidth]{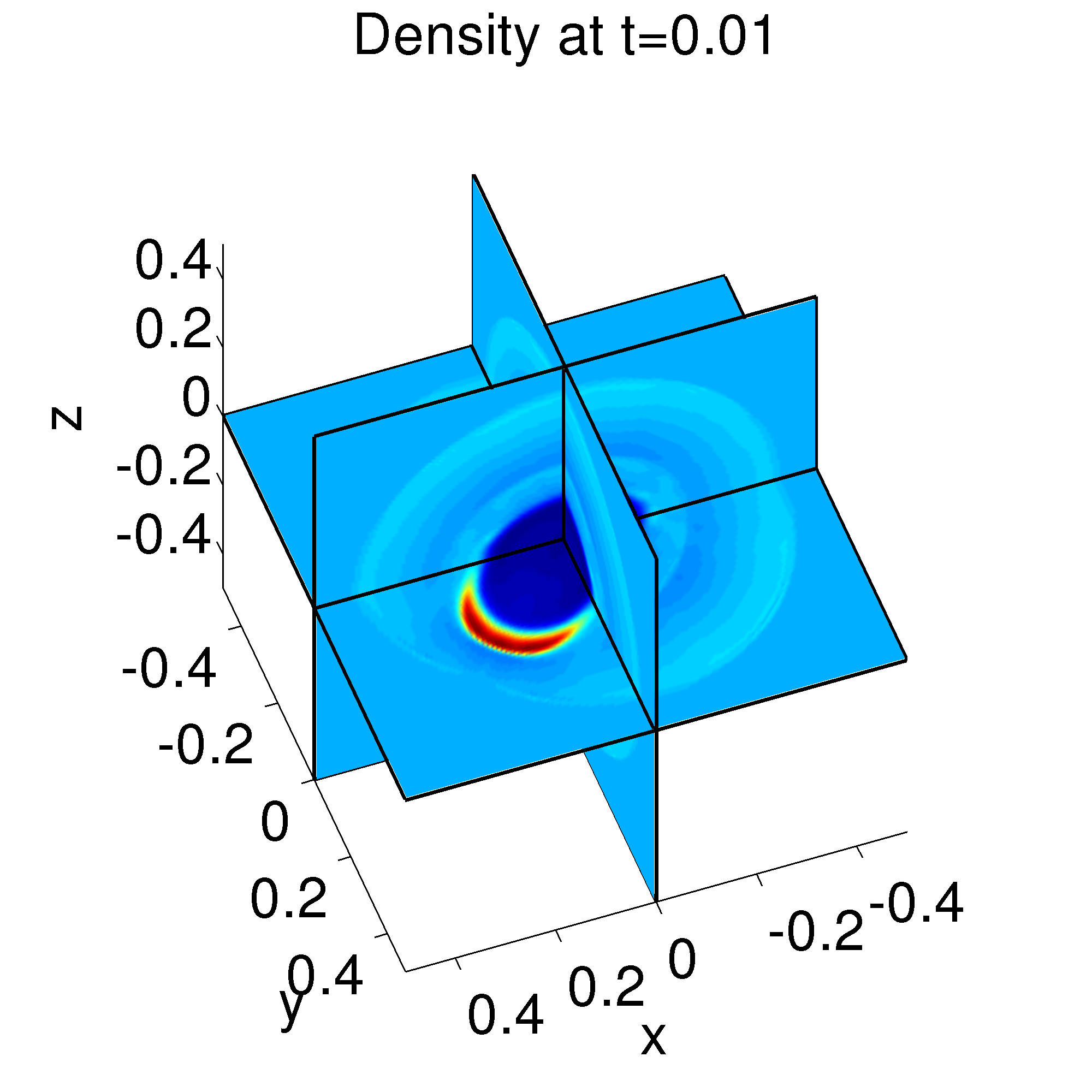} &
        \includegraphics[width=0.4\textwidth]{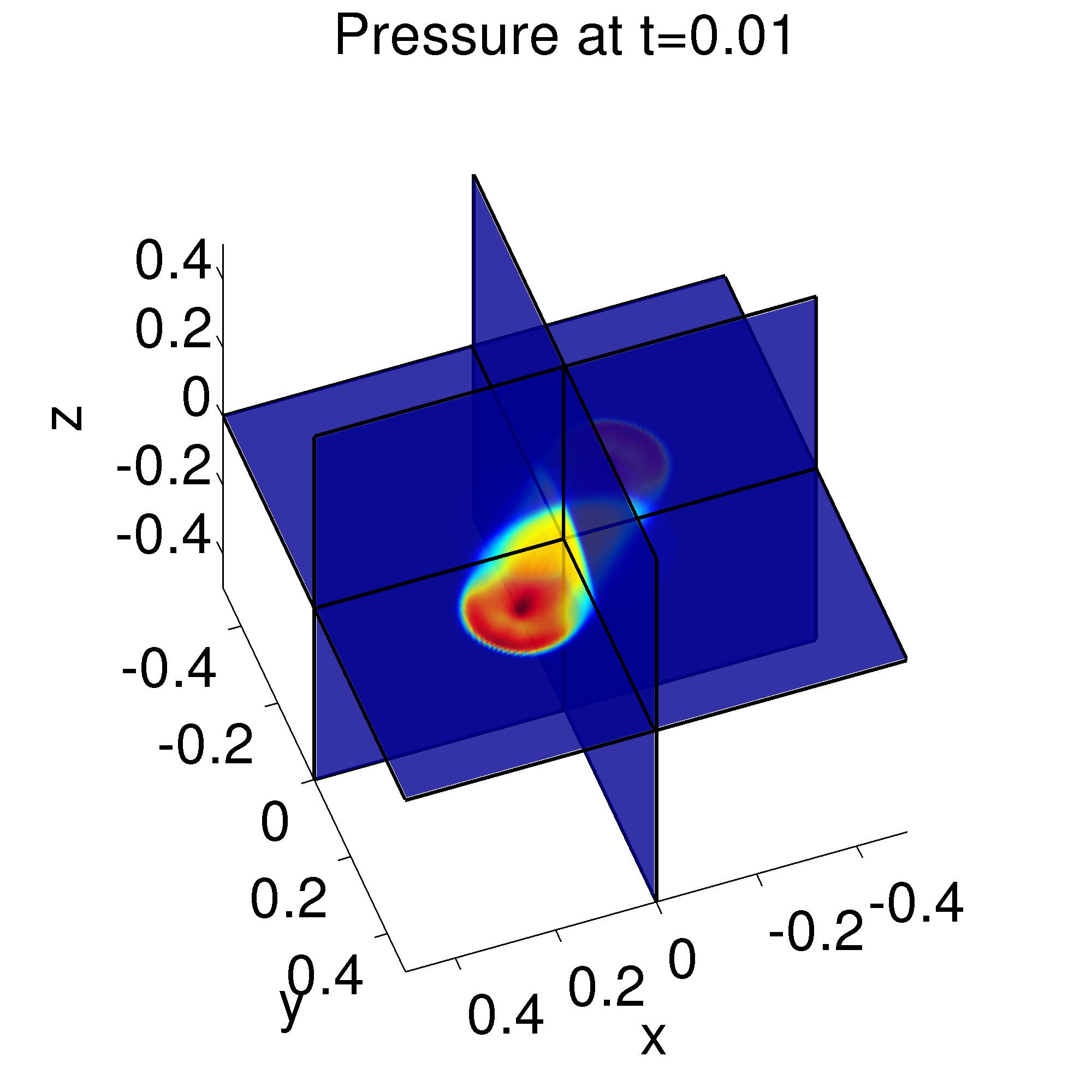} \\
        (a) & (b)
    \end{tabular}
\caption{3D blast problem.  Shown \newtext{here} are the pseudocolor plots at $t = 0.01$ of (a) density and (b) pressure.  The mesh size is $150\times 150 \times 150$. 
\newtext{In
Figure~\ref{fig:3DBlastSlice} we plot a cut of the solution along $z=0$.  The
positivity-preserving limiter is required to simulate this problem.}
\label{fig:3DBlast} }  
\end{center}
\end{figure}

\begin{figure}
\begin{center}
    \begin{tabular}{cc}
        \includegraphics[width=0.4\textwidth]{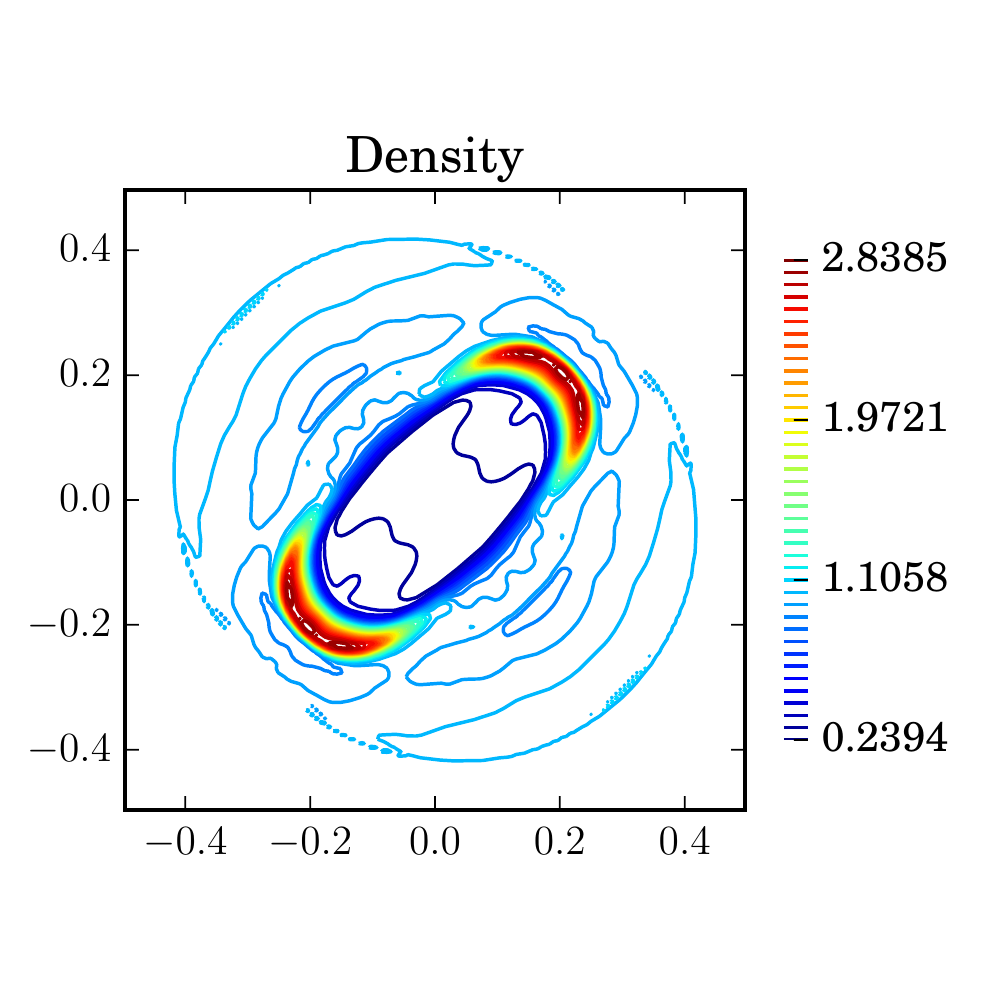} &
        \includegraphics[width=0.4\textwidth]{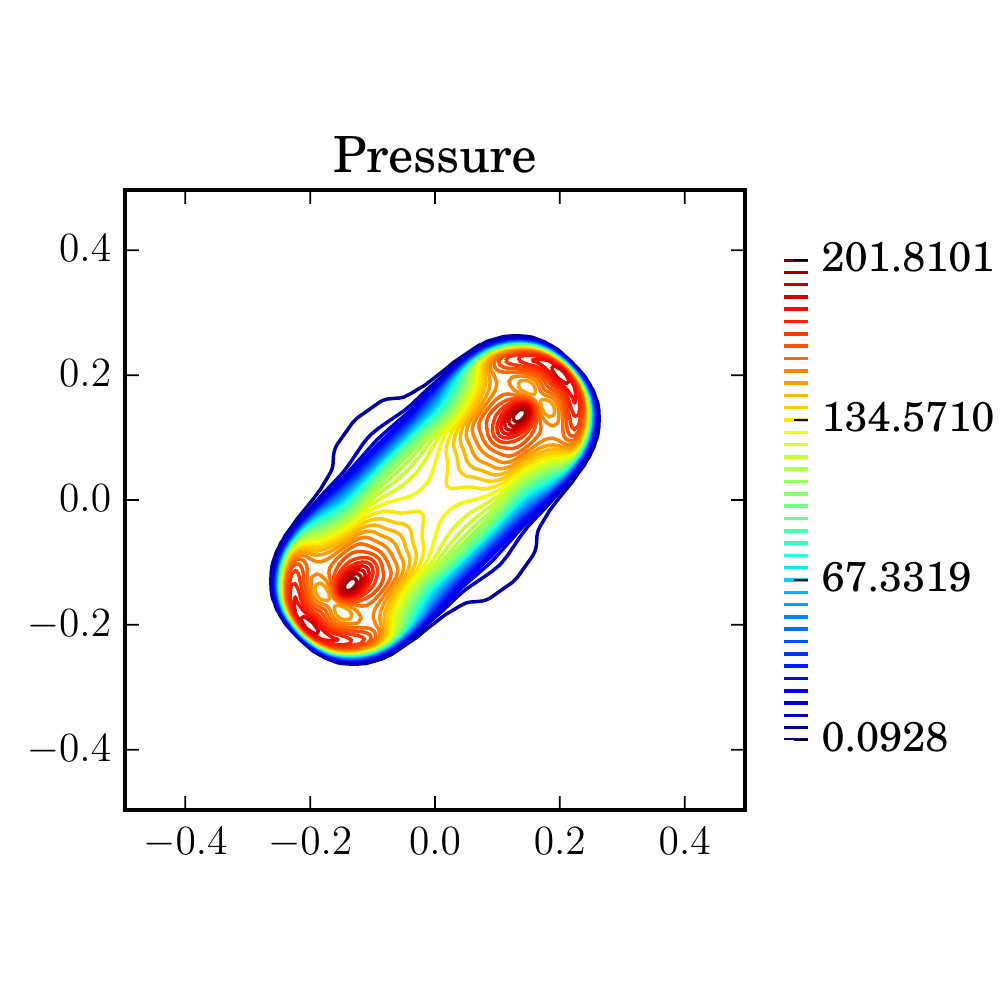} \\
        (a) & (b) \\
        \includegraphics[width=0.4\textwidth]{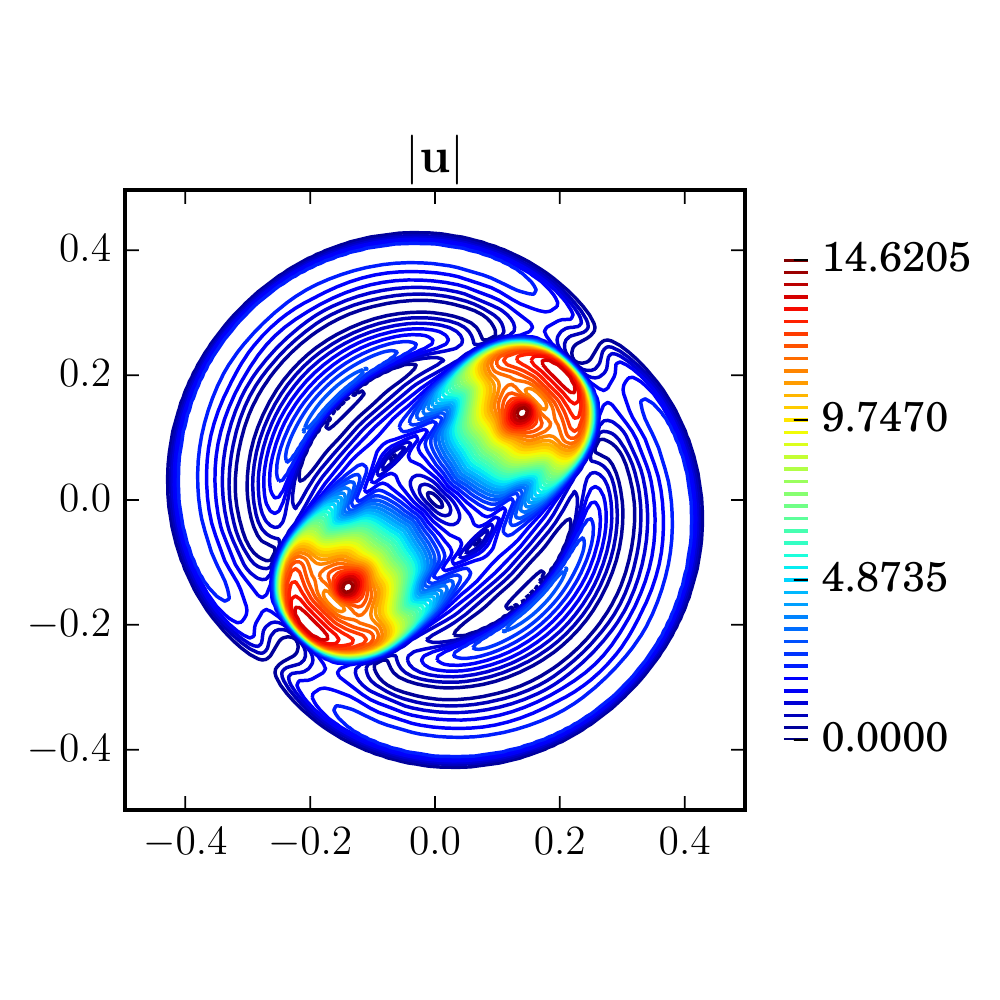} &
        \includegraphics[width=0.4\textwidth]{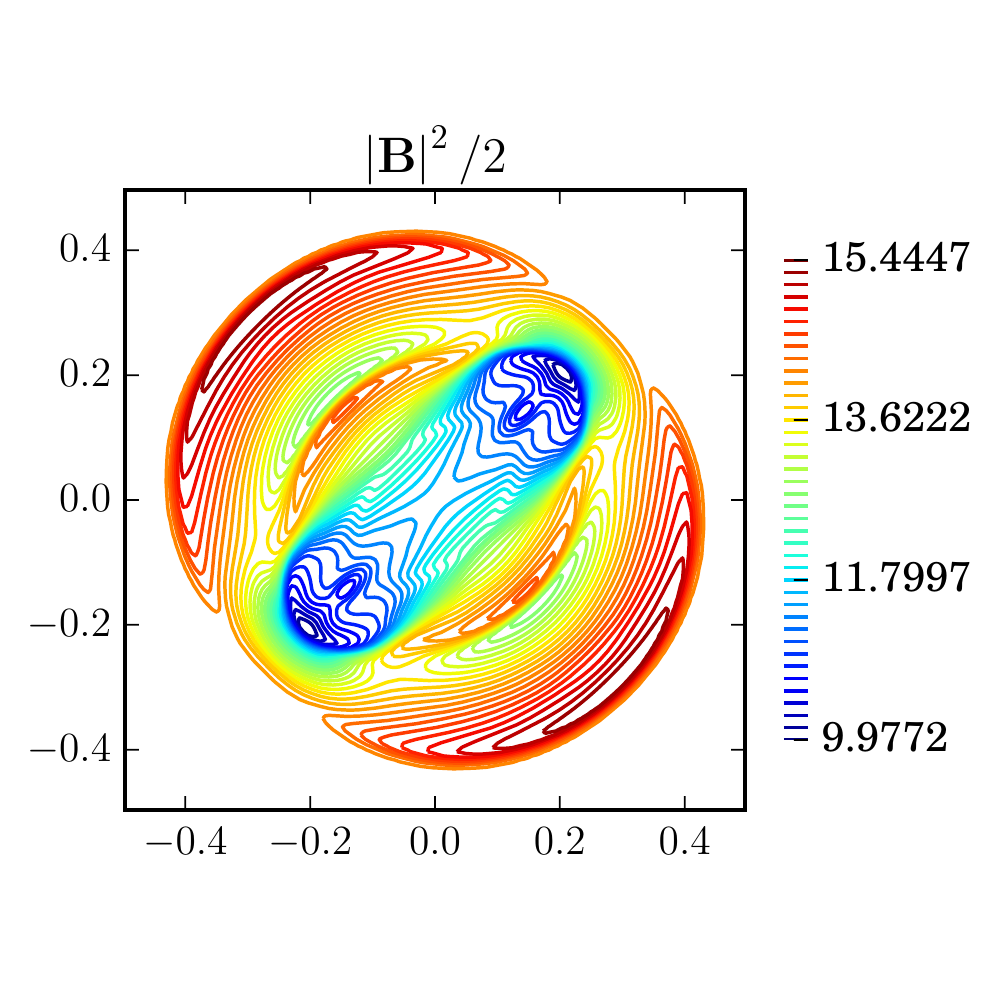} \\
        (c) & (d)
    \end{tabular}
    \caption{3D blast problem.  Shown are the contour plots at time $t = 0.01$
    cut at $z = 0$ of (a) density, (b) thermal pressure, (c) norm of velocity,
    and (d) magnetic pressure.  The solution is obtained using a $150 \times
    150 \times 150$ mesh.  \newtext{A total of} 40 equally spaced contours are used for each plot. 
    \label{fig:3DBlastSlice}}
\end{center}
\end{figure}

The solution at $t = 0.01$ is computed using a $150 \times 150 \times 150$ mesh.
We present in Figure~{\ref{fig:3DBlast}} the plots of the density and pressure, and also in Figure~{\ref{fig:3DBlastSlice}} the contour plots of the slice at $z = 0$ of 
the density, pressure,  velocity, and magnetic pressure.
These results are comparable to those found in~\cite{Christlieb2015,Gardiner2008,Mignone2010,Ziegler2004}.  We note here that negative pressure occurs in the second time step if \newtext{the} positivity-preserving limiter is turned off.

\subsection{\newtext{Errors in energy conservation}} 

\newtext{
When the positivity-preserving limiter is turned on, we make use of an energy
correction step in Eqn.~\eqref{eq:KeepP} in order to keep the pressure the same as before
the magnetic field correction.  This breaks the conservation of the energy,
and therefore we investigate the effect of this step for several test
problems.  Because (global) energy conservation only holds for problems that
have either periodic boundary conditions or constant values near the boundary
throughout the entire simulation, we only choose problems with this property for our test
cases.
}

\newtext{
For the 2D problems, the (relative) energy conservation error at time $t=t^n$ is defined as
    \begin{equation}
       \label{eq:EnergyConservationError}
       \text{Energy conservation error} :=  \frac{\left|\sum_{i,j} \left( \En^{n}_{i,j}- \En^{0}_{i,j} \right) \right|}{\sum_{i,j} \En^{0}_{i,j}},
    \end{equation}
and the energy conservation errors for the 3D problems are defined similarly.
}

\newtext{
Results for the 2D and 3D smooth \Alfven test case are presented in
Figure~\ref{fig:ConservationAlfven}, where we observe negligible errors
produced by the energy correction step.  We attribute this to the fact that
this problem retains a smooth solution for the entirety of the simulation.
Results for problems with shocks and vortices are presented in
Figure~\ref{fig:ConservationPP}, where we find non-zero errors.
For each of these test problems, we present the results from several different
sizes of meshes.  All the problems are
run to the final time found in
Sections~{\ref{sec:Alfven}--\ref{sec:Blast}} save one.
For the 2D Orszag-Tang problem, we run the
simulations to a much later time of $t=30$ in order to quantify the energy
conservation errors for a long time simulation on a non-trivial problem.
}

\newtext{
Finally, in Figure~{\ref{fig:ConservationNoPP}} we also include 
the conservation errors when the positivity-preserving limiter is turned off.
We observe that the solver retains total energy up to machine roundoff
errors, as should be the case.  
}

\newtext{
We note the following patterns in the energy conservation errors:
\begin{itemize}
    \item The errors are below 1\% for all the test problems in the duration of the simulations presented;
    \item When the positivity-preserving limiter is turned on, the errors grow linearly in time and decrease as the mesh is refined.  
\end{itemize}
We therefore conclude that the violation in energy conservation introduced by
the positivity-preserving limiter is insignificant for the problems tested in
this work.
}

\begin{figure}
\begin{center}
    \begin{tabular}{cc}
        \includegraphics[width=0.48\textwidth]{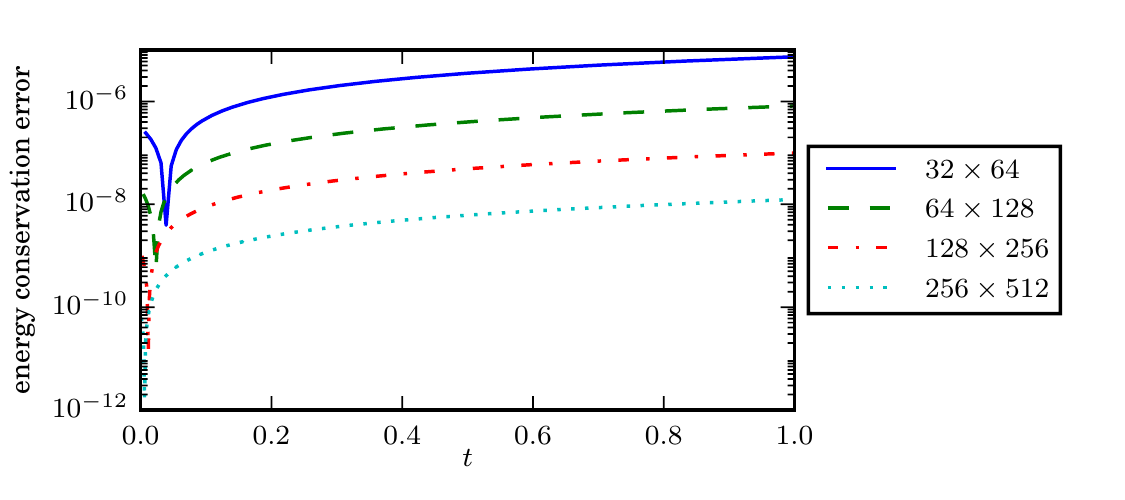} &
        \includegraphics[width=0.48\textwidth]{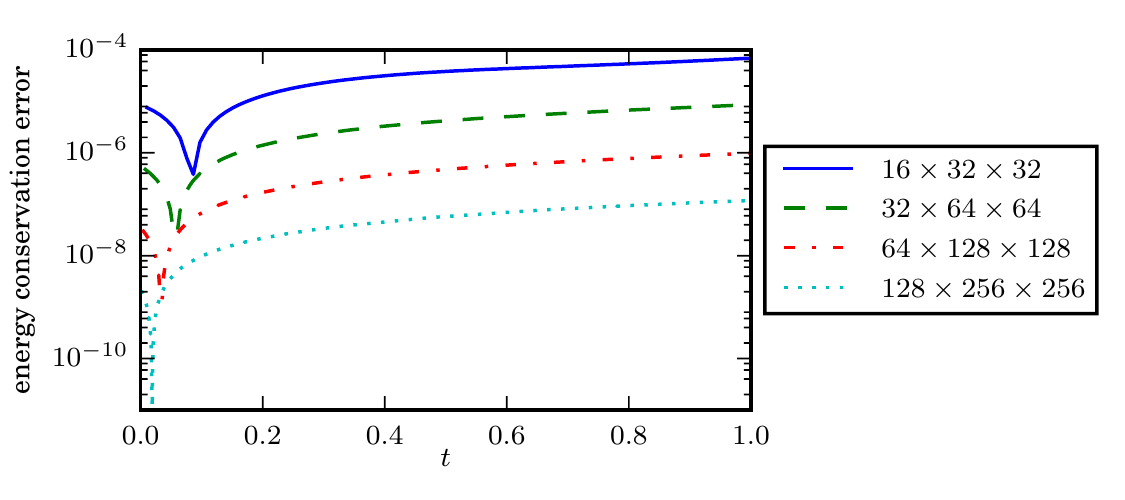} \\
          (a) & (b)
    \end{tabular}
    \caption{\newtext{Energy conservation errors for the smooth \Alfven test cases.
    Shown here are results for  the
    (a) 2D smooth \Alfven, and (b) 3D smooth \Alfven test cases. In order to
    extract the errors, we plot the results on a semi-log scale because
    otherwise the results are indiscernible from the $t$-axis.  For this
    smooth test case, the effect of the 
    the energy correction step (and hence the
    positivity-preserving limiter)
    is negligible because the solution remains
    smooth for the entire simulation.
    }
    \label{fig:ConservationAlfven}}
    \end{center}
\end{figure}

\begin{figure}
\begin{center}
    \begin{tabular}{cc}
        \includegraphics[width=0.48\textwidth]{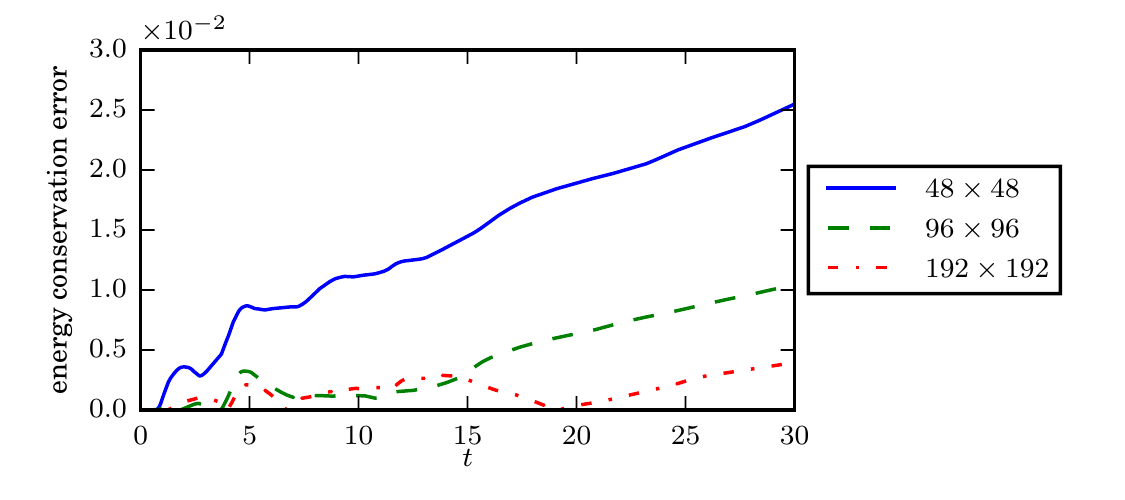} &
        \includegraphics[width=0.48\textwidth]{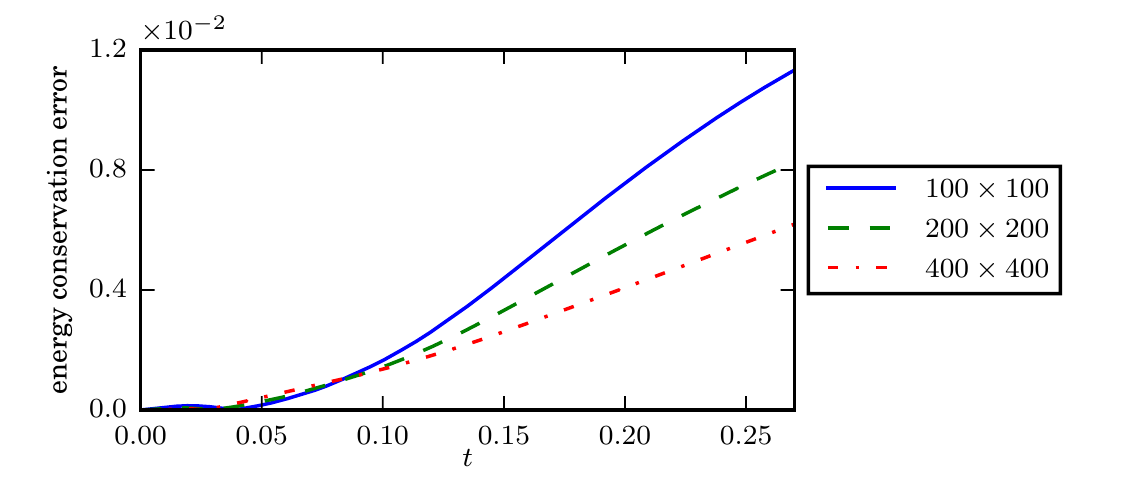} \\
        (a) \hspace{0.08\textwidth} & (b) \hspace{0.05\textwidth} \\
        \includegraphics[width=0.48\textwidth]{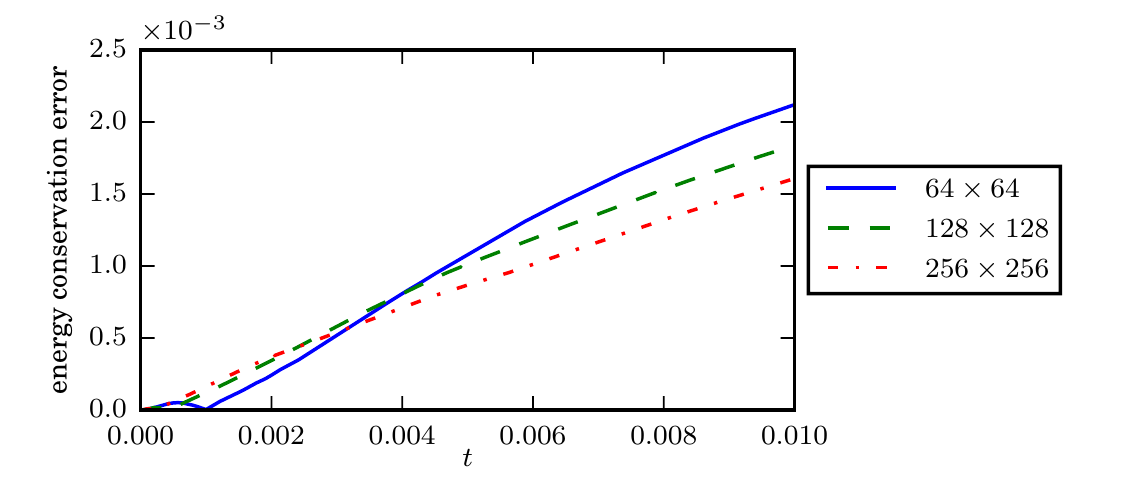} &
        \includegraphics[width=0.48\textwidth]{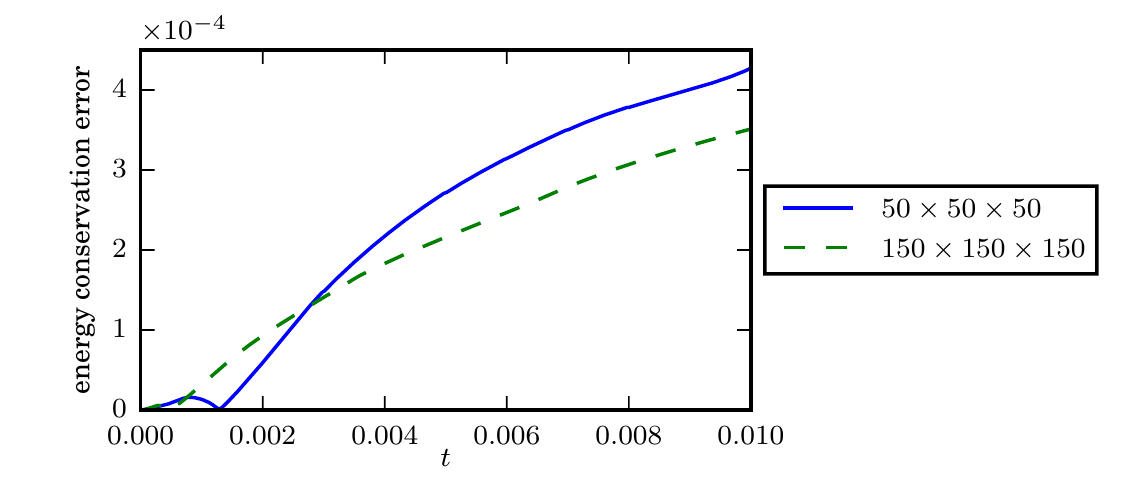} \\
        (c) \hspace{0.08\textwidth} & (d) \hspace{0.05\textwidth} \\
    \end{tabular}
    \caption{\newtext{Energy conservation errors.  Shown here are conservation errors
    when the positivity-preserving limiter (and hence the energy correction
    step) is turned on.
    (a) 2D Orszag-Tang problem, (b) 2D rotor problem, (c) 2D blast problem,
    and (d) 3D blast problem.
    Note that the rotor and blast problems require the application of a
    positivity-preserving limiter in order to run, but this comes at the
    expense of losing energy conservation.
    For the Orszag-Tang test problem, we run to a late final time.
	Again, we observe that the errors in energy conservation decrease as the
    mesh is refined.}
    \label{fig:ConservationPP}}
    \end{center}
\end{figure}

\begin{figure}
\begin{center}
    \begin{tabular}{cc}
        \includegraphics[width=0.48\textwidth]{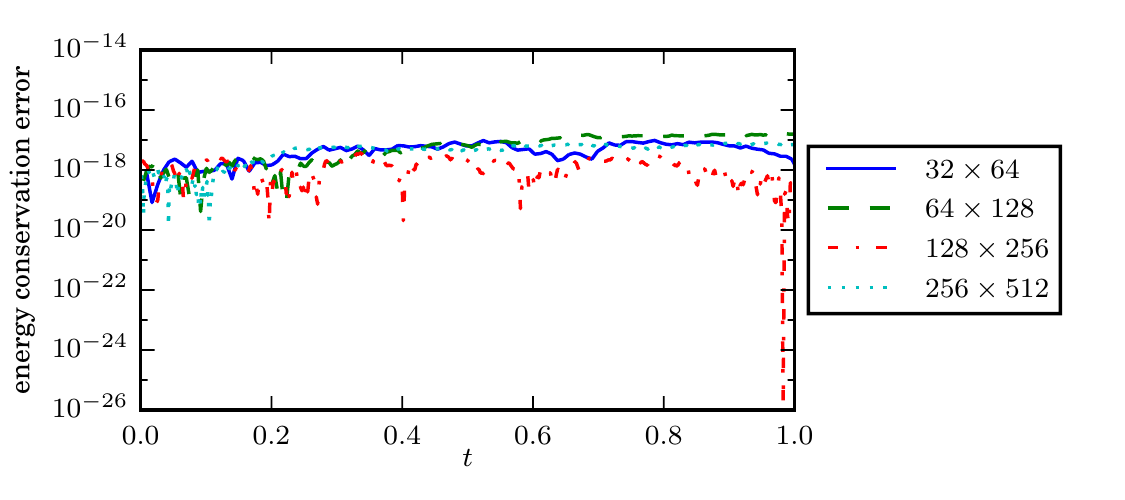} &
        \includegraphics[width=0.48\textwidth]{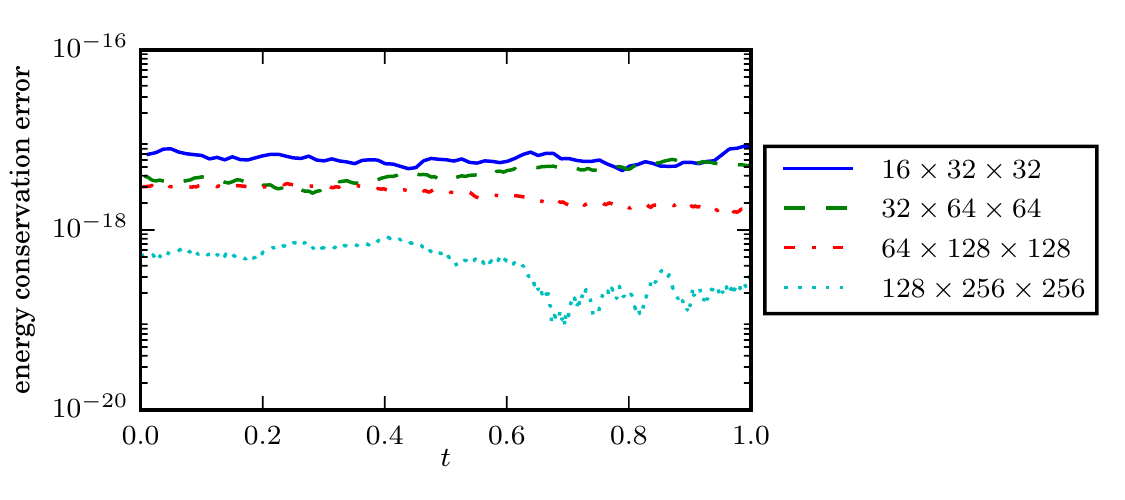} \\
        (a) \hspace{0.08\textwidth} & (b) \hspace{0.05\textwidth} \\
    \end{tabular}
    \caption{\newtext{Energy conservation errors.  Here, we show results for
    conservation errors when the positivity-preserving limiter (and hence the
    energy correction step) is turned off.
    The solver analytically conserves the discrete total
    energy up to machine precision.
    Shown here are results for  the
    (a) 2D smooth \Alfven, and (b) 3D smooth \Alfven test cases.  Note the
    logarithmic scale for the axes, and that these errors are numerically
    conserved up to machine precision.
    }
    \label{fig:ConservationNoPP}}
    \end{center}
\end{figure}

\section{Conclusion and future work}
\label{sec:CFW}

In this paper we propose a high-order single-stage single-step
positivity-preserving method for the ideal magnetohydrodynamic equations.  The
base scheme uses a finite difference WENO method with a Lax-Wendroff time
discretization that is based on the Picard integral formulation of hyperbolic
conservation laws.  A discrete divergence-free condition on the magnetic field
is found by using an unstaggered constrained transport method that evolves a
vector potential alongside the conserved quantities on the same mesh as the
conserved variables.  This vector potential is evolved with a modified version
of a finite difference Lax-Wendroff WENO method that was originally developed
for Hamilton-Jacobi equations.  This allows us to define non-oscillatory
derivatives for the magnetic field.  To further enhance the robustness of our
scheme, a flux limiter is added to preserve the positivity of the density and
pressure.  
\newtext{Unlike our previous solvers that are based on SSP-RK time stepping
\cite{Christlieb2014,Christlieb2015}, this solver does not require the use of
intermediate stages.  This reduces the total storage required for the method,
and may lead to more efficient implementation for an AMR setting.  Moreover,
we need only apply one WENO reconstruction per time step for the fluid
variables, whereas the third and fourth-order solvers in \cite{Christlieb2014}
require three and ten, respectively, WENO reconstructions per time step.
}
Numerical results show that our scheme has the expected high-order
accuracy \newtext{for smooth problems}, and is capable of solving some very
stringent test problems, even when low density and pressure are present.
Future work includes embedding this scheme into an AMR framework, extending
the scheme to curvilinear meshes to accommodate complex geometry, and 
incorporate non-ideal terms in the MHD equations by means of a semi-implicit solver.

\section*{Acknowledgements}

\newtext{We would like to thank the anonymous reviewers for their thoughtful comments and suggestions to improve this work.  
This work was supported by: AFOSR grants FA9550-15-1-0282, and FA9550-12-1-0343; the Michigan State University Foundation; NSF grant DMS-1418804;  and Oak Ridge National Laboratory (ORAU HPC LDRD).}

\bibliographystyle{unsrt}       
\bibliography{library}

\end{document}